%% file: lagrangian.tex
\documentclass[a4paper,12pt]{article}

\pdfoutput=1  % Uncomment before ArXiV submission. See https://tex.stackexchange.com/questions/182523/publishing-to-arxiv-with-hyperlinks
\input{preamble}

\input{macros}

\title{On the Lagrangian capacity of convex or concave toric domains}
\author{Miguel Pereira \\ \small{\href{mailto:miguel.b.per@gmail.com}{miguel.b.per@gmail.com}}}
\date{}

\hypersetup{
    pdftitle={\thetitle},
    pdfauthor={\theauthor},
    pdflang={en-GB}
}

\begin{document}

\maketitle

\begin{abstract}
    We establish computational results concerning the Lagrangian capacity, originally defined by Cieliebak--Mohnke. More precisely, we show that the Lagrangian capacity of a 4-dimensional convex toric domain is equal to its diagonal. Working under the assumption that there is a suitable virtual perturbation scheme which defines the curve counts of linearized contact homology, we extend the previous result to any convex or concave toric domain. This result gives a positive answer to a conjecture of Cieliebak--Mohnke for the Lagrangian capacity of the ellipsoid.
    % We establish computational results concerning the Lagrangian capacity from \cite{cieliebakPuncturedHolomorphicCurves2018}. More precisely, we show that the Lagrangian capacity of a 4-dimensional convex toric domain is equal to its diagonal. The proof involves comparisons between the Lagrangian capacity, the McDuff--Siegel capacities from \cite{mcduffSymplecticCapacitiesUnperturbed2022}, and the Gutt--Hutchings capacities from \cite{guttSymplecticCapacitiesPositive2018}. Working under the assumption that there is a suitable virtual perturbation scheme which defines the curve counts of linearized contact homology, we extend the previous result to any convex or concave toric domain. For this, we use the higher symplectic capacities from \cite{siegelHigherSymplecticCapacities2020}. The key step is showing that moduli spaces of asymptotically cylindrical holomorphic curves in ellipsoids are transversely cut out.
\end{abstract}

\tableofcontents

\section{Introduction}

\subsection{Motivation}

A \textbf{symplectic capacity} is a function $c$ that assigns to every symplectic manifold $(X,\omega)$ (in a restricted subclass) a number $c(X,\omega) \in [0,+\infty]$, satisfying
\begin{description}
    \item[(Monotonicity)] If there exists a symplectic embedding (possibly in a restricted subset of all symplectic embeddings) $(X, \omega_X) \longrightarrow (Y, \omega_Y)$, then $c(X, \omega_X) \leq c(Y, \omega_Y)$;
    \item[(Conformality)] If $\alpha > 0$ then $c(X, \alpha \omega_X) = \alpha  c(X, \omega_X)$.
\end{description}
By the monotonicity property, symplectic capacities can provide obstructions to the existence of symplectic embeddings.

An example of a symplectic capacity is the \textbf{Lagrangian capacity} $c_L$, first defined in \cite[Section 1.2]{cieliebakPuncturedHolomorphicCurves2018}. It is defined as follows. If $(X, \omega)$ is a $2n$-dimensional symplectic manifold\footnote{Unless otherwise stated, every symplectic manifold we will consider will be $2n$-dimensional.} and $L \subset X$ is a Lagrangian submanifold, then the \textbf{minimal symplectic area} of $L$ is given by
\begin{IEEEeqnarray*}{c}
    A_{\mathrm{min}}(L) \coloneqq \inf \{ \omega(\sigma) \mid \sigma \in \pi_2(X,L), \, \omega(\sigma) > 0 \}.
\end{IEEEeqnarray*}
Then, the {Lagrangian capacity} of $(X, \omega)$ is given by
\begin{IEEEeqnarray*}{c}
    c_L(X,\omega) \coloneqq \sup \{ A_{\mathrm{min}}(L) \mid L \subset X \text{ is an embedded Lagrangian torus}\}.
\end{IEEEeqnarray*}

The main goal of this paper is to compute the Lagrangian capacity of (some) toric domains. A \textbf{toric domain} is a Liouville domain of the form $X_{\Omega} \coloneqq \mu^{-1}(\Omega) \subset \C^n$, where $\Omega \subset \R^n_{\geq 0}$ and $\mu(z_1,\ldots,z_n) = \pi(|z_1|^2,\ldots,|z_n|^2)$.

Some examples of toric domains which are going to be relevant in this introduction are the \textbf{ball} $B(a)$, the \textbf{cylinder} $Z(a)$, the \textbf{ellipsoid} $E(a_1,\ldots,a_n)$ and the \textbf{nondisjoint union of cylinders} $N(a)$, which are given by
\begin{IEEEeqnarray*}{rClCrCl}
    B(a)              & \coloneqq & \mu^{-1}(\Omega_{B(a)}),              & \quad & \Omega_{B(a)}              & \coloneqq & \{ x \in \R^n_{\geq 0} \mid x_1 + \cdots + x_n \leq a \}, \\
    Z(a)              & \coloneqq & \mu^{-1}(\Omega_{Z(a)}),              & \quad & \Omega_{Z(a)}              & \coloneqq & \{ x \in \R^n_{\geq 0} \mid x_1 \leq a \}, \\
    E(a_1,\ldots,a_n) & \coloneqq & \mu^{-1}(\Omega_{E(a_1,\ldots,a_n)}), & \quad & \Omega_{E(a_1,\ldots,a_n)} & \coloneqq & \{ x \in \R^n_{\geq 0} \mid x_1 / a_1 + \cdots + x_n / a_n \leq 1 \}, \\
    % P(a)              & \coloneqq & \mu^{-1}(\Omega_{P(a)}),              & \quad & \Omega_{P(a)}              & \coloneqq & \{ x \in \R^n_{\geq 0} \mid \forall i = 1, \ldots, n \colon x_i \leq a \}, \\
    N(a)              & \coloneqq & \mu^{-1}(\Omega_{N(a)}),              & \quad & \Omega_{N(a)}              & \coloneqq & \{ x \in \R^n_{\geq 0} \mid \exists i = 1, \ldots, n \colon x_i \leq a \}.
\end{IEEEeqnarray*}

The \textbf{diagonal} of a toric domain $X_{\Omega}$ is
\begin{IEEEeqnarray*}{c}
    \delta_\Omega \coloneqq \max \{ a \mid (a,\ldots,a) \in \Omega \}.
\end{IEEEeqnarray*}
It is easy to show (see \cref{lem:c square leq c lag,lem:c square geq delta}) that $c_L(X_\Omega) \geq \delta_\Omega$ for any convex or concave toric domain $X_{\Omega}$. Also, Cieliebak--Mohnke give the following results for the Lagrangian capacity of the ball and the cylinder.

\begin{copiedtheorem}[{\cite[Corollary 1.3]{cieliebakPuncturedHolomorphicCurves2018}}]{prp:cl of ball}
    The Lagrangian capacity of the ball is
    \begin{IEEEeqnarray*}{c+x*}
        c_L(B(1)) = \frac{1}{n}.
        %  \footnote{In this introduction, we will be showcasing many results from the main text. The theorems appear here as they do on the main text, in particular with the same numbering. The numbers of the theorems in the introduction have hyperlinks to their corresponding location in the main text.}
    \end{IEEEeqnarray*}
\end{copiedtheorem}

\begin{copiedtheorem}[{\cite[p.~215-216]{cieliebakPuncturedHolomorphicCurves2018}}]{prp:cl of cylinder}
    The Lagrangian capacity of the cylinder is
    \begin{IEEEeqnarray*}{c+x*}
        c_L(Z(1)) = 1.
    \end{IEEEeqnarray*}
\end{copiedtheorem}

In other words, if $X_{\Omega}$ is the ball or the cylinder then $c_L(X_{\Omega}) = \delta_\Omega$. This motivates the following conjecture by Cieliebak--Mohnke.

\begin{copiedtheorem}[{\cite[Conjecture 1.5]{cieliebakPuncturedHolomorphicCurves2018}}]{conj:cl of ellipsoid}
    The Lagrangian capacity of the ellipsoid is%
    \begin{equation*}
        c_L(E(a_1,\ldots,a_n)) = \p{}{2}{\frac{1}{a_1} + \cdots + \frac{1}{a_n}}^{-1}.
    \end{equation*}
\end{copiedtheorem}

A more general form of the previous conjecture is the following.

\begin{copiedtheorem}{conj:the conjecture}
    If $X_{\Omega}$ is a convex or concave toric domain then
    \begin{IEEEeqnarray*}{c+x*}
        c_L(X_{\Omega}) = \delta_\Omega.
    \end{IEEEeqnarray*}
\end{copiedtheorem}

So, more precisely, the goal of this paper is to prove \cref{conj:the conjecture}. We will offer two main results in this direction.
\begin{enumerate}
    \item In \cref{lem:computation of cl}, we prove that $c_L(X_\Omega) = \delta_\Omega$ whenever $X_{\Omega}$ is convex and $4$-dimensional.
    \item In \cref{thm:my main theorem}, using techniques from contact homology we prove that $c_L(X_\Omega) = \delta_\Omega$ for any convex or concave toric domain $X_{\Omega}$. More specifically, in this case we are working under the assumption that there is a virtual perturbation scheme such that the linearized contact homology of a nondegenerate Liouville domain can be defined (see \cref{sec:assumptions of virtual perturbation scheme}).
\end{enumerate}

\begin{remark}
    In \cite{guttCubeNormalizedSymplectic}, Jean Gutt, the Author and Vinicius Ramos explain that the proof of \cref{thm:my main theorem} that we will give carries over (with minor changes) to the case where the toric domain $X_{\Omega}$ is not necessarily convex or concave, thus giving a formula for the Lagrangian capacity of a more general class of toric domains. More precisely, it is shown that $c_L(X_{\Omega}) = \eta_{\Omega}$ for any toric domain $X_{\Omega}$ with $(\eta_{\Omega},\ldots,\eta_{\Omega}) \in \partial \Omega$, where $\eta_{\Omega} \coloneqq \inf \{a \mid X_{\Omega} \subset N(a) \}$. Notice that if $X_{\Omega}$ is convex or concave, then $\eta_{\Omega} = \delta_{\Omega}$.
\end{remark}

\subsection{Main results}

Notice that by the previous discussion, we only need to prove the hard inequality $c_L(X_{\Omega}) \leq \delta_\Omega$. For this, we will need to use other symplectic capacities. The following is a list of the symplectic capacities we will use in this paper, in addition to the Lagrangian capacity:
\begin{enumerate}
    \item the \textbf{Gutt--Hutchings capacities} from \cite{guttSymplecticCapacitiesPositive2018}, denoted by $\cgh{k}$ (\cref{def:gutt hutchings capacities});
    \item the \textbf{$S^1$-equivariant symplectic homology capacities} from \cite{irieSymplecticHomologyFiberwise2021}, denoted by $\csh{k}$ (see \cref{def:s1esh capacities});
    \item the \textbf{McDuff--Siegel capacities} from \cite{mcduffSymplecticCapacitiesUnperturbed2022}, denoted by $\tilde{\mathfrak{g}}_k$ (see \cref{def:g tilde});
    \item the \textbf{higher symplectic capacities} from \cite{siegelHigherSymplecticCapacities2020}, denoted by $\mathfrak{g}_k$ (see \cref{def:capacities glk}).
\end{enumerate}

We now describe our results concerning the capacities mentioned so far. The key step in proving $c_L(X_{\Omega}) \leq \delta_\Omega$ is the following inequality between $c_L$ and $\tilde{\mathfrak{g}}_k$.

\begin{copiedtheorem}{thm:lagrangian vs g tilde}
    If $(X, \lambda)$ is a Liouville domain then
    \begin{IEEEeqnarray*}{c+x*}
        c_L(X) \leq \inf_k^{} \frac{\tilde{\mathfrak{g}}_k(X)}{k}.
    \end{IEEEeqnarray*}
\end{copiedtheorem}

Indeed, this result can be combined with the following results from \cite{mcduffSymplecticCapacitiesUnperturbed2022} and \cite{guttSymplecticCapacitiesPositive2018}.

\begin{copiedtheorem}[{\cite[Proposition 5.6.1]{mcduffSymplecticCapacitiesUnperturbed2022}}]{prp:g tilde and cgh}
    If $X_{\Omega}$ is a $4$-dimensional convex toric domain then
    \begin{IEEEeqnarray*}{c+x*}
        \tilde{\mathfrak{g}}_k(X_\Omega) = \cgh{k}(X_\Omega).
    \end{IEEEeqnarray*}
\end{copiedtheorem}

\begin{copiedtheorem}[{\cite[Lemma 1.19]{guttSymplecticCapacitiesPositive2018}}]{lem:cgh of nondisjoint union of cylinders}
    $\cgh{k}(N(a)) = a (k + n - 1)$.
\end{copiedtheorem}

Combining the three previous results, we get the following particular case of \cref{conj:the conjecture}. Since the proof is short, we present it here as well.

\begin{copiedtheorem}{lem:computation of cl}
    If $X_{\Omega}$ is a $4$-dimensional convex toric domain then
    \begin{IEEEeqnarray*}{c+x*}
        c_L(X_{\Omega}) = \delta_\Omega.
    \end{IEEEeqnarray*}
\end{copiedtheorem}
\begin{proof}
    For every $k \in \Z_{\geq 1}$,
    \begin{IEEEeqnarray*}{rCls+x*}
        \delta_\Omega
        & \leq & c_L(X_{\Omega})                                         & \quad [\text{by \cref{lem:c square geq delta,lem:c square leq c lag}}] \\
        & \leq & \frac{\tilde{\mathfrak{g}}_{k}(X_{\Omega})}{k} & \quad [\text{by \cref{thm:lagrangian vs g tilde}}] \\
        & =    & \frac{\cgh{k}(X_{\Omega})}{k}                           & \quad [\text{by \cref{prp:g tilde and cgh}}] \\
        & \leq & \frac{\cgh{k}(N(\delta_\Omega))}{k}                     & \quad [\text{$X_{\Omega}$ is convex, hence $X_{\Omega} \subset N(\delta_\Omega)$}] \\
        & =    & \frac{\delta_\Omega(k+1)}{k}                          & \quad [\text{by \cref{lem:cgh of nondisjoint union of cylinders}}].
    \end{IEEEeqnarray*}
    The result follows by taking the infimum over $k$.
\end{proof}

Before we move on to the discussion about computations using linearized contact homology, we show one final result which uses only the properties of $S^1$-equivariant symplectic homology.
\begin{copiedtheorem}{thm:ghc and s1eshc}
    If $(X, \lambda)$ is a Liouville domain, then
    \begin{enumerate}
        \item $\cgh{k}(X) \leq \csh{k}(X)$;
        \item $\cgh{k}(X) = \csh{k}(X)$ provided that $X$ is star-shaped.
    \end{enumerate}
\end{copiedtheorem}

We now present another approach that can be used to compute $c_L$, using linearized contact homology. This has the disadvantage that at the time of writing, linearized contact homology has not yet been defined in the generality that we need (see \cref{sec:assumptions of virtual perturbation scheme} and more specifically \cref{assumption}). Using linearized contact homology, together with an augmentation map, one can define the higher symplectic capacities $\mathfrak{g}_k$. The key idea is that the capacities $\mathfrak{g}_k$ can be compared to $\tilde{\mathfrak{g}}_k$ and $\cgh{k}$.

\begin{copiedtheorem}[{\cite[Section 3.4]{mcduffSymplecticCapacitiesUnperturbed2022}}]{thm:g tilde vs g hat}
    If $X$ is a Liouville domain then
    \begin{IEEEeqnarray*}{c+x*}
        \tilde{\mathfrak{g}}_k(X) \leq {\mathfrak{g}}_k(X).
    \end{IEEEeqnarray*}
\end{copiedtheorem}

\begin{copiedtheorem}{thm:g hat vs gh}
    If $X$ is a Liouville domain such that $\pi_1(X) = 0$ and $2 c_1(TX) = 0$ then%
    \begin{IEEEeqnarray*}{c+x*}
        {\mathfrak{g}}_k(X) = \cgh{k}(X).
    \end{IEEEeqnarray*}
\end{copiedtheorem}

These two results show that $\tilde{\mathfrak{g}}_k(X_\Omega) \leq \cgh{k}(X_\Omega)$ (under \cref{assumption}). Using the same proof as before, we conclude that $c_L(X_{\Omega}) = \delta_\Omega$.

\begin{copiedtheorem}{thm:my main theorem}
    Under \cref{assumption}, if $X_\Omega$ is a convex or concave toric domain then%
    \begin{IEEEeqnarray*}{c+x*}
        c_L(X_{\Omega}) = \delta_\Omega.
    \end{IEEEeqnarray*}
\end{copiedtheorem}

\subsection{Outline}

In \textbf{\cref{sec:2}}, we review some basics about asymptotically cylindrical holomorphic curves in symplectizations. We start by reviewing the definitions of the various types of symplectic manifolds that we will work with, namely Liouville domains, star-shaped domains and toric domains. After, we consider asymptotically cylindrical holomorphic curves, as well as the moduli spaces that they form. We state the (virtual) dimension formula for these moduli spaces, as well as the SFT compactness theorem, which describes their compactifications. Finally, we give a list of properties of $S^1$-equivariant symplectic homology, which is required to define the Gutt--Hutchings capacities.

\textbf{\cref{sec:3}} is about symplectic capacities. The first three subsections are each devoted to defining and proving the properties of a specific capacity, namely the Lagrangian capacity $c_L$, the Gutt--Hutchings capacities $\cgh{k}$ and the $S^1$-equivariant symplectic homology capacities $\csh{k}$, and finally the McDuff--Siegel capacities $\tilde{\mathfrak{g}}_k$. In the subsection about the Lagrangian capacity, we also state the conjecture that we will try to solve in the remainder of the paper, i.e. $c_L(X_{\Omega}) = \delta_\Omega$ for a convex or concave toric domain $X_{\Omega}$. The final subsection is devoted to computations. We show that $c_L(X) \leq \inf_k^{} \tilde{\mathfrak{g}}_k(X) / k$. We use this result to prove the conjecture in the case where $X_{\Omega}$ is $4$-dimensional and convex.

\textbf{\cref{sec:4}} introduces the linearized contact homology of a nondegenerate Liouville domain. The idea is that using the linearized contact homology, one can define the higher symplectic capacities, which will allow us to prove $c_L(X_{\Omega}) = \delta_\Omega$ for any convex or concave toric domain $X_{\Omega}$ (but under the assumption that linearized contact homology and the augmentation map are well-defined). We give a review of real linear Cauchy--Riemann operators on complex vector bundles, with a special emphasis on criteria for surjectivity in the case where the bundle has complex rank $1$. We use this theory to prove that moduli spaces of curves in ellipsoids are transversely cut out and in particular that the augmentation map of an ellipsoid is an isomorphism. The final subsection is devoted to computations. We show that $\mathfrak{g}_k(X) = \cgh{k}(X)$, and use this result to prove our conjecture (again, under \cref{assumption}).

\textbf{Acknowledgements.} I thank Kai Cieliebak for introducing me to this problem and for many valuable discussions about this topic. I am also grateful for useful conversations with Kyler Siegel, and to Michael Hutchings and Jean Gutt for their interest in this project. The contents of this paper are essentially those of my PhD thesis (\cite{Per}), which was carried out at the University of Augsburg under the supervision of Kai Cieliebak.

\section{Preliminaries on holomorphic curves}
\label{sec:2}

\subsection{Liouville domains}

A \textbf{symplectic cobordism} is a compact symplectic manifold $(X, \omega)$ with boundary $\partial X$, together with a $1$-form $\lambda$ defined on an open neighbourhood of $\partial X$, such that $\edv \lambda = \omega$ and the restriction of $\lambda$ to $\partial X$ is a contact form. In this case, we let $\partial^+ X$ (respectively $\partial^- X$) be the subset of $\partial X$ where the orientation defined by $\lambda|_{\partial X}$ as a contact form agrees with the boundary orientation (respectively negative boundary orientation). In the case where $\lambda$ is defined on $X$, we say that $(X,\lambda)$ is a \textbf{Liouville cobordism}. Finally, a \textbf{Liouville domain} is a Liouville cobordism $(X,\lambda)$ such that $\partial^- X = \varnothing$.

Consider the canonical symplectic potential of $\C^n$, given by
\begin{IEEEeqnarray*}{c+x*}
    \lambda \coloneqq \frac{1}{2} \sum_{j=1}^{n} (x^j \edv y^j - y^j \edv x^j).
\end{IEEEeqnarray*}
A \textbf{star-shaped domain} is a subset $X \subset \C^n$ such that $(X,\lambda)$ is a Liouville domain. We will be interested in a further subclass of domains, namely \textbf{toric domains}. To define this notion, first consider the \textbf{moment map} $\mu \colon \C^n \longrightarrow \R^n_{\geq 0}$, which is given by
\begin{IEEEeqnarray*}{c+x*}
    \mu(z_1,\ldots,z_n) \coloneqq \pi(|z_1|^2,\ldots,|z_n|^2),
\end{IEEEeqnarray*}
and define
\begin{IEEEeqnarray*}{rClClrCl}
    \Omega_X      & \coloneqq & \hphantom{{}^{-1}} \mu(X) \subset \R_{\geq 0}^n, & \qquad & \text{for every } & X      & \subset & \C^n, \\
    X_{\Omega}    & \coloneqq & \mu^{-1}(\Omega) \subset \C^n,                   & \qquad & \text{for every } & \Omega & \subset & \R^{n}_{\geq 0}, \\
    \delta_\Omega & \coloneqq & \sup \{ a \mid (a, \ldots, a) \in \Omega \},     & \qquad & \text{for every } & \Omega & \subset & \R^{n}_{\geq 0}.
\end{IEEEeqnarray*}
We call $\delta_{\Omega}$ the \textbf{diagonal} of $\Omega$. With this notation, a toric domain is a star-shaped domain $X$ of the form $X = X_{\Omega}$ for some $\Omega \subset \R^n_{\geq 0}$. We say that a toric domain $X_{\Omega}$ is \textbf{convex} if
\begin{IEEEeqnarray*}{c+x*}
    \hat{\Omega} \coloneqq \{ (x_1, \ldots, x_n) \in \R^n \mid (|x_1|,\ldots,|x_n|) \in \Omega \}
\end{IEEEeqnarray*}
is convex and that it is \textbf{concave} if $\R^n_{\geq 0} \setminus \Omega$ is convex. Some examples of toric domains are the \textbf{ball} $B(a)$, the \textbf{cylinder} $Z(a)$, the \textbf{ellipsoid} $E(a_1,\ldots,a_n)$, the \textbf{cube} $P(a)$ and the \textbf{nondisjoint union of cylinders} $N(a)$\footnote{Strictly speaking, $Z(a)$, $N(a)$ are noncompact and $P(a)$, $N(a)$ have corners, so they do not fit into the definition of toric domain. We will mostly ignore this small discrepancy in nomenclature and refer to them as toric domains anyway.}, which are given by
\begin{IEEEeqnarray*}{rClCrCl}
    B(a)              & \coloneqq & \mu^{-1}(\Omega_{B(a)}),              & \quad & \Omega_{B(a)}              & \coloneqq & \{ x \in \R^n_{\geq 0} \mid x_1 + \cdots + x_n \leq a \}, \\
    Z(a)              & \coloneqq & \mu^{-1}(\Omega_{Z(a)}),              & \quad & \Omega_{Z(a)}              & \coloneqq & \{ x \in \R^n_{\geq 0} \mid x_1 \leq a \}, \\
    E(a_1,\ldots,a_n) & \coloneqq & \mu^{-1}(\Omega_{E(a_1,\ldots,a_n)}), & \quad & \Omega_{E(a_1,\ldots,a_n)} & \coloneqq & \{ x \in \R^n_{\geq 0} \mid x_1 / a_1 + \cdots + x_n / a_n \leq 1 \}, \\
    P(a)              & \coloneqq & \mu^{-1}(\Omega_{P(a)}),              & \quad & \Omega_{P(a)}              & \coloneqq & \{ x \in \R^n_{\geq 0} \mid \forall i = 1, \ldots, n \colon x_i \leq a \}, \\
    N(a)              & \coloneqq & \mu^{-1}(\Omega_{N(a)}),              & \quad & \Omega_{N(a)}              & \coloneqq & \{ x \in \R^n_{\geq 0} \mid \exists i = 1, \ldots, n \colon x_i \leq a \}.
\end{IEEEeqnarray*}

Any Liouville cobordism $(X,\lambda)$ has a \textbf{Liouville vector field} $Z$ which is defined by the equation $\lambda = \iota_Z \edv \lambda$. If $\varphi \colon (X, \lambda_X) \longrightarrow (Y, \lambda_Y)$ is an embedding between Liouville cobordisms, we say that $\varphi$ is
\begin{enumerate}
    \item \label{def:types of embeddings 1} \textbf{symplectic} if $\varphi^* \lambda_Y - \lambda_X$ is closed;
    \item \label{def:types of embeddings 2} \textbf{generalized Liouville} if $\varphi^* \lambda_Y - \lambda_X$ is closed and $(\varphi^* \lambda_Y - \lambda_X)|_{\partial X}$ is exact;
    \item \label{def:types of embeddings 3} \textbf{exact symplectic} if $\varphi^* \lambda_Y - \lambda_X$ is exact;
    \item \label{def:types of embeddings 4} \textbf{Liouville} if $\varphi^* \lambda_Y - \lambda_X = 0$.
\end{enumerate}

The \textbf{symplectization} of a contact manifold $(M,\alpha)$ is the exact symplectic manifold $\R \times M$ whose symplectic potential is $e^r \alpha$, where $r$ denotes the coordinate on $\R$. If $(X,\omega,\lambda)$ is a symplectic cobordism, the \textbf{completion} of $X$ is given by gluing half-symplectizations at $\partial^\pm X$, i.e.
\begin{IEEEeqnarray*}{c+x*}
    (\hat{X}, \hat{\lambda}) \coloneqq (\R_{\leq 0} \times \partial^- X, e^r \lambda|_{\partial^- X}) \cup (X,\lambda) \cup (\R_{\geq 0} \times \partial^+ X, e^r \lambda|_{\partial^+ X}).
\end{IEEEeqnarray*}
where $r$ denotes the coordinate on $\R$. If $(X, \lambda_X)$ and $(Y, \lambda_Y)$ are Liouville cobordisms and $\varphi \colon (X, \lambda_X) \longrightarrow (Y, \lambda_Y)$ is a Liouville embedding such that $Z_X$ is $\varphi$-related to $Z_Y$, then one can define a Liouville embedding $\hat{\varphi} \colon (\hat{X}, \hat{\lambda}_X) \longrightarrow (\hat{Y}, \hat{\lambda}_Y)$. With these definitions, the operation of taking the completion is actually a functor.

\subsection{Holomorphic curves}

In this section, we review some basics about asymptotically cylindrical holomorphic curves. Standard references for this are \cite{hoferPseudoholomorphicCurvesSymplectizations1993} and \cite{eliashbergIntroductionSymplecticField2010}. Our presentation will be based on \cite{wendlLecturesSymplecticField2016} and \cite[Section 2.1]{mcduffSymplecticCapacitiesUnperturbed2022}.

Let $(M, \alpha)$ be a contact manifold and consider its symplectization $(\R \times M, e^r \alpha)$. Recall that $M$ has a \textbf{Reeb vector field} $R$, given by
\begin{IEEEeqnarray*}{c+x*}
    \iota_R \alpha = 1, \quad \iota_R \edv \alpha = 0,
\end{IEEEeqnarray*}
and a \textbf{contact distribution} $\xi \coloneqq \ker \alpha$. An almost complex structure $J$ on $\R \times M$ is \textbf{cylindrical} if $J(\partial_r) = R$, if $J(\xi) \subset \xi$, and if the almost complex structure $J \colon \xi \longrightarrow \xi$ is compatible with $\edv \alpha$ and independent of $r$. Denote by $\mathcal{J}(M)$ the set of such $J$. If $(X,\omega,\lambda)$ is a symplectic cobordism and $J$ is an almost complex structure on $\hat{X}$, we say that $J$ is \textbf{cylindrical} if its restriction to each symplectization end is cylindrical. Denote by $\mathcal{J}(X)$ the set of such $J$. If $J^{\pm} \in \mathcal{J}(\partial^{\pm} X)$, denote
\begin{IEEEeqnarray*}{rCls+x*}
    \mathcal{J}^{J^+}(X)                  & \coloneqq & \{ J \in \mathcal{J}(X) \mid J = J^{+} \text{ on } \R_{\geq 0} \times \partial^+ X \}, \\
    \mathcal{J}_{J^-}^{\hphantom{J^+}}(X) & \coloneqq & \{ J \in \mathcal{J}(X) \mid J = J^{-} \text{ on } \R_{\leq 0} \times \partial^- X \}, \\
    \mathcal{J}^{J^+}_{J^-}(X)            & \coloneqq & \mathcal{J}^{J^+}(X) \cap \mathcal{J}_{J^-}(X).
\end{IEEEeqnarray*}

Let $(\Sigma,j)$ be a compact Riemann surface without boundary and $\mathbf{z}^{\pm} \coloneqq \{z^{\pm}_1,\ldots,z^{\pm}_{p^{\pm}}\} \subset \Sigma$ be finite sets of positive and negative punctures, and denote $\dot{\Sigma} \coloneqq \Sigma \setminus \mathbf{z}^{-} \cup \mathbf{z}^+$. An \textbf{asymptotically cylindrical holomorphic curve} is a holomorphic map $u \colon (\dot{\Sigma}, j) \longrightarrow (\hat{X}, J)$ such that $u$ is positively (respectively negatively) asymptotic to a Reeb orbit of $\partial^+ X$ (respectively $\partial^- X$) at every $z \in \mathbf{z}^+$ (respectively $z \in \mathbf{z}^-$). For more details see \cite{wendlLecturesSymplecticField2016}. We will denote by $\Gamma^{\pm} = (\gamma^{\pm}_1,\ldots,\gamma^{\pm}_{p^{\pm}})$ the tuples of Reeb orbits in $\partial^{\pm} X$ that $u$ is asymptotic to.

Define a piecewise smooth $2$-form $\tilde{\omega} \in \Omega^2(\hat{X})$ by
\begin{IEEEeqnarray}{c+x*}
    \tilde{\omega}
    \coloneqq
    \begin{cases}
        \edv \lambda|_{\partial^+ X} & \text{on } \R_{\geq 0} \times \partial^+ X, \\
        \omega                       & \text{on } X, \\
        \edv \lambda|_{\partial^- X} & \text{on } \R_{\leq 0} \times \partial^- X.
    \end{cases}
    \phantomsection\label{eq:form 2}
\end{IEEEeqnarray}
If $u$ is an asymptotically cylindrical holomorphic curve, its \textbf{energies} are given by
\begin{IEEEeqnarray}{rClCl}
    E_{\hat{\omega}}(u)   & \coloneqq & \int_{\dot{\Sigma}}^{} u^* \hat{\omega},   \phantomsection\label{eq:energy 1} \\
    E_{\tilde{\omega}}(u) & \coloneqq & \int_{\dot{\Sigma}}^{} u^* \tilde{\omega}. \phantomsection\label{eq:energy 2}
\end{IEEEeqnarray}
In the case where $(X,\omega,\lambda)$ is a Liouville cobordism, Stokes' theorem implies that
\begin{IEEEeqnarray}{c+x*}
    \phantomsection\label{eq:energy identity}
    0 \leq E_{\tilde{\omega}}(u) = \mathcal{A}(\Gamma^+) - \mathcal{A}(\Gamma^-),
\end{IEEEeqnarray}
where
\begin{IEEEeqnarray*}{c+x*}
    \mathcal{A}(\Gamma^{\pm}) \coloneqq \sum_{i=1}^{p^{\pm}} \mathcal{A}(\gamma^{\pm}_i), \quad \mathcal{A}(\gamma^{\pm}_i) \coloneqq \int_{\gamma^{\pm}_i}^{} \lambda|_{\partial^{\pm} X}.
\end{IEEEeqnarray*}
Here, $\mathcal{A}(\gamma^{\pm}_i)$ is the \textbf{action} of the Reeb orbit $\gamma^{\pm}_i$. In particular, $u$ must have at least one positive puncture. Another useful result to rule out certain behaviours of holomorphic curves is the \textbf{maximum principle}, which works as follows. Suppose that the target of $u$ is a symplectization, i.e. $u = (a,f) \colon \dot{\Sigma} \longrightarrow \R \times M$. The fact that $u$ is holomorphic with respect to a cylindrical almost complex structure implies that $\Delta a \geq 0$, where $\Delta$ denotes the Laplacian. By the maximum principle for elliptic partial differential operators, $a$ cannot have any local maxima. We finish this subsection with a result which we will need to prove \cref{thm:lagrangian vs g tilde}.

\begin{lemma}
    \label{lem:energy wrt different forms}
    Assume that $\Sigma$ has no positive punctures. Let $(X, \omega, \lambda)$ be a symplectic cobordism, and $J \in \mathcal{J}(X)$ be a cylindrical almost complex structure on $\hat{X}$. Assume that the canonical symplectic embedding
    \begin{align*}
        (\R_{\leq 0} \times \partial^- X, \edv (e^r \lambda|_{\partial^- X})) \longrightarrow (\hat{X}, \hat{\omega}) & \\
        \intertext{can be extended to a symplectic embedding}
        (\R_{\leq K} \times \partial^- X, \edv (e^r \lambda|_{\partial^- X})) \longrightarrow (\hat{X}, \hat{\omega}) &
    \end{align*}
    for some $K > 0$. Let $u \colon \dot{\Sigma} \longrightarrow \hat{X}$ be a $J$-holomorphic curve which is negatively asymptotic to a tuple of Reeb orbits $\Gamma$ of $\partial^- X$. Consider the energies $E_{\hat{\omega}}(u)$ and $E_{\tilde{\omega}}(u)$. Then,
    \begin{IEEEeqnarray}{rCls+x*}
        \mathcal{A}(\Gamma) & \leq & \frac{1  }{e^K - 1} E_{\tilde{\omega}}(u), \phantomsection\label{eq:action is bounded by vertical energy} \\
        E_{\hat{\omega}}(u) & \leq & \frac{e^K}{e^K - 1} E_{\tilde{\omega}}(u). \phantomsection\label{eq:energy is bounded by vertical energy}
    \end{IEEEeqnarray}
\end{lemma}
\begin{proof}
    It is enough to show that
    \begin{IEEEeqnarray}{rCls+x*}
        E_{\hat{\omega}}(u) - E_{\tilde{\omega}}(u) & =    & \mathcal{A}(\Gamma), \phantomsection\label{eq:vertical energy bounds 1} \\
        E_{\hat{\omega}}(u)                         & \geq & e^K \mathcal{A}(\Gamma), \phantomsection\label{eq:vertical energy bounds 2}
    \end{IEEEeqnarray}
    since these equations imply Equations \eqref{eq:action is bounded by vertical energy} and \eqref{eq:energy is bounded by vertical energy}. Since $u$ has no positive punctures, the maximum principle implies that $u$ is contained in $\R_{\leq 0} \times \partial^- X \cup X$. We prove Equation \eqref{eq:vertical energy bounds 1}. For simplicity, denote $M = \partial^- X$ and $\alpha = \lambda|_{\partial^- X}$.
    \begin{IEEEeqnarray*}{rCls+x*}
        E_{\hat{\omega}}(u) - E_{\tilde{\omega}}(u)
        & = & \int_{\dot{\Sigma}}^{} u^* (\hat{\omega} - \tilde{\omega})         & \quad [\text{by definition of $E_{\hat{\omega}}$ and $E_{\tilde{\omega}}$}] \\
        & = & \int_{u^{-1}(\R_{\leq 0} \times M)}^{} u^* \edv ((e^r - 1) \alpha) & \quad [\text{by definition of $\hat{\omega}$ and $\tilde{\omega}$}] \\
        & = & \mathcal{A}(\Gamma)                                                & \quad [\text{by Stokes' theorem}].
    \end{IEEEeqnarray*}
    We prove Equation \eqref{eq:vertical energy bounds 2}.
    \begin{IEEEeqnarray*}{rCls+x*}
        E_{\hat{\omega}}(u)
        & =    & \int_{\dot{\Sigma}}^{} u^* \hat{\omega}                                               & \quad [\text{by definition of $E_{\hat{\omega}}$}] \\
        & \geq & \int_{u^{-1}(\R_{\leq K} \times M)}^{} u^* \edv (e^r \alpha)                          & \quad [\text{by definition of $\hat{\omega}$ and $u^* \hat{\omega} \geq 0$}] \\
        & =    & e^K \int_{u^{-1}( \{K\} \times M)}^{} u^* \alpha                                      & \quad [\text{by Stokes' theorem}] \\
        & =    & e^K \int_{u^{-1}( \R_{\leq K} \times M)}^{} u^* \edv \alpha + e^K \mathcal{A}(\Gamma) & \quad [\text{by Stokes' theorem}] \\
        & \geq & e^K \mathcal{A}(\Gamma)                                                               & \quad [\text{since $J$ is cylindrical}].     & \qedhere
    \end{IEEEeqnarray*}
\end{proof}

\subsection{Moduli spaces}

Let $\Gamma^{\pm} = (\gamma^{\pm}_1, \ldots, \gamma^{\pm}_{p ^{\pm}})$ be a tuple of Reeb orbits in $\partial^{\pm} X$ and $J \in \mathcal{J}(X)$ be a cylindrical almost complex structure on $\hat{X}$. Define a moduli space
\begin{IEEEeqnarray*}{c+x*}
    \mathcal{M}^{J}_{X}(\Gamma^+, \Gamma^-)
    \coloneqq
    \left\{
        (\Sigma, u)
        \ \middle\vert
        \begin{array}{l}
            \Sigma \text{ is a connected closed Riemann surface} \\
            \text{of genus $0$ with punctures $\mathbf{z}^{\pm} = \{z^{\pm}_1, \ldots, z^{\pm}_{p ^{\pm}}\}$,} \\
            u \colon \dot{\Sigma} \longrightarrow \hat{X} \text{ is holomorphic and asymptotic to } \Gamma^{\pm}
        \end{array}
        \right\} / \sim,
\end{IEEEeqnarray*}
where $(\Sigma_0, u_0) \sim (\Sigma_1, u_1)$ if and only if there exists a biholomorphism $\phi \colon \Sigma_0 \longrightarrow \Sigma_1$ such that $u_1 \circ \phi = u_0$ and $\phi(z^{\pm}_{0,i}) = z^{\pm}_{1,i}$ for every $i = 1,\ldots,p ^{\pm}$. If $\Gamma^{\pm} = (\gamma^{\pm}_1, \ldots, \gamma^{\pm}_{p ^{\pm}})$ is a tuple of Reeb orbits on a contact manifold $M$ and $J \in \mathcal{J}(M)$, we define a moduli space $\mathcal{M}_{M}^{J}(\Gamma^+, \Gamma^-)$ of holomorphic curves in $\R \times M$ analogously. Since $J$ is invariant with respect to translations in the $\R$ direction, $\mathcal{M}_{M}^{J}(\Gamma^+, \Gamma^-)$ admits an action of $\R$ by composition on the target by a translation.

One can try to show that the moduli space $\mathcal{M}_{X}^{J}(\Gamma^+, \Gamma^-)$ is transversely cut out by showing that the relevant linearized Cauchy--Riemann operator is surjective at every point of the moduli space. In this case, the moduli space is an orbifold whose dimension is given by the Fredholm index of the linearized Cauchy--Riemann operator. However, since the curves in $\mathcal{M}_{X}^{J}(\Gamma^+, \Gamma^-)$ are not necessarily simple, this proof will in general not work, and we cannot say that the moduli space is an orbifold. However, the Fredholm theory part of the proof still works, which means that we still have a dimension formula. In this case the expected dimension given by the Fredholm theory is usually called a virtual dimension. For the moduli space above, the virtual dimension at a point $u$ is given by (see \cite[Section 4]{bourgeoisCoherentOrientationsSymplectic2004})
\begin{IEEEeqnarray*}{c}
    \operatorname{virdim}_u \mathcal{M}_{X}^{J}(\Gamma^+, \Gamma^-) = (n - 3)(2 - p^+ - p^-) + c_1^{\tau}(u^* T \hat{X}) + \conleyzehnder^{\tau} (\Gamma^+) - \conleyzehnder^{\tau} (\Gamma^-),
\end{IEEEeqnarray*}
where $\tau$ is a unitary trivialization of the contact distribution over each Reeb orbit, $c_1^{\tau}$ is the first Chern class and $\conleyzehnder^{\tau}(\Gamma)$ is the sum of the Conley--Zehnder indices of the Reeb orbits in $\Gamma$.

We now discuss curves satisfying a tangency constraint. Our presentation is based on \cite[Section 2.2]{mcduffSymplecticCapacitiesUnperturbed2022} and \cite[Section 3]{cieliebakPuncturedHolomorphicCurves2018}. Let $(X,\omega,\lambda)$ be a symplectic cobordism and $x \in \itr X$. A \textbf{symplectic divisor} through $x$ is a germ of a $2$-codimensional symplectic submanifold $D \subset X$ containing $x$. A cylindrical almost complex structure $J \in \mathcal{J}(X)$ is \textbf{compatible} with $D$ if $J$ is integrable near $x$ and $D$ is holomorphic with respect to $J$. We denote by $\mathcal{J}(X,D)$ the set of such almost complex structures. In this case, there are complex coordinates $(z_1, \ldots, z_n)$ near $x$ such that $D$ is given by $h(z_1,\ldots,z_n) = 0$, where $h(z_1,\ldots,z_n) = z_1$. Let $u \colon \Sigma \longrightarrow X$ be a $J$-holomorphic curve together with a marked point $w \in \Sigma$. For $k \geq 1$, we say that $u$ has \textbf{contact order $k$} to $D$ at $x$ if $u(w) = x$ and%
\begin{IEEEeqnarray*}{c+x*}
    (h \circ u \circ \varphi)^{(1)}(0) = \cdots = (h \circ u \circ \varphi)^{(k-1)}(0) = 0,
\end{IEEEeqnarray*}
for some local biholomorphism $\varphi \colon (\C,0) \longrightarrow (\Sigma, w)$. We point out that the condition of having ``contact order $k$'' as written above is equal to the condition of being ``tangent of order $k-1$'' as defined in \cite[Section 3]{cieliebakPuncturedHolomorphicCurves2018}. Following \cite{mcduffSymplecticCapacitiesUnperturbed2022}, we will use the notation $\p{<}{}{\mathcal{T}^{(k)}x}$ to denote moduli spaces of curves which have contact order $k$, i.e. we will denote them by $\mathcal{M}_{X}^{J}(\Gamma^+, \Gamma^-)\p{<}{}{\mathcal{T}^{(k)}x}$ and $\mathcal{M}_{M}^{J}(\Gamma^+, \Gamma^-)\p{<}{}{\mathcal{T}^{(k)}x}$. The virtual dimension is given by (see \cite[Equation (2.2.1)]{mcduffSymplecticCapacitiesUnperturbed2022})
\begin{IEEEeqnarray}{l}
    \phantomsection\label{eq:virtual dimension}
    \operatorname{virdim}_u \mathcal{M}_{X}^{J}(\Gamma^+, \Gamma^-)\p{<}{}{\mathcal{T}^{(k)}x} \\
    \quad = (n - 3)(2 - p^+ - p^-) + c_1^{\tau}(u^* T \hat{X}) + \conleyzehnder^{\tau} (\Gamma^+) - \conleyzehnder^{\tau} (\Gamma^-) - 2n - 2k + 4. \nonumber
\end{IEEEeqnarray}

We finish this subsection with two lemmas by Cieliebak--Mohnke which we will use in the proof of \cref{thm:lagrangian vs g tilde}.

\begin{lemma}[{\cite[Lemma 2.2]{cieliebakPuncturedHolomorphicCurves2018}}]
    \label{lem:geodesics lemma CM abs}
    Let $L$ be a compact $n$-dimensional manifold without boundary. Let $\mathrm{Riem}(L)$ be the set of Riemannian metrics on $L$, equipped with the $C^2$-topology. If $g_0 \in \mathrm{Riem}(L)$ is a Riemannian metric of nonpositive sectional curvature and $\mathcal{U} \subset \mathrm{Riem}(L)$ is an open neighbourhood of $g_0$, then for all $\ell_0 > 0$ there exists a Riemannian metric $g \in \mathcal{U}$ on $L$ such that with respect to $g$, any closed geodesic $c$ in $L$ of length $\ell(c) \leq \ell_0$ is noncontractible, nondegenerate, and such that $0 \leq \morse(c) \leq n - 1$.
\end{lemma}

\begin{lemma}[{\cite[Corollary 3.3]{cieliebakPuncturedHolomorphicCurves2018}}]
    \label{lem:punctures and tangency}
    Let $(L,g)$ be an $n$-dimensional Riemannian manifold with the property that for some $\ell_0 > 0$, all closed geodesics $\gamma$ of length $\ell(\gamma) \leq \ell_0$ are noncontractible and nondegenerate and have Morse index $\morse(\gamma) \leq n - 1$. Let $x \in T^*L$ and $D$ be a symplectic divisor through $x$. For generic $J$ every (not necessarily simple) punctured $J$-holomorphic sphere $\tilde{C}$ in $T^*L$ which is asymptotic at the punctures to geodesics of length $\leq \ell_0$ and which has contact order $\tilde{k}$ to $D$ at $x$ must have at least $\tilde{k} + 1$ punctures.
\end{lemma}

\subsection{SFT compactness}

In this subsection we present the SFT compactness theorem, which describes the compactifications of the moduli spaces of the previous subsection. This theorem was first proven by Bourgeois--Eliashberg--Hofer--Wysocki--Zehnder \cite{bourgeoisCompactnessResultsSymplectic2003}. Cieliebak--Mohnke \cite{cieliebakCompactnessPuncturedHolomorphic2005} have given a proof of this theorem using different methods. Our presentation is based primarily on \cite{cieliebakPuncturedHolomorphicCurves2018} and \cite{mcduffSymplecticCapacitiesUnperturbed2022}.

Let $(X, \omega, \lambda)$ be a symplectic cobordism and choose almost complex structures $J^{\pm} \in \mathcal{J}(\partial^{\pm} X)$ and $J \in \mathcal{J}^{J^+}_{J^-}(X)$. Let $\Gamma^{\pm} = (\gamma^{\pm}_1, \ldots, \gamma^{\pm}_{p ^{\pm}})$ be a tuple of Reeb orbits in $\partial^{\pm} X$. For $1 \leq L \leq N$, let $\alpha^{\pm} \coloneqq \lambda|_{\partial^{\pm} X}$ and define
\begin{IEEEeqnarray*}{rCl}
    (X^{\nu}, \omega^\nu, \tilde{\omega}^{\nu}, J^{\nu})
    & \coloneqq &
    \begin{cases}
        (\R \times \partial^- X, \edv(e^r \alpha^-), \edv \alpha^-  , J^-) & \text{if } \nu = 1   , \ldots, L - 1, \\
        (\hat{X}               , \hat{\omega}      , \tilde{\omega} , J  ) & \text{if } \nu = L   , \\
        (\R \times \partial^+ X, \edv(e^r \alpha^+), \edv \alpha^+  , J^+) & \text{if } \nu = L+1 ,\ldots ,N     ,
    \end{cases} \\
    (X^*, \omega^*, \tilde{\omega}^*, J^*) & \coloneqq & \coprod_{\nu = 1}^N (X^{\nu}, \omega^\nu, \tilde{\omega}^{\nu}, J^{\nu}).
\end{IEEEeqnarray*}
The moduli space of \textbf{holomorphic buildings}, denoted $\overline{\mathcal{M}}^{J}_X(\Gamma^+, \Gamma^-)$, is the set of tuples $F = (F^1, \ldots, F^N)$, where $F^{\nu} \colon \dot{\Sigma}^\nu \longrightarrow X^\nu$ is an asymptotically cylindrical nodal $J^{\nu}$-holomorphic curve in $X^{\nu}$ with sets of asymptotic Reeb orbits $\Gamma^{\pm}_{\nu}$. Here, each $F^{\nu}$ is possibly disconnected and if $X^{\nu}$ is a symplectization then $F^{\nu}$ is only defined up to translation in the $\R$ direction. We assume in addition that $F$ satisfies the following conditions.
\begin{enumerate}
    \item \label{item:sft compactification 1} The tuples of asymptotic Reeb orbits $\Gamma_{\nu}^{\pm}$ are such that
        \begin{IEEEeqnarray*}{c+x*}
            \Gamma^-_1 = \Gamma^-, \quad \Gamma^+_N = \Gamma^+, \quad \Gamma^+_{\nu} = \Gamma^-_{\nu + 1} \quad \text{for every } \nu = 1, \ldots, N - 1.
        \end{IEEEeqnarray*}
    \item \label{item:sft compactification 2} Define the graph of $F$ to be the graph whose vertices are the components of $F^1, \ldots, F^N$ and whose edges are determined by the asymptotic Reeb orbits. Then the graph of $F$ is a tree.
    \item \label{item:sft compactification 3} The building $F$ has no symplectization levels consisting entirely of trivial cylinders, and any constant component of $F$ has negative Euler characteristic after removing all special points.
\end{enumerate}
The \textbf{energy} of a holomorphic building $F = (F^1, \ldots, F^N)$ is $E_{\tilde{\omega}^*}(F) \coloneqq \sum_{\nu = 1}^{N} E_{\tilde{\omega}^{\nu}}(F^{\nu})$. The moduli space $\overline{\mathcal{M}}_X^J(\Gamma^+, \Gamma^-)$ admits a metrizable topology (see \cite[Appendix B]{bourgeoisEquivariantSymplecticHomology2016}). With this language, the SFT compactness theorem can be stated by saying that $\overline{\mathcal{M}}_X^J(\Gamma^+, \Gamma^-)$ is compact.

We now consider the case where the almost complex structure on $\hat{X}$ is replaced by a family of almost complex structures obtained via \textbf{neck stretching}. Let $(X^{\pm}, \omega^{\pm}, \lambda^{\pm})$ be symplectic cobordisms with common boundary
\begin{IEEEeqnarray*}{c+x*}
    (M, \alpha) = (\partial^- X^{+}, \lambda^+|_{\partial^- X^+}) = (\partial^+ X^-, \lambda^-|_{\partial^+ X^-}).
\end{IEEEeqnarray*}
Choose almost complex structures $J_M \in \mathcal{J}(M)$, $J_+ \in \mathcal{J}_{J_M}(X^+)$, and $J_- \in \mathcal{J}^{J_M}(X^-)$ and denote by $J_{\partial^{\pm} X^{\pm}} \in \mathcal{J}(\partial^{\pm} X^{\pm})$ the induced cylindrical almost complex structure on $\R \times \partial^{\pm} X^{\pm}$. Let $(X, \omega, \lambda) \coloneqq (X^-, \omega^-, \lambda^-) \circledcirc (X^+, \omega^+, \lambda^+)$ be the gluing of $X^-$ and $X^+$ along $M$. We wish to define a family of almost complex structures $(J_t)_{t \in \R_{\geq 0}} \subset \mathcal{J}(X)$. For every $t \geq 0$, let
\begin{IEEEeqnarray*}{c+x*}
    X_t \coloneqq X^- \cup_M [-t, 0] \times M \cup_M X^+.
\end{IEEEeqnarray*}
There exists a canonical diffeomorphism $\phi_t \colon X \longrightarrow X_t$. Define an almost complex structure $J_t$ on $X_t$ by
\begin{IEEEeqnarray*}{c+x*}
    J_t \coloneqq
    \begin{cases}
        J^{\pm} & \text{on } X^{\pm}, \\
        J_M     & \text{on } [-t, 0] \times M.
    \end{cases}
\end{IEEEeqnarray*}
Denote also by $J_t$ the pullback of $J_t$ to ${X}$ along $\phi_t$, as well as the induced almost complex structure on the completion $\hat{X}$. Finally, consider the moduli space
\begin{IEEEeqnarray*}{c+x*}
    \mathcal{M}_X^{(J_t)_t}(\Gamma^+, \Gamma^-) \coloneqq \coprod_{t \in \R_{\geq 0}} \mathcal{M}^{J_t}_{X}(\Gamma^+, \Gamma^-).
\end{IEEEeqnarray*}
As before, we wish to define a suitable compactification for $\mathcal{M}_X^{(J_t)_t}(\Gamma^+, \Gamma^-)$. For $1 \leq L^- < L^+ \leq N$, let $\alpha^{\pm} \coloneqq \lambda^{\pm}|_{\partial^{\pm} X^\pm}$ and define
\begin{IEEEeqnarray*}{rCls+x*}
    (X^{\nu}, \omega^\nu, \tilde{\omega}^{\nu}, J^{\nu})
    & \coloneqq &
    \begin{cases}
        (\R \times \partial^- X^-, \edv(e^r \alpha^-)  , \edv \alpha^-   , J_{\partial^- X^-}) & \text{if } \nu = 1        , \ldots, L^- - 1, \\
        (X^-                     , \omega^-            , \tilde{\omega}^-, J^-)                & \text{if } \nu = L^-, \\
        (\R \times M             , \edv(e^r \alpha)    , \edv \alpha     , J_M)                & \text{if } \nu = L^- + 1    , \ldots, L^+ - 1, \\
        (X^+                     , \omega^+            , \tilde{\omega}^+, J^+)                & \text{if } \nu = L^+, \\
        (\R \times \partial^+ X^+, \edv (e^r \alpha^+) , \edv \alpha^+   , J_{\partial^+ X^+}) & \text{if } \nu = L^+ + 1  , \ldots, N      , \\
    \end{cases} \\
    (X^*, \omega^*, \tilde{\omega}^*, J^*) & \coloneqq & \coprod_{\nu = 1}^N (X^{\nu}, \omega^\nu, \tilde{\omega}^{\nu}, J^{\nu}).
\end{IEEEeqnarray*}
Define $\overline{\mathcal{M}}^{(J_t)_t}_X(\Gamma^+, \Gamma^-)$ to be the set of tuples $F = (F^1, \ldots, F^N)$, where $F^{\nu} \colon \dot{\Sigma}^\nu \longrightarrow X^\nu$ is an asymptotically cylindrical nodal $J^{\nu}$-holomorphic curve in $X^{\nu}$ with sets of asymptotic Reeb orbits $\Gamma^{\pm}_{\nu}$, such that $F$ satisfies conditions analogous to those of \cref{item:sft compactification 1,item:sft compactification 2,item:sft compactification 3}. Then, $\overline{\mathcal{M}}^{(J_t)_t}_X(\Gamma^+, \Gamma^-)$ is compact.

\begin{remark}
    \phantomsection\label{rmk:compactifications with tangency}
    The discussion above also applies to compactifications of moduli spaces of curves satisfying tangency constraints. The compactification $\overline{\mathcal{M}}^{J}_{X}(\Gamma^+,\Gamma^-)\p{<}{}{\mathcal{T}^{(k)}x}$ consists of buildings $F = (F^1, \ldots, F^N) \in \overline{\mathcal{M}}^J_X(\Gamma^+, \Gamma^-)$ such that exactly one component $C$ of $F$ inherits the tangency constraint $\p{<}{}{\mathcal{T}^{(k)}x}$, and which satisfy the following additional condition. Consider the graph obtained from the graph of $F$ by collapsing adjacent constant components to a point. Let $C_1, \ldots, C_p$ be the (necessarily nonconstant) components of $F$ which are adjacent to $C$ in the new graph. Then we require that there exist $k_1, \ldots, k_p \in \Z_{\geq 1}$ such that $k_1 + \cdots + k_p \geq k$ and $C_i$ satisfies the constraint $\p{<}{}{\mathcal{T}^{(k_i)}x}$ for every $i = 1, \ldots, p$. This definition is natural to consider by \cite[Lemma 7.2]{cieliebakSymplecticHypersurfacesTransversality2007}. We can define $\overline{\mathcal{M}}^{(J_t)_t}_X(\Gamma^+, \Gamma^-)\p{<}{}{\mathcal{T}^{(k)}x}$ analogously.
\end{remark}

\begin{remark}
    We point out that in \cite[Definition 2.2.1]{mcduffSymplecticCapacitiesUnperturbed2022}, the compactification of \cref{rmk:compactifications with tangency} is denoted by $\overline{\overline{\mathcal{M}}}^{J}_{X}(\Gamma^+,\Gamma^-)\p{<}{}{\mathcal{T}^{(k)}x}$, while the notation $\overline{\mathcal{M}}^{J}_{X}(\Gamma^+,\Gamma^-)\p{<}{}{\mathcal{T}^{(k)}x}$ is used to denote the moduli space of buildings $F = (F^1, \ldots, F^N) \in \overline{\mathcal{M}}^J_X(\Gamma^+, \Gamma^-)$ such that exactly one component $C$ of $F$ inherits the tangency constraint $\p{<}{}{\mathcal{T}^{(k)}x}$, but which do not necessarily satisfy the additional condition of \cref{rmk:compactifications with tangency}.
\end{remark}

The following lemma will be useful to us in proving \cref{thm:lagrangian vs g tilde}.

\begin{lemma}[{\cite[Lemma 2.8]{cieliebakPuncturedHolomorphicCurves2018}}]
    \label{lem:no nodes}
    The homology class $A \coloneqq [\overline{F}] \in H_2(X;\Z)$ of a nonconstant broken holomorphic curve $F \colon (\Sigma^*, j) \longrightarrow (X^*, J^*)$ satisfies $\omega(A) > 0$.
\end{lemma}

\subsection{\texorpdfstring{$S^1$}{S1}-equivariant symplectic homology}

If $(X, \lambda)$ is a nondegenerate Liouville domain, one can define its \textbf{$S^1$-equivariant symplectic homology}, denoted $SH^{S^1}(X, \lambda)$. The presentation we will give will be based on \cite{guttSymplecticCapacitiesPositive2018}. Other references discussing $S^1$-equivariant symplectic homology are \cite{guttMinimalNumberPeriodic2014,guttPositiveEquivariantSymplectic2017,bourgeoisGysinExactSequence2013,bourgeoisFredholmTheoryTransversality2010,bourgeoisEquivariantSymplecticHomology2016,seidelBiasedViewSymplectic2008}. The $S^1$-equivariant symplectic homology is a $\Q$-module which has the following structural properties.

\begin{enumerate}
    \item \textbf{Action filtration.} For every $a, b \in \R$ we have $\Q$-modules and maps%
        \begin{IEEEeqnarray*}{rCl}
            \iota^a     \colon SH^{S^1,a}(X,\lambda) & \longrightarrow & SH^{S^1}(X,\lambda), \\
            \iota^{b,a} \colon SH^{S^1,a}(X,\lambda) & \longrightarrow & SH^{S^1,b}(X,\lambda),
        \end{IEEEeqnarray*}
        which compose in a functorial way. In particular, we can define the $S^1$-equivariant symplectic homology associated to intervals $(a,b] \subset \R$ and $(a, +\infty) \subset \R$ by taking the quotient:
        \begin{IEEEeqnarray*}{rCl}
            SH^{S^1,(a,b]}(X,\lambda)       & \coloneqq & SH^{S^1,b}(X,\lambda) / \iota^{b,a}(SH^{S^1,a}(X,\lambda)), \\
            SH^{S^1,(a,+\infty)}(X,\lambda) & \coloneqq & SH^{S^1}(X,\lambda) / \iota^{a}(SH^{S^1,a}(X,\lambda)).
        \end{IEEEeqnarray*}
        The \textbf{positive $S^1$-equivariant symplectic homology} is given by $SH^{S^1,+}(X,\lambda) = SH^{S^1,(\varepsilon, +\infty)}(X,\lambda)$, where $\varepsilon$ is half of the minimal action of a Reeb orbit in $\partial X$.
    \item \textbf{$U$ map.} There is a map $U \colon SH^{S^1}(X,\lambda) \longrightarrow SH^{S^1}(X,\lambda)$ which respects the action filtration, i.e. there exist maps $U^a \colon SH^{S^1,a}(X,\lambda) \longrightarrow SH^{S^1,a}(X,\lambda)$ such that $\iota^{a} \circ U^{a} = U \circ \iota^{a}$ and $\iota^{b,a} \circ U^{a} = U^{b} \circ \iota^{b,a}$.
    \item \textbf{$\delta$ map.} There is a map $\delta \colon SH^{S^1}(X,\lambda) \longrightarrow H_{\bullet}(BS^1; \Q) \otimes H_{\bullet}(X, \partial X; \Q)$, which is of the form $\delta \coloneqq \alpha \circ \delta_0$. Here, $\alpha \colon SH^{S^1,\varepsilon}(X) \longrightarrow H_{\bullet}(BS^1; \Q) \otimes H_{\bullet}(X, \partial X; \Q)$ is an isomorphism and $\delta_0$ is the continuation map of the long exact homology sequence
        \begin{IEEEeqnarray*}{c+x*}
            \begin{tikzcd}
                \cdots \ar[r] & SH^{S^1}(X) \ar[r] & SH^{S^1,+}(X) \ar[r, "\delta_0"] & SH^{S^1,\varepsilon}(X) \ar[r] & \cdots
            \end{tikzcd}
        \end{IEEEeqnarray*}
    \item \textbf{Viterbo transfer map.} If $\varphi \colon (X,\lambda_X) \longrightarrow (Y,\lambda_Y)$ is a generalized Liouville embedding with $\varphi(X) \subset \itr(Y)$, one can define a map $\varphi_! \colon SH^{S^1}(Y) \longrightarrow SH^{S^1}(X)$. This map has the following properties. First, $\varphi_!$ commutes with the action filtration, in the sense that for each $a \in \R$ there exists $\varphi_!^{a} \colon SH^{S^1,a}(Y) \longrightarrow SH^{S^1,a}(X)$ such that $\iota^{a}_X \circ \varphi_!^{a} = \varphi_! \circ \iota^{a}_Y$ and $\iota^{b,a}_X \circ \varphi_!^{a} = \varphi_!^{b} \circ \iota^{b,a}_Y$. Second, $\varphi_!$ commutes with the $U$ maps, i.e. $\varphi_!^{a} \circ U_Y^{a} = U_X^{a} \circ \varphi_!^{a}$. Finally, $\varphi_!$ commutes with the $\delta$ map, i.e. $\delta_X \circ \varphi_! = (1 \otimes \rho) \circ \delta_Y$, where $\rho \colon H_{\bullet}(Y, \partial Y; \Q) \longrightarrow H_{\bullet}(X, \partial X; \Q)$ is the composition%
        \begin{IEEEeqnarray*}{c+x*}
            \begin{tikzcd}
                H_{\bullet}(Y, \partial Y; \Q) \ar[r] \ar[rr, bend right = 20, swap, "\rho"] & H_{\bullet}(Y, Y \setminus \varphi(\itr X); \Q) & H_{\bullet}(X, \partial X; \Q) \ar[l, hook', two heads]
            \end{tikzcd}
        \end{IEEEeqnarray*}
    \item \textbf{Grading.} In the case where $\pi_1(X) = 0$ and $c_1(TX)|_{\pi^2(X)} = 0$, the $S^1$-equivariant symplectic homology admits an integer grading. With respect to this grading, the maps $\iota^{a}$, $\iota^{b,a}$ and $\varphi_!$ are of degree 0 and the $U$ map is of degree $-2$.
    \item \textbf{Star-shaped domains.} Suppose that $(X,\lambda)$ is a star-shaped domain. Then,
        \begin{IEEEeqnarray*}{c+x*}
            SH^{S^1}_\bullet(X,\lambda) \cong
            \begin{cases}
                \Q & \text{if } \bullet \in n - 1 + 2 \Z_{\geq 1}, \\
                0  & \text{otherwise}
            \end{cases}
        \end{IEEEeqnarray*}
        and $\delta \colon SH^{S^1}_{n-1+2k}(X,\lambda) \longrightarrow H_{2k-2}(BS^1;\Q) \otimes H_{2n}(X,\partial X; \Q)$ is an isomorphism.
\end{enumerate}

\section{Computations using only classical transversality}
\label{sec:3}

\subsection{Lagrangian capacity}

Here, we define the Lagrangian capacity (\cref{def:lagrangian capacity}) and state its properties (\cref{prop:properties of cL}). One of the main goals of this paper is to study whether the Lagrangian capacity can be computed in some cases, for example for toric domains. In the end of the section, we state some easy inequalities concerning the Lagrangian capacity (\cref{lem:c square leq c lag,lem:c square geq delta}), known computations (\cref{prp:cl of ball,prp:cl of cylinder}) and finally the main conjecture of this paper (\cref{conj:the conjecture}), which is inspired by all the previous results. The Lagrangian capacity is defined in terms of the minimal area of Lagrangian submanifolds, which we now define.

\begin{definition}
    Let $(X,\omega)$ be a symplectic manifold. If $L$ is a Lagrangian submanifold of $X$, then we define the \textbf{minimal symplectic area of} $L$, denoted $A_{\mathrm{min}}(L)$, by%
    \begin{IEEEeqnarray*}{c+x*}
        A_{\mathrm{min}}(L) \coloneqq \inf \{ \omega(\sigma) \mid \sigma \in \pi_2(X,L), \, \omega(\sigma) > 0 \}.
    \end{IEEEeqnarray*}
\end{definition}

\begin{lemma}
    \phantomsection\label{lem:a min with exact symplectic manifold}
    Let $(X,\lambda)$ be an exact symplectic manifold and $L \subset X$ be a Lagrangian submanifold. If $\pi_1(X) = 0$, then
    \begin{IEEEeqnarray*}{c+x*}
        A _{\mathrm{min}}(L) = \inf \left\{ \lambda(\rho) \ | \ \rho \in \pi_1(L), \ \lambda(\rho) > 0 \right\}.
    \end{IEEEeqnarray*}
\end{lemma}
\begin{proof}
    The diagram
    \begin{IEEEeqnarray*}{c+x*}
        \begin{tikzcd}
            \pi_2(X,L) \ar[dr, swap, "\omega"] \ar[r, two heads,"\del"] & \pi_1(L) \ar[d, "\lambda"] \ar[r, "0"] & \pi_1(X) \\
            & \R
        \end{tikzcd}
    \end{IEEEeqnarray*}
    commutes, where $\del([\sigma]) = [\sigma|_{S^1}]$, and the top row is exact.
\end{proof}

\begin{definition}[{\cite[Section 1.2]{cieliebakPuncturedHolomorphicCurves2018}}]
    \phantomsection\label{def:lagrangian capacity}
    Let $(X,\omega)$ be a symplectic manifold. We define the \textbf{Lagrangian capacity} of $(X,\omega)$, denoted $c_L(X,\omega)$, by%
    \begin{IEEEeqnarray*}{c}
        c_L(X,\omega) \coloneqq \sup \{ A_{\mathrm{min}}(L) \mid L \subset X \text{ is an embedded Lagrangian torus}\}.% \in [0,\infty].
    \end{IEEEeqnarray*}
\end{definition}

\begin{proposition}[{\cite[Section 1.2]{cieliebakPuncturedHolomorphicCurves2018}}]
    \label{prop:properties of cL}
    The Lagrangian capacity $c_L$ satisfies:
    \begin{description}
        \item[(Monotonicity)] If $(X,\omega_X) \longrightarrow (Y,\omega_Y)$ is a symplectic embedding with $\pi_2(Y,\iota(X)) = 0$, then $c_L(X,\omega_X) \leq c_L(Y,\omega_Y)$.
        \item[(Conformality)] If $\alpha \neq 0$, then $c_L(X,\alpha \omega) = |\alpha|  c_L(X,\omega)$.
    \end{description}
\end{proposition}

We now wish to show that if $X_{\Omega}$ is a convex or concave toric domain, then $c_L(X_{\Omega}) \geq \delta_{\Omega} \coloneqq \sup \{ a \mid (a,\ldots,a) \in \Omega \}$. For this, we consider the following symplectic capacity.

\begin{definition}[{\cite[Definition 1.17]{guttSymplecticCapacitiesPositive2018}}]
    If $(X,\omega)$ is a symplectic manifold, its \textbf{cube capacity} is given by
    \begin{IEEEeqnarray*}{c+x*}
        c_P(X,\omega) \coloneqq \sup \{ a \mid \text{there exists a symplectic embedding } P(a) \longrightarrow X \}.
    \end{IEEEeqnarray*}
\end{definition}

\begin{lemma}
    \label{lem:c square leq c lag}
    If $X$ is a star-shaped domain, then $c_L(X) \geq c_P(X)$.
\end{lemma}
\begin{proof}
    Let $\iota \colon P(a) \longrightarrow X$ be a symplectic embedding, for some $a > 0$. We want to show that $c_L(X) \geq a$. Define $T = \mu^{-1}(a,\ldots,a) \subset \partial P(a)$ and $L = \iota(T)$. Then,
    \begin{IEEEeqnarray*}{rCls+x*}
        c_L(X)
        & \geq & A_{\mathrm{min}}(L) & \quad [\text{by definition of $c_L$}] \\
        & =    & A_{\mathrm{min}}(T) & \quad [\text{since $X$ is star-shaped}] \\
        & =    & a                   & \quad [\text{by \cref{lem:a min with exact symplectic manifold}}]. & \qedhere
    \end{IEEEeqnarray*}
\end{proof}

\begin{lemma}
    \label{lem:c square geq delta}
    If $X_{\Omega}$ is a convex or concave toric domain, then $c_P(X_{\Omega}) \geq \delta_\Omega$.
\end{lemma}
\begin{proof}
    Since $X_{\Omega}$ is a convex or concave toric domain, we have that $P(\delta_\Omega) \subset X_{\Omega}$. The result follows by definition of $c_P$.
\end{proof}

Actually, Gutt--Hutchings show that $c_P(X_{\Omega}) = \delta_\Omega$ for any convex or concave toric domain $X_{\Omega}$ (\cite[Theorem 1.18]{guttSymplecticCapacitiesPositive2018}). However, for our purposes we will only need the inequality in \cref{lem:c square geq delta}. We now consider the results by Cieliebak--Mohnke for the Lagrangian capacity of the ball and the cylinder.

\begin{proposition}[{\cite[Corollary 1.3]{cieliebakPuncturedHolomorphicCurves2018}}]
    \phantomsection\label{prp:cl of ball}
    The Lagrangian capacity of the ball is
    \begin{IEEEeqnarray*}{c+x*}
        c_L(B(1)) = \frac{1}{n}.
    \end{IEEEeqnarray*}
\end{proposition}

\begin{proposition}[{\cite[p.~215-216]{cieliebakPuncturedHolomorphicCurves2018}}]
    \label{prp:cl of cylinder}
    The Lagrangian capacity of the cylinder is
    \begin{IEEEeqnarray*}{c+x*}
        c_L(Z(1)) = 1.
    \end{IEEEeqnarray*}
\end{proposition}

By \cref{lem:c square leq c lag,lem:c square geq delta}, if $X_{\Omega}$ is a convex or concave toric domain then $c_L(X_\Omega) \geq \delta_\Omega$. But as we have seen in \cref{prp:cl of ball,prp:cl of cylinder}, if $X_\Omega$ is the ball or the cylinder then $c_L(X_\Omega) = \delta_\Omega$. This motivates \cref{conj:cl of ellipsoid} below for the Lagrangian capacity of an ellipsoid, and more generally \cref{conj:the conjecture} below for the Lagrangian capacity of any convex or concave toric domain.

\begin{conjecture}[{\cite[Conjecture 1.5]{cieliebakPuncturedHolomorphicCurves2018}}]
    \label{conj:cl of ellipsoid}
    The Lagrangian capacity of the ellipsoid is%
    \begin{IEEEeqnarray*}{c+x*}
        c_L(E(a_1,\ldots,a_n)) = \p{}{2}{\frac{1}{a_1} + \cdots + \frac{1}{a_n}}^{-1}.
    \end{IEEEeqnarray*}
\end{conjecture}

\begin{conjecture}
    \label{conj:the conjecture}
    If $X_{\Omega}$ is a convex or concave toric domain then
    \begin{IEEEeqnarray*}{c+x*}
        c_L(X_{\Omega}) = \delta_\Omega.
    \end{IEEEeqnarray*}
\end{conjecture}

In \cref{lem:computation of cl,thm:my main theorem} we present our results concerning \cref{conj:the conjecture}.

\subsection{Gutt--Hutchings capacities}
\label{sec:equivariant capacities}

% ------------------
% Intro and diagrams
% ------------------

In this subsection we will define the Gutt--Hutchings capacities (\cref{def:gutt hutchings capacities}) and the $S^1$-equivariant symplectic homology capacities (\cref{def:s1esh capacities}), and list their properties (\cref{thm:properties of gutt-hutchings capacities,prp:properties of s1esh capacities} respectively). We will also compare the two capacities (\cref{thm:ghc and s1eshc}). The definition of these capacities relies on $S^1$-equivariant symplectic homology. In the commutative diagram below, we display the modules and maps which will play a role in this subsection, for a nondegenerate Liouville domain $X$.

\begin{IEEEeqnarray}{c+x*}
    \phantomsection\label{eq:diagram for s1esh capacities}
    \begin{tikzcd}
        SH^{S^1,(\varepsilon,a]}_{}(X) \ar[r, "\delta^a_0"] \ar[d, swap, "\iota^a"] & SH^{S^1,\varepsilon}_{}(X) \ar[d, two heads, hook, "\alpha"] \ar[r, "\iota^{a,\varepsilon}"] & SH^{S^1,a}_{}(X) \\
        SH^{S^1,+}_{}(X) \ar[ur, "\delta_0"] \ar[r, swap, "\delta"]                 & H_\bullet(BS^1;\Q) \otimes H_\bullet(X, \partial X;\Q)
    \end{tikzcd}
\end{IEEEeqnarray}

We point out that every vertex in the above diagram has a $U$ map and every map in the diagram commutes with this $U$ map. Specifically, all the $S^1$-equivariant symplectic homologies have the $U$ map, and $H_\bullet(BS^1;\Q) \otimes H_\bullet(X, \partial X;\Q) \cong \Q[u] \otimes H_\bullet(X, \partial X;\Q)$ has the map $U \coloneqq u^{-1} \otimes \id$. We will also make use of a version of Diagram \eqref{eq:diagram for s1esh capacities} in the case where $X$ is star-shaped, namely Diagram \eqref{eq:diagram for s1esh capacities case ss} below. In this case, the modules in the diagram admit gradings and every map is considered to be a map in a specific degree. By \cite[Proposition 3.1]{guttSymplecticCapacitiesPositive2018}, $\delta$ and $\delta_0$ are isomorphisms.

\begin{IEEEeqnarray}{c+x*}
    \phantomsection\label{eq:diagram for s1esh capacities case ss}
    \begin{tikzcd}
        SH^{S^1,(\varepsilon,a]}_{n - 1 + 2k}(X) \ar[r, "\delta^a_0"] \ar[d, swap, "\iota^a"]                   & SH^{S^1,\varepsilon}_{n - 2 + 2k}(X) \ar[d, two heads, hook, "\alpha"] \ar[r, "\iota^{a,\varepsilon}"] & SH^{S^1,a}_{n - 2 + 2k}(X) \\
        SH^{S^1,+}_{n - 1 + 2k}(X) \ar[ur, two heads, hook, "\delta_0"] \ar[r, swap, two heads, hook, "\delta"] & H_{2k-2}(BS^1;\Q) \otimes H_{2n}(X, \partial X;\Q)
    \end{tikzcd}
\end{IEEEeqnarray}

% -------------
% gh capacities
% -------------

% gh capacities
\begin{definition}[{\cite[Definition 4.1]{guttSymplecticCapacitiesPositive2018}}]
    \label{def:gutt hutchings capacities}
    If $k \in \Z_{\geq 1}$ and $(X,\lambda)$ is a nondegenerate Liouville domain, the \textbf{Gutt--Hutchings capacities} of $X$, denoted $\cgh{k}(X)$, are defined as follows. Consider the map
    \begin{IEEEeqnarray*}{c+x*}
        \delta \circ U^{k-1} \circ \iota^a \colon SH^{S^1,(\varepsilon,a]}_{}(X) \longrightarrow H_\bullet(BS^1;\Q) \otimes H_\bullet(X, \partial X;\Q)
    \end{IEEEeqnarray*}
    from Diagram \eqref{eq:diagram for s1esh capacities}. Then, we define
    \begin{IEEEeqnarray*}{c+x*}
        \cgh{k}(X) \coloneqq \inf \{ a > 0 \mid [\mathrm{pt}] \otimes [X] \in \img (\delta \circ U^{k-1} \circ \iota^a) \}.
    \end{IEEEeqnarray*}
\end{definition}

\begin{remark}
    \label{lem:can prove ineqs for ndg}
    In this paper, we consider symplectic capacities $\cgh{k}$ (see \cref{def:gutt hutchings capacities}), $\csh{k}$ (see \cref{def:s1esh capacities}), $\tilde{\mathfrak{g}}_k$ (see \cref{def:g tilde}) and $\mathfrak{g}_k$ (see \cref{def:capacities glk}). All these capacities are defined for nondegenerate Liouville domains, but their definition can be extended to Liouville domains which are not necessarily nondegenerate as in \cite[Section 4.2]{guttSymplecticCapacitiesPositive2018}. In addition, if we wish to prove inequalities involving the capacities above, it will be enough to prove these inequalities for Liouville domains which are nondegenerate.
\end{remark}

% properties of gh capacities
\begin{theorem}[{\cite[Theorem 1.24]{guttSymplecticCapacitiesPositive2018}}]
    \label{thm:properties of gutt-hutchings capacities}
    The functions $\cgh{k}$ of Liouville domains satisfy the following axioms, for all equidimensional Liouville domains $(X,\lambda_X)$ and $(Y,\lambda_Y)$:% of the same dimension:
    \begin{description}
        \item[(Monotonicity)] If $X \longrightarrow Y$ is a generalized Liouville embedding then $\cgh{k}(X) \leq \cgh{k}(Y)$.
        \item[(Conformality)] If $\alpha > 0$ then $\cgh{k}(X, \alpha \lambda_X) = \alpha  \cgh{k}(X, \lambda_X)$.
        \item[(Nondecreasing)] $\cgh{1}(X) \leq \cgh{2}(X) \leq \cdots \leq +\infty$.
        \item[(Reeb orbits)] If $\cgh{k}(X) < + \infty$, then $\cgh{k}(X) = \mathcal{A}(\gamma)$ for some Reeb orbit $\gamma$ which is contractible in $X$.
    \end{description}
\end{theorem}

The following lemma provides an alternative definition of $\cgh{k}$, in the spirit of \cite{floerApplicationsSymplecticHomology1994}.

% alternative definition of cgh
\begin{lemma}
    \label{def:ck alternative}
    Let $(X,\lambda)$ be a nondegenerate Liouville domain such that $\pi_1(X) = 0$ and $c_1(TX)|_{\pi_2(X)} = 0$. Let $E \subset \C^n$ be a nondegenerate star-shaped domain and suppose that $\phi \colon E \longrightarrow X$ is a symplectic embedding. Consider the map
    \begin{IEEEeqnarray*}{c+x*}
        \begin{tikzcd}
            SH^{S^1,(\varepsilon,a]}_{n - 1 + 2k}(X) \ar[r, "\iota^a"] & SH^{S^1,+}_{n - 1 + 2k}(X) \ar[r, "\phi_!"] & SH^{S^1,+}_{n - 1 + 2k}(E)
        \end{tikzcd}
    \end{IEEEeqnarray*}
    Then, $\cgh{k}(X) = \inf \{ a > 0 \mid \phi_! \circ \iota^a \text{ is nonzero} \}$.
\end{lemma}
\begin{proof}
    For every $a \in \R$ consider the following commutative diagram:
    \begin{IEEEeqnarray*}{c+x*}
        \begin{tikzcd}
            SH^{S^1,(\varepsilon, a]}_{n - 1 + 2k}(X) \ar[r, "\iota^a_X"] \ar[d, swap, "\phi_!^a"] & SH^{S^1,+}_{n - 1 + 2k}(X) \ar[r, "U ^{k-1}_X"] \ar[d, "\phi_!"]       & SH^{S^1,+}_{n+1}(X) \ar[r, "\delta_X"] \ar[d, "\phi_!"]       & H_0(BS^1) \otimes H_{2n}(X,\del X) \ar[d, hook, two heads, "\id \otimes \rho"] \\
            SH^{S^1,(\varepsilon, a]}_{n - 1 + 2k}(E) \ar[r, swap, "\iota^a_E"]                    & SH^{S^1,+}_{n - 1 + 2k}(E) \ar[r, swap, hook, two heads, "U ^{k-1}_E"] & SH^{S^1,+}_{n+1}(E) \ar[r, swap, hook, two heads, "\delta_E"] & H_0(BS^1) \otimes H_{2n}(E,\del E)
        \end{tikzcd}
    \end{IEEEeqnarray*}
    By \cite[Proposition 3.1]{guttSymplecticCapacitiesPositive2018} and since $E$ is star-shaped, the maps $U_E$ and $\delta_E$ are isomorphisms. Since $\rho([X]) = [E]$, the map $\rho$ is an isomorphism. By definition, $\cgh{k}$ is the infimum over $a$ such that the top arrow is surjective. This condition is equivalent to $\phi_! \circ \iota^a_X$ being nonzero.
\end{proof}
% Note: In this diagram, \delta map of X really means, "do the \delta map, and then project to the correct degree". I don't think we know that the \delta map lives in those degrees. I think the proof still works.

The following computation will be useful to us in the proofs of \cref{lem:computation of cl,thm:my main theorem}.

% cgh of cyl
\begin{lemma}[{\cite[Lemma 1.19]{guttSymplecticCapacitiesPositive2018}}]
    \label{lem:cgh of nondisjoint union of cylinders}
    $\cgh{k}(N(a)) = a  (k + n - 1)$.
\end{lemma}

% ---------------------
% capacities from s1esh
% ---------------------

We now consider other capacities which can be defined using $S^1$-equivariant symplectic homology.

% capacities from s1esh
\begin{definition}[{\cite[Section 2.5]{irieSymplecticHomologyFiberwise2021}}]
    \phantomsection\label{def:s1esh capacities}
    If $k \in \Z_{\geq 1}$ and $(X,\lambda)$ is a nondegenerate Liouville domain, the \textbf{$S^1$-equivariant symplectic homology capacities} of $X$, denoted $\csh{k}(X)$, are defined as follows. Consider the map
    \begin{IEEEeqnarray*}{c+x*}
        \iota^{a,\varepsilon} \circ \alpha^{-1} \colon H_\bullet(BS^1;\Q) \otimes H_\bullet(X, \partial X;\Q) \longrightarrow SH^{S^1,a}_{}(X)
    \end{IEEEeqnarray*}
    from Diagram \eqref{eq:diagram for s1esh capacities}. Then, we define
    \begin{IEEEeqnarray*}{c+x*}
        \csh{k}(X) \coloneqq \inf \{ a > 0 \mid \iota^{a,\varepsilon} \circ \alpha^{-1}([\C P^{k-1}] \otimes [X]) = 0 \}.
    \end{IEEEeqnarray*}
\end{definition}

We now state the properties that the capacities $\csh{k}$ satisfy. For the sake of completeness, we include proofs as well.

% properties of capacities from s1esh
\begin{theorem}
    \label{prp:properties of s1esh capacities}
    The functions $\csh{k}$ of Liouville domains satisfy the following axioms, for all Liouville domains $(X,\lambda_X)$ and $(Y,\lambda_Y)$ of the same dimension:
    \begin{description}
        \item[(Monotonicity)] If $X \longrightarrow Y$ is a generalized Liouville embedding then $\csh{k}(X) \leq \csh{k}(Y)$.
        \item[(Conformality)] If $\mu > 0$ then $\csh{k}(X, \mu \lambda_X) = \mu  \csh{k}(X, \lambda_X)$.
        \item[(Nondecreasing)] $\csh{1}(X) \leq \csh{2}(X) \leq \cdots \leq +\infty$.
    \end{description}
\end{theorem}
\begin{proof}
    We prove monotonicity. Consider the following commutative diagram:
    \begin{IEEEeqnarray}{c+x*}
        \phantomsection\label{eq:s1eshc diagram}
        \begin{tikzcd}
            H_\bullet(BS^1;\Q) \otimes H_\bullet(Y, \partial Y;\Q) \ar[d, swap, "\id \otimes \rho"] & SH^{S^1,\varepsilon}_{}(Y) \ar[l, swap, hook', two heads, "\alpha_Y"] \ar[r, "\iota^{a, \varepsilon}_Y"] \ar[d, "\phi_!^\varepsilon"] & SH^{S^1,a}_{}(Y) \ar[d, "\phi^a_!"] \\
            H_\bullet(BS^1;\Q) \otimes H_\bullet(X, \partial X;\Q)                                  & SH^{S^1,\varepsilon}_{}(X) \ar[l, hook', two heads, "\alpha_X"] \ar[r, swap, "\iota^{a, \varepsilon}_X"]                              & SH^{S^1,a}_{}(X)
        \end{tikzcd}
    \end{IEEEeqnarray}
    If $\iota_Y^{a,\varepsilon} \circ \alpha_Y^{-1}([\C P^{k-1}] \otimes [Y]) = 0$, then
    \begin{IEEEeqnarray*}{rCls+x*}
        \IEEEeqnarraymulticol{3}{l}{\iota_X^{a,\varepsilon} \circ \alpha_X^{-1}([\C P^{k-1}] \otimes [X])} \\
        \quad & = & \iota_X^{a,\varepsilon} \circ \alpha_X^{-1} \circ (\id \otimes \rho)([\C P^{k-1}] \otimes [Y]) & \quad [\text{since $\rho([Y]) = [X]$}] \\
              & = & \phi_! \circ \iota_Y^{a,\varepsilon} \circ \alpha_{Y}^{-1} ([\C P^{k-1}] \otimes [Y])          & \quad [\text{by Diagram \eqref{eq:s1eshc diagram}}] \\
              & = & 0                                                                                              & \quad [\text{by assumption}].
    \end{IEEEeqnarray*}

    To prove conformality, choose $\varepsilon > 0$ such that $\varepsilon, \mu \varepsilon < \min \operatorname{Spec}(\partial X, \lambda|_{\partial X})$. Since the diagram
    \begin{IEEEeqnarray*}{c+x*}
        \begin{tikzcd}
            H_\bullet(BS^1;\Q) \otimes H_\bullet(X, \partial X;\Q) \ar[d, equals] & SH^{S^1,\varepsilon}_{}(X, \lambda) \ar[d, equals] \ar[l, swap, hook', two heads, "\alpha_{\lambda}"] \ar[r, "\iota^{a, \varepsilon}_\lambda"]            & SH^{S^1,a}_{}(X, \lambda) \ar[d, equals] \\
            H_\bullet(BS^1;\Q) \otimes H_\bullet(X, \partial X;\Q)                & SH^{S^1,\mu \varepsilon}_{}(X, \mu \lambda) \ar[l, hook', two heads, "\alpha_{\mu \lambda}"] \ar[r, swap, "\iota^{\mu a, \mu \varepsilon}_{\mu \lambda}"] & SH^{S^1,\mu a}_{}(X, \mu \lambda)
        \end{tikzcd}
    \end{IEEEeqnarray*}
    commutes (by \cite[Proposition 3.1]{guttSymplecticCapacitiesPositive2018}), the result follows.

    To prove the nondecreasing property, note that if $\iota^{a,\varepsilon} \circ \alpha^{-1}([\C P ^{k}] \otimes [X]) = 0$, then
    \begin{IEEEeqnarray*}{rCls+x*}
        \IEEEeqnarraymulticol{3}{l}{\iota^{a,\varepsilon} \circ \alpha^{-1}([\C P ^{k-1}] \otimes [X])}\\
        \quad & = & \iota^{a,\varepsilon} \circ \alpha^{-1} \circ U ([\C P ^{k}] \otimes [X])     & \quad [\text{since $U([\C P^k] \otimes [X]) = [\C P^{k-1}] \otimes [X]$}] \\
              & = & U^{a} \circ \iota^{a,\varepsilon} \circ \alpha^{-1} ([\C P ^{k}] \otimes [X]) & \quad [\text{since $\iota^{a,\varepsilon}$ and $\alpha$ commute with $U$}] \\
              & = & 0                                                                             & \quad [\text{by assumption}].                                        & \qedhere
    \end{IEEEeqnarray*}
\end{proof}

% comparison between gh and s1esh
\begin{theorem}
    \label{thm:ghc and s1eshc}
    If $(X, \lambda)$ is a Liouville domain, then
    \begin{enumerate}
        \item \label{thm:comparison cgh csh 1} $\cgh{k}(X) \leq \csh{k}(X)$;
        \item \label{thm:comparison cgh csh 2} $\cgh{k}(X) = \csh{k}(X)$ provided that $X$ is star-shaped.
    \end{enumerate}
\end{theorem}
\begin{proof}
    By \cref{lem:can prove ineqs for ndg}, we may assume that $X$ is nondegenerate. Since
    \begin{IEEEeqnarray*}{rCls+x*}
        \IEEEeqnarraymulticol{3}{l}{\iota^{a,\varepsilon} \circ \alpha^{-1}([\C P ^{k-1}] \otimes [X]) = 0}\\
        \quad & \Longleftrightarrow & \alpha^{-1}([\C P ^{k-1}] \otimes [X]) \in \ker \iota^{a,\varepsilon}   & \quad [\text{by definition of kernel}] \\
        \quad & \Longleftrightarrow & \alpha^{-1}([\C P ^{k-1}] \otimes [X]) \in \img \delta^a_0              & \quad [\text{since the top row of \eqref{eq:diagram for s1esh capacities} is exact}] \\
        \quad & \Longleftrightarrow & [\C P ^{k-1}] \otimes [X] \in \img (\alpha \circ \delta^a_0)            & \quad [\text{by definition of image}] \\
        \quad & \Longleftrightarrow & [\C P ^{k-1}] \otimes [X] \in \img (\delta \circ \iota^a)               & \quad [\text{since Diagram \eqref{eq:diagram for s1esh capacities} commutes}] \\
        \quad & \Longrightarrow     & [\mathrm{pt}] \otimes [X] \in \img (U^{k-1} \circ \delta \circ \iota^a) & \quad [\text{since $U^{k-1}([\C P ^{k-1}] \otimes [X]) = [\mathrm{pt}] \otimes [X]$}] \\
        \quad & \Longleftrightarrow & [\mathrm{pt}] \otimes [X] \in \img (\delta \circ U^{k-1} \circ \iota^a) & \quad [\text{since $\delta$ and $U$ commute}],
    \end{IEEEeqnarray*}
    we have that $\cgh{k}(X) \leq \csh{k}(X)$. If $X$ is a star-shaped domain, we can view the maps of the computation above as being the maps in Diagram \eqref{eq:diagram for s1esh capacities case ss}, i.e. they are defined in a specific degree. In this case, $U^{k-1} \colon H_{2k-2}(BS^1) \otimes H_{2n}(X, \partial X) \longrightarrow H_{0}(BS^1) \otimes H_{2n}(X, \partial X)$ is an isomorphism, and therefore the implication in the previous computation is actually an equivalence.
\end{proof}

\begin{remark}
    The capacities $\cgh{k}$ and $\csh{k}$ are defined in terms of a certain homology class being in the kernel or in the image of a map with domain or target the $S^1$-equivariant symplectic homology. Other authors have constructed capacities in an analogous manner, for example Viterbo \cite[Definition 2.1]{viterboSymplecticTopologyGeometry1992} and \cite[Section 5.3]{viterboFunctorsComputationsFloer1999}, Schwarz \cite[Definition 2.6]{schwarzActionSpectrumClosed2000} and Ginzburg--Shon \cite[Section 3.1]{ginzburgFilteredSymplecticHomology2018}.
\end{remark}

\subsection{McDuff--Siegel capacities}

We now define the McDuff--Siegel capacities. These will assist us in our goal of proving \cref{conj:the conjecture} (at least in particular cases) because they can be compared with the Lagrangian capacity (\cref{thm:lagrangian vs g tilde}) and with the Gutt--Hutchings capacities (\cref{prp:g tilde and cgh}).

\begin{definition}[{\cite[Definition 3.3.1]{mcduffSymplecticCapacitiesUnperturbed2022}}]
    \label{def:g tilde}
    Let $(X,\lambda)$ be a nondegenerate Liouville domain. For $k \in \Z_{\geq 1}$, we define the \textbf{McDuff--Siegel capacities} of $X$, denoted $\tilde{\mathfrak{g}}_k(X)$, as follows. Choose $x \in \itr X$ and $D$ a symplectic divisor at $x$. Then,
    \begin{IEEEeqnarray*}{c+x*}
        \tilde{\mathfrak{g}}_k(X) \coloneqq \sup_{J \in \mathcal{J}(X,D)} \mathop{\inf\vphantom{\sup}}_{\gamma} \mathcal{A}(\gamma),
    \end{IEEEeqnarray*}
    where the infimum is over Reeb orbits $\gamma$ such that $\mathcal{M}_X^J(\gamma)\p{<}{}{\mathcal{T}^{(k)}x} \neq \varnothing$.
\end{definition}

\begin{remark}
    Actually, the McDuff--Siegel capacities (given as in \cite[Definition 3.3.1]{mcduffSymplecticCapacitiesUnperturbed2022}) are a family of symplectic capacities $\tilde{\mathfrak{g}}^{\leq \ell}_k$ indexed by $\ell, k \in \Z_{\geq 1}$. The capacity $\tilde{\mathfrak{g}}_k$ from \cref{def:g tilde} is the capacity $\tilde{\mathfrak{g}}^{\leq 1}_k$ from \cite[Definition 3.3.1]{mcduffSymplecticCapacitiesUnperturbed2022}. We point out that in \cite{mcduffSymplecticCapacitiesUnperturbed2022}, the notation $\tilde{\mathfrak{g}}_k$ is used for the case $\ell = \infty$, while we use this notation for the case $\ell = 1$. A similar discussion holds for the higher symplectic capacities $\mathfrak{g}_k$ of \cref{def:capacities glk}.
\end{remark}

\begin{theorem}[{\cite[Theorem 3.3.2]{mcduffSymplecticCapacitiesUnperturbed2022}}]
    \label{thm:properties of g tilde}
    The functions $\tilde{\mathfrak{g}}_k$ are independent of the choices of $x$ and $D$ and satisfy the following properties, for all nondegenerate Liouville domains $(X,\lambda_X)$ and $(Y,\lambda_Y)$ of the same dimension:
    \begin{description}
        \item[(Monotonicity)] If $X \longrightarrow Y$ is a generalized Liouville embedding then $\tilde{\mathfrak{g}}_k(X) \leq \tilde{\mathfrak{g}}_k(Y)$.
        \item[(Conformality)] If $\alpha > 0$ then $\tilde{\mathfrak{g}}_k(X, \alpha \lambda_X) = \alpha  \tilde{\mathfrak{g}}_k(X, \lambda_X)$.
        \item[(Nondecreasing)] $\tilde{\mathfrak{g}}_1(X) \leq \tilde{\mathfrak{g}}_{2}(X) \leq \cdots \leq +\infty$.
    \end{description}
\end{theorem}

\begin{proposition}[{\cite[Proposition 5.6.1]{mcduffSymplecticCapacitiesUnperturbed2022}}]
    \label{prp:g tilde and cgh}
    If $X_{\Omega}$ is a $4$-dimensional convex toric domain then
    \begin{IEEEeqnarray*}{c+x*}
        \tilde{\mathfrak{g}}_k(X_\Omega) = \cgh{k}(X_\Omega).
    \end{IEEEeqnarray*}
\end{proposition}

Finally, we state two stabilization results which we will use in \cref{sec:augmentation map of an ellipsoid}. The fact that will be relevant to us is \cref{lem:stabilization 2} \ref{lem:stabilization 2 1}, which we will use to argue that the moduli space of curves in an ellipsoid satisfying a point constraint is independent of the dimension of the ellipsoid.

\begin{lemma}[{\cite[Lemma 3.6.2]{mcduffSymplecticCapacitiesUnperturbed2022}}]
    \label{lem:stabilization 1}
    Let $(X, \lambda)$ be a Liouville domain. For any $c, \varepsilon \in \R_{> 0}$, there is a subdomain with smooth boundary $\tilde{X} \subset X \times B^2(c)$ such that:
    \begin{enumerate}
        \item The Liouville vector field $Z_{\tilde{X}} = Z_{X} + Z_{B^2(c)}$ is outwardly transverse along $\partial \tilde{X}$.
        \item $X \times \{0\} \subset \tilde{X}$ and the Reeb vector field of $\partial \tilde{X}$ is tangent to $\partial X \times \{0\}$.
        \item Any Reeb orbit of the contact form $(\lambda + \lambda_0)|_{\partial \tilde{X}}$ (where $\lambda_0 = 1/2 (x \edv y - y \edv x)$) with action less than $c - \varepsilon$ is entirely contained in $\partial X \times \{0\}$ and has normal Conley--Zehnder index equal to $1$.
    \end{enumerate}
\end{lemma}

\begin{lemma}[{\cite[Lemma 3.6.3]{mcduffSymplecticCapacitiesUnperturbed2022}}]
    \label{lem:stabilization 2}
    Let $X$ be a Liouville domain, and let $\tilde{X}$ be a smoothing of $X \times B^2(c)$ as in \cref{lem:stabilization 1}.
    \begin{enumerate}
        \item \label{lem:stabilization 2 1} Let $J \in \mathcal{J}(\tilde{X})$ be a cylindrical almost complex structure on the completion of $\tilde{X}$ for which $\hat{X} \times \{0\}$ is $J$-holomorphic. Let $C$ be an asymptotically cylindrical $J$-holomorphic curve in $\hat{X}$, all of whose asymptotic Reeb orbits are nondegenerate and lie in $\partial X \times \{0\}$ with normal Conley--Zehnder index $1$. Then $C$ is either disjoint from the slice $\hat{X} \times \{0\}$ or entirely contained in it.
        \item Let $J \in \mathcal{J}(\partial \tilde{X})$ be a cylindrical almost complex structure on the symplectization of $\partial \tilde{X}$ for which $\R \times \partial X \times \{0\}$ is $J$-holomorphic. Let $C$ be an asymptotically cylindrical $J$-holomorphic curve in $\R \times \partial \tilde{X}$, all of whose asymptotic Reeb orbits are nondegenerate and lie in $\partial X \times \{0\}$ with normal Conley--Zehnder index $1$. Then $C$ is either disjoint from the slice $\R \times \partial X \times \{0\}$ or entirely contained in it. Moreover, only the latter is possible if $C$ has at least one negative puncture.
    \end{enumerate}
\end{lemma}

\subsection{Computations without contact homology}

% -------------------------------------------------------------------
% Main theorem. Comparison of Lagrangian and McDuff--Siegel capacities
% -------------------------------------------------------------------

We now state and prove one of our main theorems, which is going to be a key step in proving that $c_L(X_{\Omega}) = \delta_{\Omega}$. The proof uses techniques similar to those used in the proof of \cite[Theorem 1.1]{cieliebakPuncturedHolomorphicCurves2018}.

\begin{theorem}
    \label{thm:lagrangian vs g tilde}
    If $(X, \lambda)$ is a Liouville domain then
    \begin{IEEEeqnarray*}{c+x*}
        c_L(X) \leq \inf_k^{} \frac{\tilde{\mathfrak{g}}_k(X)}{k}.
    \end{IEEEeqnarray*}
\end{theorem}
\begin{proof}
    By \cref{lem:can prove ineqs for ndg}, we may assume that $X$ is nondegenerate. Let $k \in \Z_{\geq 1}$ and $L \subset \itr X$ be an embedded Lagrangian torus. We wish to show that for every $\varepsilon > 0$ there exists $\sigma \in \pi_2(X,L)$ such that $0 < \omega(\sigma) \leq \tilde{\mathfrak{g}}_k(X) / k + \varepsilon$. Define
    \begin{IEEEeqnarray*}{rCls+x*}
        a      & \coloneqq & \tilde{\mathfrak{g}}_k(X), \\
        K_1    & \coloneqq & \ln(2), \\
        K_2    & \coloneqq & \ln(1 + a / \varepsilon k), \\
        K      & \coloneqq & \max \{K_1, K_2\}, \\
        \delta & \coloneqq & e^{-K}, \\
        \ell_0 & \coloneqq & a / \delta.
    \end{IEEEeqnarray*}

    By \cref{lem:geodesics lemma CM abs} and the Lagrangian neighbourhood theorem, there exists a Riemannian metric $g$ on $L$ and a symplectic embedding $\phi \colon D^*L \longrightarrow X$ such that $\phi(D^*L) \subset \itr X$, $\phi|_L = \id_L$ and such that if $\gamma$ is a closed geodesic in $L$ with length $\ell(\gamma) \leq \ell_0$ then $\gamma$ is noncontractible, nondegenerate and satisfies $0 \leq \morse(\gamma) \leq n - 1$.

    Let $D^*_{\delta} L$ be the codisk bundle of radius $\delta$. Notice that $\delta$ has been chosen in such a way that the symplectic embedding $\phi \colon D^* L \longrightarrow X$ can be seen as an embedding like that of \cref{lem:energy wrt different forms}. Define symplectic cobordisms%
    \begin{IEEEeqnarray*}{rCl}
        (X^+, \omega^+) & \coloneqq & (X \setminus \phi(D^*_{\delta} L), \omega), \\
        (X^-, \omega^-) & \coloneqq & (D^*_{\delta} L, \edv \lambda_{T^* L}),
    \end{IEEEeqnarray*}
    which have the common contact boundary
    \begin{IEEEeqnarray*}{c+x*}
        (M, \alpha) \coloneqq (S^*_{\delta} L, \lambda_{T^* L}).
    \end{IEEEeqnarray*}
    Here, it is implicit that we are considering the restriction of the form $\lambda_{T^*L}$ on $T^* L$ to $D^*_{\delta} L$ or $S^*_{\delta} L$. Then, $(X,\omega) = (X^-, \omega^-) \circledcirc (X^+, \omega^+)$. Recall that there are piecewise smooth $2$-forms $\tilde{\omega} \in \Omega^2(\hat{X})$ and $\tilde{\omega}^{\pm} \in \Omega^2(\hat{X}^{\pm})$ given as in Equation \eqref{eq:form 2}. Choose $x \in \itr \phi(D^*_{\delta} L)$ and let $D \subset \phi(D^*_{\delta} L)$ be a symplectic divisor through $x$. Choose also generic almost complex structures
    \begin{IEEEeqnarray*}{rCls+x*}
        J_M & \in & \mathcal{J}(M), \\
        J^+ & \in & \mathcal{J}_{J_M}(X^+), \\
        J^- & \in & \mathcal{J}^{J_M}(X^-, D),
    \end{IEEEeqnarray*}
    and denote by $J_{\partial X} \in \mathcal{J}(\partial X)$ the ``restriction'' of $J^+$ to $\R \times \partial X$. Let $(J_t)_{t} \subset \mathcal{J}(X, D)$ be the corresponding neck stretching family of almost complex structures. Since $a = \tilde{\mathfrak{g}}_k(X)$, for every $t$ there exists a Reeb orbit $\gamma_t$ in $\partial X = \partial^+ X^+$ and a $J_t$-holomorphic curve $u_t \in \mathcal{M}_X^{J_t}(\gamma_t)\p{<}{}{\mathcal{T}^{(k)}x}$ such that $\mathcal{A}(\gamma_t) \leq a$. Since $\partial X$ has nondegenerate Reeb orbits, there are only finitely many Reeb orbits in $\partial X$ with action less than $a$. Therefore, possibly after going to a subsequence, we may assume that $\gamma_t \eqqcolon \gamma_0$ is independent of $t$.

    The curves $u_t$ satisfy the energy bound $E_{\tilde{\omega}}(u_t) \leq a$. By the SFT compactness theorem, the sequence $(u_t)_{t}$ converges to a holomorphic building
    \begin{IEEEeqnarray*}{c+x*}
        F = (F^1, \ldots, F^{L_0-1}, F^{L_0}, F^{{L_0}+1}, \ldots, F^N) \in \overline{\mathcal{M}}_X^{(J_t)_{t}}(\gamma_0)\p{<}{}{\mathcal{T}^{(k)}x},
    \end{IEEEeqnarray*}
    where
    \begin{IEEEeqnarray*}{rCls+x*}
        (X^{\nu}, \omega^\nu, \tilde{\omega}^{\nu}, J^{\nu})
        & \coloneqq &
        \begin{cases}
            (T^* L               , \edv \lambda_{T^* L}             , \tilde{\omega}^-           , J^-)            & \text{if } \nu = 1    , \\
            (\R \times M         , \edv(e^r \alpha)                 , \edv \alpha                , J_M)            & \text{if } \nu = 2    , \ldots, {L_0} - 1, \\
            (\hat{X} \setminus L , \hat{\omega}                     , \tilde{\omega}^+           , J^+)            & \text{if } \nu = {L_0}    , \\
            (\R \times \partial X, \edv (e^r \lambda|_{\partial X}) , \edv \lambda|_{\partial X} , J_{\partial X}) & \text{if } \nu = {L_0} + 1, \ldots, N    , \\
        \end{cases} \\
        (X^*, \omega^*, \tilde{\omega}^*, J^*) & \coloneqq & \coprod_{\nu = 1}^N (X^{\nu}, \omega^\nu, \tilde{\omega}^{\nu}, J^{\nu}),
    \end{IEEEeqnarray*}
    and $F^{\nu}$ is a $J^\nu$-holomorphic curve in $X^{\nu}$ with asymptotic Reeb orbits $\Gamma^{\pm}_{\nu}$ (see \cref{fig:holomorphic building in the proof}). The holomorphic building $F$ satisfies the energy bound
    \begin{IEEEeqnarray}{c+x*}
        \phantomsection\label{eq:energy of holo building in proof}
        E_{\tilde{\omega}^*}(F) \coloneqq \sum_{\nu = 1}^{N} E_{\tilde{\omega}^{\nu}}(F^{\nu}) \leq a.
    \end{IEEEeqnarray}

    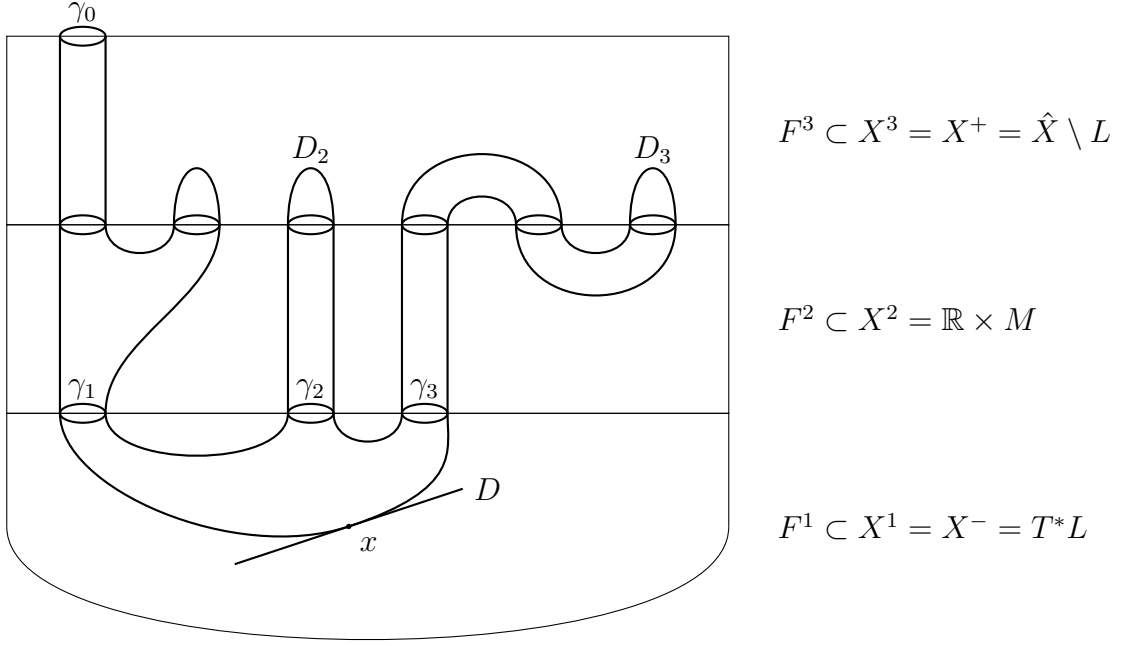
\begin{figure}[ht]
        \centering

        \begin{tikzpicture}
            [
                scale = 0.5,
                help/.style = {very thin, draw = black!50},
                curve/.style = {thick}
            ]

            \tikzmath{
                \rx = 0.6;
                \ry = 0.25;
            }

            % X^3
            \node[anchor=west] at (20, 13.5) {$F^3 \subset X^3 = X^+ = \hat{X} \setminus L$};
            \draw (0,6) rectangle (19,11);

            % X^2
            \node[anchor=west] at (20, 8.5) {$F^2 \subset X^2 = \R \times M$};
            \draw (0,11) rectangle (19,16);

            % X^1
            \node[anchor=west] at (20, 3) {$F^1 \subset X^1 = X^- = T^* L$};
            \draw (0,3) -- (0,6) -- (19,6) -- (19,3);
            \draw (0,3) .. controls (0,-1) and (19,-1) .. (19,3);

            % Reeb orbit locations
            \coordinate (G0) at ( 2,16);
            \coordinate (G1) at ( 2, 6);
            \coordinate (G2) at ( 8, 6);
            \coordinate (G3) at (11, 6);
            \coordinate (F1) at ( 2,11);
            \coordinate (F2) at ( 8,11);
            \coordinate (F3) at (11,11);
            \coordinate (F4) at ( 5,11);
            \coordinate (F5) at (14,11);
            \coordinate (F6) at (17,11);

            \coordinate (L) at (-\rx,0);
            \coordinate (R) at (+\rx,0);

            \coordinate (G0L) at ($ (G0) + (L) $);
            \coordinate (G1L) at ($ (G1) + (L) $);
            \coordinate (G2L) at ($ (G2) + (L) $);
            \coordinate (G3L) at ($ (G3) + (L) $);
            \coordinate (F1L) at ($ (F1) + (L) $);
            \coordinate (F2L) at ($ (F2) + (L) $);
            \coordinate (F3L) at ($ (F3) + (L) $);
            \coordinate (F4L) at ($ (F4) + (L) $);
            \coordinate (F5L) at ($ (F5) + (L) $);
            \coordinate (F6L) at ($ (F6) + (L) $);

            \coordinate (G0R) at ($ (G0) + (R) $);
            \coordinate (G1R) at ($ (G1) + (R) $);
            \coordinate (G2R) at ($ (G2) + (R) $);
            \coordinate (G3R) at ($ (G3) + (R) $);
            \coordinate (F1R) at ($ (F1) + (R) $);
            \coordinate (F2R) at ($ (F2) + (R) $);
            \coordinate (F3R) at ($ (F3) + (R) $);
            \coordinate (F4R) at ($ (F4) + (R) $);
            \coordinate (F5R) at ($ (F5) + (R) $);
            \coordinate (F6R) at ($ (F6) + (R) $);

            % Tangency constraint vectors
            \coordinate (P) at (9,3);
            \coordinate (D) at (3,1);

            % Reeb orbits
            \draw[curve] (G0) ellipse [x radius = \rx, y radius = \ry] node[above = 1] {$\gamma_0$};
            \draw[curve] (G1) ellipse [x radius = \rx, y radius = \ry] node[above = 1] {$\gamma_1$};
            \draw[curve] (G2) ellipse [x radius = \rx, y radius = \ry] node[above = 1] {$\gamma_2$};
            \draw[curve] (G3) ellipse [x radius = \rx, y radius = \ry] node[above = 1] {$\gamma_3$};
            \draw[curve] (F1) ellipse [x radius = \rx, y radius = \ry];
            \draw[curve] (F2) ellipse [x radius = \rx, y radius = \ry];
            \draw[curve] (F3) ellipse [x radius = \rx, y radius = \ry];
            \draw[curve] (F4) ellipse [x radius = \rx, y radius = \ry];
            \draw[curve] (F5) ellipse [x radius = \rx, y radius = \ry];
            \draw[curve] (F6) ellipse [x radius = \rx, y radius = \ry];

            % Symplectic divisor
            \fill (P) circle (2pt) node[anchor = north west] {$x$};
            \draw[curve] ($ (P) - (D) $) -- ( $ (P) + (D) $ ) node[anchor = west] {$D$};

            % Curve
            \draw[curve] (G1L) -- (G0L);
            \draw[curve] (F1R) -- (G0R);
            \draw[curve] (G2L) -- (F2L);
            \draw[curve] (G2R) -- (F2R);
            \draw[curve] (G3L) -- (F3L);
            \draw[curve] (G3R) -- (F3R);

            \draw[curve] (F4L) .. controls ($ (F4L) + (0,2) $) and ($ (F4R) + (0,2) $) .. (F4R);
            \draw[curve] (F2L) .. controls ($ (F2L) + (0,2) $) and ($ (F2R) + (0,2) $) .. (F2R);
            \draw[curve] (F6L) .. controls ($ (F6L) + (0,2) $) and ($ (F6R) + (0,2) $) .. (F6R);

            \draw[curve] (F3R) .. controls ($ (F3R) + (0,1) $) and ($ (F5L) + (0,1) $) .. (F5L);
            \draw[curve] (F5R) .. controls ($ (F5R) - (0,1) $) and ($ (F6L) - (0,1) $) .. (F6L);

            \draw[curve] (F3L) .. controls ($ (F3L) + (0,2.5) $) and ($ (F5R) + (0,2.5) $) .. (F5R);
            \draw[curve] (F5L) .. controls ($ (F5L) - (0,2.5) $) and ($ (F6R) - (0,2.5) $) .. (F6R);

            \draw[curve] (F1R) .. controls ($ (F1R) - (0,1) $) and ($ (F4L) - (0,1) $) .. (F4L);
            \draw[curve] (G1R) .. controls ($ (G1R) + (0,2) $) and ($ (F4R) - (0,2) $) .. (F4R);

            \draw[curve] (G1R) .. controls ($ (G1R) - (0,1.5) $) and ($ (G2L) - (0,1.5) $) .. (G2L);
            \draw[curve] (G2R) .. controls ($ (G2R) - (0,1) $) and ($ (G3L) - (0,1) $) .. (G3L);

            \draw[curve] (G1L) .. controls ($ (G1L) - (0,2) $) and ($ (P) - (D) $) .. (P);
            \draw[curve] (G3R) .. controls ($ (G3R) - (0,1) $) and ($ (P) + (D) $) .. (P);

            % Disks label
            \node at ($ (F2) + (0,2) $) {$D_2$};
            \node at ($ (F6) + (0,2) $) {$D_3$};
        \end{tikzpicture}

        \caption{The holomorphic building $F = (F^1, \ldots, F^N)$ in the case ${L_0} = N = p = 3$}
        \label{fig:holomorphic building in the proof}
    \end{figure}

    Moreover, by \cref{lem:no nodes}, $F$ has no nodes. Let $C$ be the component of $F$ in $X^-$ which carries the tangency constraint $\p{<}{}{\mathcal{T}^{(k)}x}$. Then, $C$ is positively asymptotic to Reeb orbits $(\gamma_1, \ldots, \gamma_p)$ of $M$. For $\mu = 1, \ldots, p$, let $C_\mu$ be the subtree emanating from $C$ at $\gamma_\mu$. For exactly one $\mu = 1, \ldots, p$, the top level of the subtree $C_\mu$ is positively asymptotic to $\gamma_0$, and we may assume without loss of generality that this is true for $\mu = 1$. By the maximum principle, $C_\mu$ has a component in $X^{L_0} = \hat{X} \setminus L$ for every $\mu = 2, \ldots, p$. Also by the maximum principle, there do not exist components of $C_\mu$ in $X^{L_0} = \hat{X} \setminus L$ which intersect $\R_{\geq 0} \times \partial X$ or components of $C_\mu$ in the top symplectization layers $X^{{L_0}+1}, \ldots, X^N$, for every $\mu = 2, \ldots, p$.

    We claim that if $\gamma$ is a Reeb orbit in $M$ which is an asymptote of $F^\nu$ for some $\nu = 2,\ldots,{L_0}-1$, then $\mathcal{A}(\gamma) \leq a$. To see this, notice that
    \begin{IEEEeqnarray*}{rCls+x*}
        a
        & \geq & E_{\tilde{\omega}^*}(F)           & \quad [\text{by Equation \eqref{eq:energy of holo building in proof}}] \\
        & \geq & E_{\tilde{\omega}^N}(F^N)         & \quad [\text{by monotonicity of $E$}] \\
        & \geq & (e^K - 1) \mathcal{A}(\Gamma^-_N) & \quad [\text{by \cref{lem:energy wrt different forms}}] \\
        & \geq & \mathcal{A}(\Gamma^-_N)           & \quad [\text{since $K \geq K_1$}] \\
        & \geq & \mathcal{A}(\Gamma^-_\nu)         & \quad [\text{by \cref{eq:energy identity}}]
    \end{IEEEeqnarray*}
    for every $\nu = 2, \ldots, {L_0}-1$. Every such $\gamma$ has a corresponding geodesic in $L$ (which by abuse of notation we denote also by $\gamma$) such that $\ell(\gamma) = \mathcal{A}(\gamma)/\delta \leq a / \delta = \ell_0$. Hence, by our choice of Riemannian metric, the geodesic $\gamma$ is noncontractible, nondegenerate and such that $\morse(\gamma) \leq n - 1$. Therefore, the Reeb orbit $\gamma$ is noncontractible, nondegenerate and such that $\conleyzehnder(\gamma) \leq n - 1$.

    We claim that if $D$ is a component of $C_\mu$ for some $\mu = 2,\ldots,p$ and $D$ is a plane, then $D$ is in $X^{L_0} = \hat{X} \setminus L$. Assume by contradiction otherwise. Notice that since $D$ is a plane, $D$ is asymptotic to a unique Reeb orbit $\gamma$ in $M = S^*_{\delta} L$ with corresponding noncontractible geodesic $\gamma$ in $L$. We will derive a contradiction by defining a filling disk for $\gamma$. If $D$ is in a symplectization layer $\R \times S^*_\delta L$, then the map $\pi \circ D$, where $\pi \colon \R \times S^*_{\delta} L \longrightarrow L$ is the projection, is a filling disk for the geodesic $\gamma$. If $D$ is in the bottom level, i.e. $X^1 = T^*L$, then the map $\pi \circ D$, where $\pi \colon T^*L \longrightarrow L$ is the projection, is also a filling disk. This proves the claim.

    So, summarizing our previous results, we know that for every $\mu = 2,\ldots,p$ there is a holomorphic plane $D_\mu$ in $X^{L_0} \setminus (\R_{\geq 0} \times \partial X) = X \setminus L$. For each plane $D_\mu$ there is a corresponding disk in $X$ with boundary on $L$, which we denote also by $D_\mu$. It is enough to show that $E_{\omega}(D_{\mu_0}) \leq a/k + \varepsilon$ for some $\mu_0 = 2,\ldots,p$. By \cref{lem:punctures and tangency}, $p \geq k + 1 \geq 2$. By definition of average, there exists $\mu_0 = 2,\ldots,p$ such that
    \begin{IEEEeqnarray*}{rCls+x*}
        E_{\omega}(D_{\mu_0})
        & \leq & \frac{1}{p-1} \sum_{\mu=2}^{p} E_{\omega}(D_{\mu})                           & \quad [\text{by definition of average}] \\
        & =    & \frac{E_{\omega}(D_2 \cup \cdots \cup D_p)}{p-1}                             & \quad [\text{since energy is additive}] \\
        & \leq & \frac{e^K}{e^K - 1} \frac{E_{\tilde{\omega}}(D_2 \cup \cdots \cup D_p)}{p-1} & \quad [\text{by \cref{lem:energy wrt different forms}}] \\
        & \leq & \frac{e^K}{e^K - 1} \frac{a}{p-1}                                            & \quad [\text{by Equation \eqref{eq:energy of holo building in proof}}] \\
        & \leq & \frac{e^K}{e^K - 1} \frac{a}{k}                                              & \quad [\text{since $p \geq k + 1$}] \\
        & \leq & \frac{a}{k} + \varepsilon                                                    & \quad [\text{since $K \geq K_2$}].                                  & \qedhere
    \end{IEEEeqnarray*}
\end{proof}

% ------------
% Consequences
% ------------

\begin{theorem}
    \label{lem:computation of cl}
    If $X_{\Omega}$ is a $4$-dimensional convex toric domain then
    \begin{IEEEeqnarray*}{c+x*}
        c_L(X_{\Omega}) = \delta_\Omega.
    \end{IEEEeqnarray*}
\end{theorem}
\begin{proof}
    For every $k \in \Z_{\geq 1}$,
    \begin{IEEEeqnarray*}{rCls+x*}
        \delta_\Omega
        & \leq & c_P(X_{\Omega})                                & \quad [\text{by \cref{lem:c square geq delta}}] \\
        & \leq & c_L(X_{\Omega})                                & \quad [\text{by \cref{lem:c square leq c lag}}] \\
        & \leq & \frac{\tilde{\mathfrak{g}}_{k}(X_{\Omega})}{k} & \quad [\text{by \cref{thm:lagrangian vs g tilde}}] \\
        & =    & \frac{\cgh{k}(X_{\Omega})}{k}                  & \quad [\text{by \cref{prp:g tilde and cgh}}] \\
        & \leq & \frac{\cgh{k}(N(\delta_\Omega))}{k}            & \quad [\text{$X_{\Omega}$ is convex, hence $X_{\Omega} \subset N(\delta_{\Omega})$}] \\
        & =    & \frac{\delta_\Omega(k+1)}{k}                   & \quad [\text{by \cref{lem:cgh of nondisjoint union of cylinders}}].
    \end{IEEEeqnarray*}
    The result follows by taking the infimum over $k$.
\end{proof}

\section{Computations using the higher symplectic capacities}
\label{sec:4}

\subsection{Assumptions on virtual perturbation scheme}
\label{sec:assumptions of virtual perturbation scheme}

In this subsection, we wish to use techniques from contact homology to prove \cref{conj:the conjecture}. Consider the proof of \cref{lem:computation of cl}: to prove the inequality $c_L(X_{\Omega}) \leq \delta_\Omega$, we needed to use the fact that $\tilde{\mathfrak{g}}_k(X_{\Omega}) \leq \cgh{k}(X_{\Omega})$ (which is true if $X_{\Omega}$ is convex and $4$-dimensional). Our approach here will be to consider the capacities $\mathfrak{g}_{k}$ from \cite{siegelHigherSymplecticCapacities2020}, which satisfy $\tilde{\mathfrak{g}}_k(X) \leq \mathfrak{g}_k(X) = \cgh{k}(X)$. As we will see, $\mathfrak{g}_{k}(X)$ is defined using the linearized contact homology of $X$, where $X$ is any nondegenerate Liouville domain.

Very briefly, the linearized contact homology chain complex, denoted $CC(X)$, is generated by the good Reeb orbits of $\partial X$, and therefore maps whose domain is $CC(X)$ should count holomorphic curves which are asymptotic to Reeb orbits. The ``naive'' way to define such counts of holomorphic curves would be to show that they are the elements of a moduli space which is a compact, $0$-dimensional orbifold. However, there is the possibility that a curve is multiply covered. This means that in general it is no longer possible to show that the moduli spaces are transversely cut out, and therefore we do not have access to counts of moduli spaces of holomorphic curves (or at least not in the usual sense of the notion of signed count). In the case where the Liouville domain is $4$-dimensional, there exists the possibility of using automatic transversality techniques to show that the moduli spaces are regular. This is the approach taken by Wendl \cite{wendlAutomaticTransversalityOrbifolds2010}. Nelson \cite{nelsonAutomaticTransversalityContact2015}, Hutchings--Nelson \cite{hutchingsCylindricalContactHomology2016} and Bao--Honda \cite{baoDefinitionCylindricalContact2018} use automatic transversality to define cylindrical contact homology.

In order to define contact homology in more general contexts, one needs to use a suitable notion of virtual count, which is obtained through a virtual perturbation scheme. This was done by Pardon \cite{pardonAlgebraicApproachVirtual2016,pardonContactHomologyVirtual2019} to define contact homology in greater generality. The theory of polyfolds by Hofer--Wysocki--Zehnder \cite{hoferPolyfoldFredholmTheory2021} can also be used to define virtual moduli counts. Alternative approaches using Kuranishi structures (see \cite{fukayaLagrangianIntersectionFloer2010,fukayaLagrangianIntersectionFloer2010a}) have been given by Ishikawa \cite{ishikawaConstructionGeneralSymplectic2018} and Bao--Honda \cite{baoSemiglobalKuranishiCharts2021}.

Unfortunately, linearized contact homology is not yet defined in the generality we need. Indeed, in order to prove \cref{conj:the conjecture}, we need to use the capacities $\mathfrak{g}_k$. These are defined using the linearized contact homology and an augmentation map which counts curves satisfying a tangency constraint. As far as we know, the current work on defining virtual moduli counts does not yet deal with moduli spaces of curves satisfying tangency constraints.

So, during this section, we will work under assumption that it is possible to define a virtual perturbation scheme which makes the invariants and maps described above well-defined (this is expected to be the case).

\begin{assumption}
    \label{assumption}
    We assume the existence of a virtual perturbation scheme which to every compactified moduli space $\overline{\mathcal{M}}$ of asymptotically cylindrical holomorphic curves (in a symplectization or in a Liouville cobordism, possibly satisfying a tangency constraint) assigns a virtual count $\#^{\mathrm{vir}} \overline{\mathcal{M}}$. We will assume in addition that the virtual perturbation scheme has the following properties.
    \begin{enumerate}
        \item If $\#^{\mathrm{vir}} \overline{\mathcal{M}} \neq 0$ then $\operatorname{virdim} \overline{\mathcal{M}} = 0$;
        \item If $\overline{\mathcal{M}}$ is transversely cut out then $\#^{\mathrm{vir}} \overline{\mathcal{M}} = \# \overline{\mathcal{M}}$. In particular, if $\overline{\mathcal{M}}$ is empty then $\#^{\mathrm{vir}} \overline{\mathcal{M}} = 0$;
        \item The virtual count of the boundary of a moduli space (defined as a sum of virtual counts of the moduli spaces that constitute the codimension one boundary strata) is zero. In particular, the expected algebraic identities ($\partial^2 = 0$ for differentials, $\varepsilon \circ \partial = 0$ for augmentations) hold, as well as independence of auxiliary choices of almost complex structure and symplectic divisor.
    \end{enumerate}
\end{assumption}

\subsection{Linearized contact homology}

In this subsection, we define the linearized contact homology of a nondegenerate Liouville domain $X$. This is the homology of a chain complex $CC(X)$, which is described in \cref{def:linearized contact homology}. We also define an augmentation map (\cref{def:augmentation map}), which is necessary to define the capacities $\mathfrak{g}_k$.

\begin{definition}
    Let $(M,\alpha)$ be a contact manifold and $\gamma$ be a Reeb orbit in $M$. We say that $\gamma$ is \textbf{bad} if $\conleyzehnder(\gamma) - \conleyzehnder(\gamma_0)$ is odd, where $\gamma_0$ is the simple Reeb orbit that corresponds to $\gamma$. We say that $\gamma$ is \textbf{good} if it is not bad.
\end{definition}

Since the parity of the Conley--Zehnder index of a Reeb orbit is independent of the choice of trivialization, the definition above is well posed.

\begin{definition}
    \label{def:linearized contact homology}
    If $(X,\lambda)$ is a nondegenerate Liouville domain, the \textbf{linearized contact homology chain complex} of $X$, denoted $CC(X)$, is a chain complex given as follows. First, let $CC(X)$ be the vector space over $\Q$ generated by the set of good Reeb orbits of $(\partial X, \lambda|_{\partial X})$. The differential of $CC(X)$, denoted $\partial$, is given as follows. Choose $J \in \mathcal{J}(X)$. If $\gamma$ is a good Reeb orbit of $\partial X$, we define
    \begin{IEEEeqnarray*}{c+x*}
        \partial \gamma = \sum_{\eta} \p{<}{}{\partial \gamma, \eta}  \eta,
    \end{IEEEeqnarray*}
    where $\p{<}{}{\partial \gamma, \eta}$ is the virtual count (with combinatorial weights) of holomorphic curves in $\R \times \partial X$ with one positive asymptote $\gamma$, one negative asymptote $\eta$, and $k \geq 0$ extra negative asymptotes $\alpha_1,\ldots,\alpha_k$ (called \textbf{anchors}), each weighted by the count of holomorphic planes in $\hat{X}$ asymptotic to $\alpha_j$ (see \cref{fig:differential of lch}).
\end{definition}

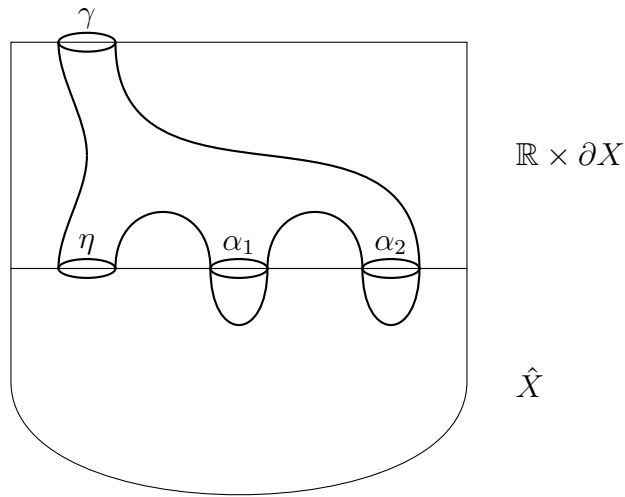
\begin{figure}[htp]
    \centering

    \begin{tikzpicture}
        [
            scale = 0.5,
            help/.style = {very thin, draw = black!50},
            curve/.style = {thick}
        ]

        \tikzmath{
            \rx = 0.75;
            \ry = 0.25;
        }

        % X^2
        \node[anchor=west] at (13,9) {$\R \times \partial X$};
        \draw (0,6) rectangle (12,12);

        % X^1
        \node[anchor=west] at (13,3) {$\hat{X}$};
        \draw (0,3) -- (0,6) -- (12,6) -- (12,3);
        \draw (0,3) .. controls (0,-1) and (12,-1) .. (12,3);

        % Reeb orbit locations
        \coordinate (G) at ( 2,12);
        \coordinate (E) at ( 2, 6);
        \coordinate (A) at ( 6, 6);
        \coordinate (B) at (10, 6);

        \coordinate (L) at (-\rx,0);
        \coordinate (R) at (+\rx,0);

        \coordinate (GL) at ($ (G) + (L) $);
        \coordinate (EL) at ($ (E) + (L) $);
        \coordinate (AL) at ($ (A) + (L) $);
        \coordinate (BL) at ($ (B) + (L) $);

        \coordinate (GR) at ($ (G) + (R) $);
        \coordinate (ER) at ($ (E) + (R) $);
        \coordinate (AR) at ($ (A) + (R) $);
        \coordinate (BR) at ($ (B) + (R) $);

        % Reeb orbits
        \draw[curve] (G) ellipse [x radius = \rx, y radius = \ry] node[above = 1] {$\gamma$};
        \draw[curve] (E) ellipse [x radius = \rx, y radius = \ry] node[above = 1] {$\eta$};
        \draw[curve] (A) ellipse [x radius = \rx, y radius = \ry] node[above = 1] {$\alpha_1$};
        \draw[curve] (B) ellipse [x radius = \rx, y radius = \ry] node[above = 1] {$\alpha_2$};

        % Curves
        \draw[curve] (ER) .. controls ($ (ER) + (0,2) $) and ($ (AL) + (0,2) $) .. (AL);
        \draw[curve] (AR) .. controls ($ (AR) + (0,2) $) and ($ (BL) + (0,2) $) .. (BL);

        \draw[curve] (AL) .. controls ($ (AL) - (0,2) $) and ($ (AR) - (0,2) $) .. (AR);
        \draw[curve] (BL) .. controls ($ (BL) - (0,2) $) and ($ (BR) - (0,2) $) .. (BR);

        \draw[curve] (GR) .. controls ($ (GR) - (0,5) $) and ($ (BR) + (0,5) $) .. (BR);

        \coordinate (C) at ($ (E) + (0,3) $);
        \draw[curve] (EL) .. controls ($ (EL) + (0,1) $) and ($ (C) - (0,1) $) .. (C);
        \draw[curve] (GL) .. controls ($ (GL) - (0,1) $) and ($ (C) + (0,1) $) .. (C);
    \end{tikzpicture}

    \caption{A holomorphic curve with anchors contributing to the coefficient $\p{<}{}{\partial \gamma, \eta}$}
    \label{fig:differential of lch}
\end{figure}

By assumption on the virtual perturbation scheme, $\partial \circ \partial = 0$ and $CC(X)$ is independent (up to chain homotopy equivalence) of the choice of almost complex structure $J$.

\begin{remark}
    \label{rmk:grading for lch}
    In general, the Conley--Zehnder index of a Reeb orbit is well-defined as an element of $\Z_2$. Therefore, the complex $CC(X)$ has a $\Z_{2}$-grading given by $\deg(\gamma) \coloneqq \conleyzehnder(\gamma)$, and with respect to this definition the differential $\partial$ has degree $-1$. If $\pi_1(X) = 0$ and $2 c_1(TX) = 0$, then the Conley--Zehnder index of Reeb orbit is well-defined as an element of $\Z$, which means that $CC(X)$ is $\Z$-graded.
\end{remark}

\begin{definition}
    \label{def:action filtration lch}
    For every $a \in \R$, we denote by $CC^{a}(X)$ the submodule of $CC(X)$ generated by the good Reeb orbits $\gamma$ with action $\mathcal{A}(\gamma) \leq a$. We call this filtration the \textbf{action filtration} of $CC(X)$.
\end{definition}

In the next lemma, we check that this filtration is compatible with the differential.

\begin{lemma}
    \label{lem:action filtration of lch}
    $\partial(CC^{a}(X)) \subset CC^{a}(X)$.
\end{lemma}
\begin{proof}
    Let $\gamma$, $\eta$ be good Reeb orbits such that
    \begin{IEEEeqnarray*}{rCls+x*}
        \mathcal{A}(\gamma)            & \leq & a, \\
        \p{<}{}{\partial \gamma, \eta} & \neq & 0.
    \end{IEEEeqnarray*}
    We wish to show that $\mathcal{A}(\eta) \leq a$. Since $\p{<}{}{\partial \gamma, \eta} \neq 0$ and by assumption on the virtual perturbation scheme, there exists a tuple of Reeb orbits $\Gamma = (\eta, \alpha_1, \ldots, \alpha_p)$ and a (nontrivial) punctured $J$-holomorphic sphere in $\R \times \partial X$ with positive asymptote $\gamma$ and negative asymptotes $\Gamma$. Then,
    \begin{IEEEeqnarray*}{rCls+x*}
        \mathcal{A}(\eta)
        & \leq & \mathcal{A}(\Gamma) & \quad [\text{since $\eta \in \Gamma$}] \\
        & \leq & \mathcal{A}(\gamma) & \quad [\text{by Equation \eqref{eq:energy identity}}] \\
        & \leq & a                   & \quad [\text{by assumption on $\gamma$}].               & \qedhere
    \end{IEEEeqnarray*}
\end{proof}

\begin{definition}
    \label{def:augmentation map}
    Consider the complex $(CC(X), \partial)$. For each $k \in \Z_{\geq 1}$, we define an augmentation ${\epsilon}_k \colon CC(X) \longrightarrow \Q$ as follows. Choose $x \in \itr X$, a symplectic divisor $D$ at $x$, and an almost complex structure $J \in \mathcal{J}(X,D)$. Then, for every good Reeb orbit $\gamma$ define ${\epsilon}_k (\gamma)$ to be the virtual count of $J$-holomorphic planes in $\hat{X}$ which are positively asymptotic to $\gamma$ and have contact order $k$ to $D$ at $x$ (see \cref{fig:augmentation of lch}).
\end{definition}

\begin{figure}[htp]
    \centering

    \begin{tikzpicture}
        [
            scale = 0.5,
            help/.style = {very thin, draw = black!50},
            curve/.style = {thick}
        ]

        \tikzmath{
            \rx = 0.75;
            \ry = 0.25;
        }

        % X^1
        \node[anchor=west] at (13,3) {$\hat{X}$};
        \draw (0,3) -- (0,6) -- (12,6) -- (12,3);
        \draw (0,3) .. controls (0,-1) and (12,-1) .. (12,3);

        % Reeb orbit locations
        \coordinate (G1) at (6,6);
        % \coordinate (G2) at (8,6);

        \coordinate (L) at (-\rx,0);
        \coordinate (R) at (+\rx,0);

        \coordinate (G1L) at ($ (G1) + (L) $);
        % \coordinate (G2L) at ($ (G2) + (L) $);

        \coordinate (G1R) at ($ (G1) + (R) $);
        % \coordinate (G2R) at ($ (G2) + (R) $);

        % Tangency constraint vectors
        \coordinate (P) at (7,2);
        \coordinate (D) at (1, 0.5);

        % Reeb orbits
        \draw[curve] (G1) ellipse [x radius = \rx, y radius = \ry] node[above = 1] {$\gamma$};
        % \draw[curve] (G2) ellipse [x radius = \rx, y radius = \ry] node[above = 1] {$\gamma_2$};

        % Symplectic divisor
        \fill (P) circle (2pt) node[anchor = north west] {$x$};
        \draw[curve] ($ (P) - (D) $) -- ( $ (P) + (D) $ ) node[anchor = west] {$D$};

        % Curve
        % \draw[curve] (G1R) .. controls ($ (G1R) - (0,2) $) and ($ (G2L) - (0,2) $) .. (G2L);
        \draw[curve] (G1L) .. controls ($ (G1L) - (0,2) $) and ($ (P) - (D) $) .. (P);
        \draw[curve] (G1R) .. controls ($ (G1R) - (0,2) $) and ($ (P) + (D) $) .. (P);
    \end{tikzpicture}

    \caption{A holomorphic curve contributing to the count $\epsilon_k(\gamma)$}
    \label{fig:augmentation of lch}
\end{figure}
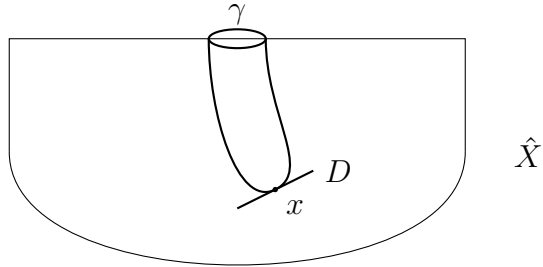

By Equation \eqref{eq:virtual dimension}, $\epsilon_k$ is a map $\epsilon_k \colon CC_{n-1+2k}(X) \longrightarrow \Q$. By assumption on the virtual perturbation scheme, ${\epsilon}_k$ is an augmentation, i.e. ${\epsilon}_k \circ \partial = 0$. Therefore, there is a corresponding map $\epsilon_k \colon CH_{n-1+2k}(X) \longrightarrow \Q$ on homology. In addition, ${\epsilon}_k$ is independent (up to chain homotopy) of the choices of $x, D, J$.

\subsection{Higher symplectic capacities}

% Definitions

% Here we define the symplectic capacities $\mathfrak{g}_k$ from \cite{siegelHigherSymplecticCapacities2020}. We will prove the usual properties of symplectic capacities (see \cref{thm:properties of hsc}), namely monotonicity and conformality. In addition, we prove that the value of the capacities $\mathfrak{g}_k$ can be represented by the action of a tuple of Reeb orbits. In \cref{rmk:computations using reeb orbits property} we show how this property could in principle be combined with results from \cite{guttSymplecticCapacitiesPositive2018} to compare the capacities $\mathfrak{g}_k(X_{\Omega})$ and $\cgh{k}(X_{\Omega})$ when $X_{\Omega}$ is a convex or concave toric domain.

\begin{definition}[{\cite[Section 6.1]{siegelHigherSymplecticCapacities2020}}]
    \label{def:capacities glk}
    Let $k, \ell \in \Z_{\geq 1}$ and $(X,\lambda)$ be a nondegenerate Liouville domain. The \textbf{higher symplectic capacities} of $X$ are given by
    \begin{IEEEeqnarray*}{c+x*}
        \mathfrak{g}_k(X) \coloneqq \inf \{ a > 0 \mid \epsilon_k \colon CH^a(X) \longrightarrow \Q \text{ is nonzero} \}.
    \end{IEEEeqnarray*}
\end{definition}

The capacities $\mathfrak{g}_{k}$ will be useful to us because they have similarities with the McDuff--Siegel capacities $\tilde{\mathfrak{g}}_k$, but also with the Gutt--Hutchings capacities $\cgh{k}$. More specifically:
\begin{enumerate}
    \item Both $\mathfrak{g}_{k}$ and $\tilde{\mathfrak{g}}_k$ are related to the energy of holomorphic planes in $X$ which are asymptotic to a Reeb orbit and satisfy a tangency constraint. In \cref{thm:g tilde vs g hat}, we will actually see that $\tilde{\mathfrak{g}}_k(X) \leq \mathfrak{g}_k(X)$. The capacities $\mathfrak{g}_k$ can be thought of as the SFT counterparts of $\tilde{\mathfrak{g}}_k$, or alternatively the capacities $\tilde{\mathfrak{g}}_k$ can be thought of as the counterparts of $\mathfrak{g}_k$ whose definition does not require the holomorphic curves to be regular.
    \item Both $\mathfrak{g}_{k}$ and $\cgh{k}$ are defined in terms of a map on homology being nonzero. In the case of $\mathfrak{g}_{k}$, we consider the linearized contact homology, and in the case of $\cgh{k}$ the invariant in question is $S^1$-equivariant symplectic homology. Taking into consideration the Bourgeois--Oancea isomorphism (see \cite{bourgeoisEquivariantSymplecticHomology2016}) between linearized contact homology and positive $S^1$-equivariant symplectic homology, one can think of $\mathfrak{g}_{k}$ and $\cgh{k}$ as restatements of one another under this isomorphism. This is the idea behind the proof of \cref{thm:g hat vs gh}, where we show that $\mathfrak{g}_{k}(X) = \cgh{k}(X)$.
\end{enumerate}

\begin{remark}
    \label{rmk:novikov coefficients}
    In the case where $X$ is only an exact symplectic manifold instead of a Liouville domain, we do not have access to an action filtration on $CC(X)$. However, it is possible to define linearized contact homology with coefficients in a Novikov ring $\Lambda_{\geq 0}$, in which case a coefficient in $\Lambda_{\geq 0}$ encodes the energy of a holomorphic curve. This is the approach taken in \cite{siegelHigherSymplecticCapacities2020} to define the capacities $\mathfrak{g}_{k}$. It is not obvious that the definition of $\mathfrak{g}_k$ we give and the one in \cite{siegelHigherSymplecticCapacities2020} are equivalent. However, \cref{def:capacities glk} seems to be the natural analogue when we have access to an action filtration, and in addition the definition we provide will be enough for our purposes.
\end{remark}

\begin{theorem}
    \label{thm:properties of hsc}
    The functions $\mathfrak{g}_k$ satisfy the following properties, for all nondegenerate Liouville domains $(X,\lambda_X)$ and $(Y,\lambda_Y)$ of the same dimension:
    \begin{description}
        \item[(Monotonicity)] If $X \longrightarrow Y$ is an exact symplectic embedding then $\mathfrak{g}_k(X) \leq \mathfrak{g}_k(Y)$.
        \item[(Conformality)] If $\mu > 0$ then ${\mathfrak{g}}_k(X, \mu \lambda_X) = \mu  {\mathfrak{g}}_k(X, \lambda_X)$.
        \item[(Reeb orbits)] If $\pi_1(X) = 0$, $2 c_1(TX) = 0$ and ${\mathfrak{g}}_k(X) < + \infty$, then there exists a Reeb orbit $\gamma$ such that ${\mathfrak{g}}_k(X) = \mathcal{A}(\gamma)$ and $\conleyzehnder(\gamma) =  n - 1 + 2 k$.
    \end{description}
\end{theorem}
\begin{proof}
    We prove monotonicity. If $(X, \lambda^X) \longrightarrow (Y, \lambda^Y)$ is an exact symplectic embedding, then it is possible to define a Viterbo transfer map $CH(Y) \longrightarrow CH(X)$. This map respects the action filtration as well as the augmentation maps, i.e. the diagram
    \begin{IEEEeqnarray*}{c+x*}
        \begin{tikzcd}
            CH^a(Y) \ar[d] \ar[r] & CH(Y) \ar[d] \ar[r, "{\epsilon}_{k}^Y"] & \Q \ar[d, equals] \\
            CH^a(X) \ar[r]        & CH(X) \ar[r, swap, "{\epsilon}_{k}^X"]  & \Q
        \end{tikzcd}
    \end{IEEEeqnarray*}
    commutes. The result then follows by definition of $\tilde{\mathfrak{g}}_k$.

    We prove conformality. If $\gamma$ is a Reeb orbit of $(\partial X, \lambda|_{\partial X})$ of action $\mathcal{A}_{\lambda}(\gamma)$ then $\gamma$ is a Reeb orbit of $(\partial X, \mu \lambda|_{\partial X})$ of action $\mathcal{A}_{\mu \lambda}(\gamma) = \mu \mathcal{A}_{\lambda}(\gamma)$. Therefore, there is a diagram
    \begin{IEEEeqnarray*}{c+x*}
        \begin{tikzcd}
            CH^a      (X,     \lambda) \ar[d, equals] \ar[r] & CH(X,     \lambda) \ar[d, equals] \ar[r, "{\epsilon}_{k}^{\lambda}"] & \Q \ar[d, equals] \\
            CH^{\mu a}(X, \mu \lambda) \ar[r]                & CH(X, \mu \lambda) \ar[r, swap, "{\epsilon}_{k}^{\mu \lambda}"]     & \Q
        \end{tikzcd}
    \end{IEEEeqnarray*}
    Again, the result follows by definition of $\mathfrak{g}_{k}$.

    We prove the Reeb orbits property. Choose a point $x \in \itr X$, a symplectic divisor $D$ through $x$ and an almost complex structure $J \in \mathcal{J}(X,D)$. Consider the complex $CC(X)$, computed with respect to $J$. By assumption and definition of $\mathfrak{g}_{k}$,
    \begin{IEEEeqnarray*}{rCls+x*}
        + \infty
        & > & \mathfrak{g}_k(X) \\
        & = & \inf \{ a > 0 \mid \epsilon_k \colon CH^a(X) \longrightarrow \Q \text{ is nonzero} \} \\
        & = & \inf \{ \mathcal{A}(\beta) \mid \beta \in CC(X) \text{ is such that } {\epsilon}_k (\beta) \neq 0 \text{ and } \partial \beta = 0 \},
    \end{IEEEeqnarray*}
    where $\beta = \sum_{i=1}^{m} a_i \gamma_i$ is a linear combination of Reeb orbits and $\mathcal{A}(\beta) \coloneqq \max_{i=1,\ldots,m} \gamma_i$. Since the action spectrum of $(\partial X, \lambda|_{\partial X})$ is a discrete subset of $\R$, we conclude that in the above expression the infimum is a minimum. More precisely, there exists $\beta = \sum_{i=1}^{m} a_i \gamma_i \in CC_{n-1+2k}(X)$ such that $\epsilon_k(\beta) \neq 0$ and $\mathfrak{g}_k(X) = \mathcal{A}(\beta)$. One of the orbits in this linear combination is such that $\mathcal{A}(\gamma_i) = \mathcal{A}(\beta) = \mathfrak{g}_k(X)$.
\end{proof}

\begin{remark}
    \label{rmk:computations using reeb orbits property}
    In \cite[Theorem 1.6]{guttSymplecticCapacitiesPositive2018} (respectively \cite[Theorem 1.14]{guttSymplecticCapacitiesPositive2018}) Gutt--Hutchings give formulas for $\cgh{k}$ of a convex (respectively concave) toric domain. However, the given proofs only depend on specific properties of the Gutt--Hutchings capacity and not on the definition of the capacity itself. These properties are monotonicity, conformality, a ``Reeb orbits'' property similar to the one of \cref{thm:properties of hsc}, and finally that the capacity be finite on star-shaped domains. If we showed that $\mathfrak{g}_{k}$ is finite on star-shaped domains, we would conclude that $\mathfrak{g}_{k} = \cgh{k}$ on convex or concave toric domains, because in this case both capacities would be given by the formulas in the previously mentioned theorems. Showing that $\mathfrak{g}_{k}$ is finite boils down to showing that the augmentation map is nonzero, which we will do in \cref{sec:augmentation map of an ellipsoid}. However, in \cref{thm:g hat vs gh} we will use this information in combination with the Bourgeois--Oancea isomorphism to conclude that $\mathfrak{g}_{k}(X) = \cgh{k}(X)$ for any nondegenerate Liouville domain $X$. Therefore, the proof suggested above will not be necessary, although it is a proof of $\mathfrak{g}_{k}(X) = \cgh{k}(X)$ alternative to that of \cref{thm:g hat vs gh} when $X$ is a convex or concave toric domain.
\end{remark}

\subsection{Cauchy--Riemann operators on bundles}
\label{sec:cr operators}

In order to show that $\mathfrak{g}_{k}(X) = \cgh{k}(X)$, we will need to show that the augmentation map of a small ellipsoid in $X$ is nonzero (see the proof of \cref{thm:g hat vs gh}). Recall that the augmentation map counts holomorphic curves satisfying a tangency constraint. In \cref{sec:augmentation map of an ellipsoid}, we will explicitly compute how many such holomorphic curves there are. However, a count obtained by explicit methods will not necessarily agree with the virtual count that appears in the definition of the augmentation map. By assumption on the virtual perturbation scheme, it does agree if the relevant moduli space is transversely cut out.

Therefore, in this subsection and the next we will describe the framework that allows us to show that this moduli space is transversely cut out. This subsection deals with the theory of real linear Cauchy--Riemann operators on line bundles, and our main reference is \cite{wendlAutomaticTransversalityOrbifolds2010}. The outline is as follows. First, we review the basic definitions about real linear Cauchy--Riemann operators. By the Riemann--Roch theorem (\cref{thm:riemann roch with punctures}), these operators are Fredholm and their index can be computed from a number of topological quantities associated to them. We will make special use of a criterion by Wendl (\cref{prp:wen D surjective injective criterion}) which guarantees that a real linear Cauchy--Riemann operator defined on a complex line bundle is surjective. For our purposes, we will also need an adaptation of this result to the case where the operator is accompanied by an evaluation map, which we state in \cref{lem:D plus E is surjective}. We now state the assumptions for the rest of this subsection.

Let $(\Sigma, j)$ be a compact Riemann surface without boundary, of genus $g$, with sets of positive and negative punctures $\mathbf{z}^{\pm} = \{z^{\pm}_1,\ldots,z^{\pm}_{p^{\pm}}\}$. Denote $\mathbf{z} = \mathbf{z}^{+} \cup \mathbf{z}^{-}$ and $\dot{\Sigma} = \Sigma \setminus \mathbf{z}$. Choose cylindrical coordinates $(s,t)$ near each puncture $z \in \mathbf{z}$ and denote $\mathcal{U}_z \subset \dot{\Sigma}$ the domain of the coordinates $(s,t)$.

% --------------------------------
% New text, to replace definitions
% --------------------------------

We assume that we are given an {asymptotically Hermitian vector bundle}
\begin{IEEEeqnarray*}{c+x*}
    (E,J) \longrightarrow \dot{\Sigma}, \quad (E_z, J_z, \omega_z) \longrightarrow S^1, \quad \text{for each } z \in \mathbf{z}
\end{IEEEeqnarray*}
over $\dot{\Sigma}$ (see \cite[p.~68]{wendlLecturesSymplecticField2016}). If $\tau = (\tau_z)_{z \in \mathbf{z}}$ is an {asymptotic trivialization} of $E$ (i.e. each $\tau_z$ is a unitary trivialization of $(E_z, J_z, \omega_z)$, see \cite[p.~68]{wendlLecturesSymplecticField2016}), then one can define {Sobolev spaces} of sections of $E$, denoted by $W^{k,p}(E)$, and {weighted Sobolev spaces} of sections of $E$, denoted by $W^{k,p,\delta}(E)$. Let $\mathbf{D}$ be a {real linear Cauchy--Riemann operator} on $E$ (see \cite[Definition C.1.5]{mcduffHolomorphicCurvesSymplectic2012}), together with corresponding asymptotic operators $(\mathbf{A}_z)_{z \in \mathbf{z}}$ (see \cite[Definition 3.25]{wendlLecturesSymplecticField2016}). Some topological quantities which are going to be relevant to us are:
\begin{enumerate}
    \item The \textbf{Euler characteristic} of $\dot{\Sigma}$, which is given by $\chi(\dot{\Sigma}) = 2 - 2 g - \# \mathbf{z}$;
    \item The \textbf{relative first Chern number} of $E$ (with respect to the trivialization $\tau$), which is an integer denoted by $c_1^{\tau}(E) \in \Z$ (see \cite[Definition 5.1]{wendlLecturesSymplecticField2016});
    \item The \textbf{Conley--Zehnder} index (with respect to the trivialization $\tau$) of an asymptotic operator $\mathbf{A}_z$, which is an integer denoted by $\conleyzehnder^{\tau}(\mathbf{A}_z)$ (see \cite[Definitions 3.30 and 3.31]{wendlLecturesSymplecticField2016}).
\end{enumerate}

Using these quantities, we can state the following version of the Riemann--Roch theorem.

\begin{theorem}[Riemann--Roch, {\cite[Theorem 5.4]{wendlLecturesSymplecticField2016}}]
    \label{thm:riemann roch with punctures}
    The operator $\mathbf{D}$ is Fredholm and its (real) Fredholm index is given by
    \begin{IEEEeqnarray*}{c+x*}
        \operatorname{ind} \mathbf{D} = n \chi (\dot{\Sigma}) + 2 c_1^{\tau}(E) + \sum_{z \in \mathbf{z}^+} \conleyzehnder^{\tau}(\mathbf{A}_z) - \sum_{z \in \mathbf{z}^-} \conleyzehnder^{\tau}(\mathbf{A}_z).
    \end{IEEEeqnarray*}
\end{theorem}

For the rest of this subsection, we restrict ourselves to the case where $n = \operatorname{rank}_{\C} E = 1$. Our goal is to state a criterion that guarantees surjectivity of $\mathbf{D}$. This criterion depends on other topological quantities whose definition we now recall (see \cite[Section 2.2]{wendlAutomaticTransversalityOrbifolds2010}).

For every $\lambda$ in the spectrum of $\mathbf{A}_z$, let $w^{\tau}(\lambda)$ be the winding number of any nontrivial section in the $\lambda$-eigenspace of $\mathbf{A}_z$ (computed with respect to the trivialization $\tau$). Define the \textbf{winding numbers}
\begin{IEEEeqnarray*}{rClls+x*}
    \alpha_-^{\tau}(\mathbf{A}_z) & \coloneqq & \max & \{ w^{\tau}(\lambda) \mid \lambda < 0 \text{ is in the spectrum of }\mathbf{A}_z \}, \\
    \alpha_+^{\tau}(\mathbf{A}_z) & \coloneqq & \min & \{ w^{\tau}(\lambda) \mid \lambda > 0 \text{ is in the spectrum of }\mathbf{A}_z \}.
\end{IEEEeqnarray*}
The \textbf{parity} (the reason for this name is Equation \eqref{eq:cz winding parity} below) and associated sets of even and odd punctures are given by
\begin{IEEEeqnarray*}{rCls+x*}
    p(\mathbf{A}_{z}) & \coloneqq & \alpha_{+}^{\tau}(\mathbf{A}_z) - \alpha^{\tau}_{-}(\mathbf{A}_z) \in \{0,1\}, \\
    \mathbf{z}_0      & \coloneqq & \{ z \in \mathbf{z} \mid p(\mathbf{A}_z) = 0 \}, \\
    \mathbf{z}_1      & \coloneqq & \{ z \in \mathbf{z} \mid p(\mathbf{A}_z) = 1 \}.
\end{IEEEeqnarray*}
Finally, the \textbf{adjusted first Chern number} is given by
\begin{IEEEeqnarray*}{c+x*}
    c_1(E,\mathbf{A}_{\mathbf{z}}) = c_1^{\tau}(E) + \sum_{z \in \mathbf{z}^+} \alpha_-^{\tau}(\mathbf{A}_z) - \sum_{z \in \mathbf{z}^-} \alpha_-^{\tau}(\mathbf{A}_z).
\end{IEEEeqnarray*}
These quantities satisfy the following equations.
\begin{IEEEeqnarray}{rCls+x*}
    \conleyzehnder^{\tau}(\mathbf{A}_z) & = & 2 \alpha_{-}^{\tau}(\mathbf{A_z}) + p(\mathbf{A}_z) = 2 \alpha_{+}^{\tau}(\mathbf{A_z}) - p(\mathbf{A}_z), \phantomsection\label{eq:cz winding parity} \\
    2 c_1 (E,\mathbf{A}_{\mathbf{z}})   & = & \operatorname{ind} \mathbf{D} - 2 - 2g + \# \mathbf{z}_0. \phantomsection\label{eq:chern and index}
\end{IEEEeqnarray}

\begin{proposition}[{\cite[Proposition 2.2]{wendlAutomaticTransversalityOrbifolds2010}}]
    \phantomsection\label{prp:wen D surjective injective criterion}
    \begin{enumerate}
        \item[]
        \item If $\operatorname{ind} \mathbf{D} \leq 0$ and $c_1(E, \mathbf{A}_{\mathbf{z}}) < 0$ then $\mathbf{D}$ is injective.
        \item If $\operatorname{ind} \mathbf{D} \geq 0$ and $c_1(E, \mathbf{A}_{\mathbf{z}}) < \operatorname{ind} \mathbf{D}$ then $\mathbf{D}$ is surjective.
    \end{enumerate}
\end{proposition}

We will apply the proposition above to moduli spaces of punctured spheres which have no even punctures. The following corollary is just a restatement of the previous proposition in this simpler case.

\begin{corollary}
    \label{lem:conditions for D surjective genus zero}
    Assume that $g = 0$ and $\# \mathbf{z}_0 = 0$. Then,
    \begin{enumerate}
        \item If $\operatorname{ind} \mathbf{D} \leq 0$ then $\mathbf{D}$ is injective.
        \item If $\operatorname{ind} \mathbf{D} \geq 0$ then $\mathbf{D}$ is surjective.
    \end{enumerate}
\end{corollary}
\begin{proof}
    By \cref{prp:wen D surjective injective criterion} and Equation \eqref{eq:chern and index}.
\end{proof}

We now wish to deal with the case where $\mathbf{D}$ is taken together with an evaluation map (see \cref{lem:D plus E is surjective} below). The tools we need to prove this result are explained in the following remark.

\begin{remark}[{\cite[p.~362-363]{wendlAutomaticTransversalityOrbifolds2010}}]
    \label{rmk:formulas for xi in ker nonzero}
    Suppose that $\ker \mathbf{D} \neq \{0\}$. If $\xi \in \ker \mathbf{D} \setminus \{0\}$, it is possible to show that $\xi$ has only a finite number of zeros, all of positive order, i.e. if $w$ is a zero of $\xi$ then $\operatorname{ord}(\xi;w) > 0$. For every $z \in \mathbf{z}$, there is an \textbf{asymptotic winding number} $\operatorname{wind}_z^{\tau}(\xi) \in \Z$, which has the properties
    \begin{IEEEeqnarray*}{rCls+x*}
        z \in \mathbf{z}^+ & \Longrightarrow & \operatorname{wind}_z^{\tau}(\xi) \leq \alpha_-^{\tau}(\mathbf{A}_z), \\
        z \in \mathbf{z}^- & \Longrightarrow & \operatorname{wind}_z^{\tau}(\xi) \geq \alpha_+^{\tau}(\mathbf{A}_z).
    \end{IEEEeqnarray*}
    Define the \textbf{asymptotic vanishing} of $\xi$, denoted $Z_{\infty}(\xi)$, and the \textbf{count of zeros}, denoted $Z(\xi)$, by
    \begin{IEEEeqnarray*}{rCls+x*}
        Z_{\infty}(\xi) & \coloneqq & \sum_{z \in \mathbf{z}^+} \p{}{1}{\alpha_-^{\tau}(\mathbf{A}_z) - \operatorname{wind}_z^{\tau}(\xi)} + \sum_{z \in \mathbf{z}^-} \p{}{1}{\operatorname{wind}_z^{\tau}(\xi) - \alpha_+^{\tau}(\mathbf{A}_z)} \in \Z_{\geq 0}, \\
        Z(\xi)          & \coloneqq & \sum_{w \in \xi^{-1}(0)} \operatorname{ord}(\xi;w) \in \Z_{\geq 0}.
    \end{IEEEeqnarray*}
    In this case, we have the formula (see \cite[Equation 2.7]{wendlAutomaticTransversalityOrbifolds2010})
    \begin{IEEEeqnarray}{c}
        \phantomsection\label{eq:c1 and asy vanishing}
        c_1(E,\mathbf{A}_{\mathbf{z}}) = Z(\xi) + Z_{\infty}(\xi).
    \end{IEEEeqnarray}
\end{remark}

\begin{lemma}
    \label{lem:D plus E is surjective}
    Let $w \in \dot{\Sigma}$ be a point and $\mathbf{E} \colon W^{1,p}(\dot{\Sigma}, E) \longrightarrow E_w$ be the evaluation map at $w$, i.e. $\mathbf{E}(\xi) = \xi_w$. Assume that $g = 0$ and $\# \mathbf{z}_0 = 0$. If $\operatorname{ind} \mathbf{D} = 2$ then $\mathbf{D} \oplus \mathbf{E} \colon W^{1,p}(\dot{\Sigma}, E) \longrightarrow L^p(\dot{\Sigma}, \Hom^{0,1}(T \dot{\Sigma}, E)) \oplus E_w$ is surjective.
\end{lemma}
\begin{proof}
    It is enough to show that the maps
    \begin{IEEEeqnarray*}{rCls+x*}
        \mathbf{D} \colon W^{1,p}(\dot{\Sigma}, E)           & \longrightarrow & L^p(\dot{\Sigma}, \Hom^{0,1}(T \dot{\Sigma}, E)), \\
        \mathbf{E}|_{\ker \mathbf{D}} \colon \ker \mathbf{D} & \longrightarrow & E_w
    \end{IEEEeqnarray*}
    are surjective. By \cref{lem:conditions for D surjective genus zero}, $\mathbf{D}$ is surjective. Since $\dim \ker \mathbf{D} = \operatorname{ind} \mathbf{D} = 2$ and $\dim_{\R} E_w = 2$, the map $\mathbf{E}|_{\ker \mathbf{D}}$ is surjective if and only if it is injective. So, we show that $\ker(E|_{\ker \mathbf{D}}) = \ker \mathbf{E} \cap \ker \mathbf{D} = \{0\}$. For this, let $\xi \in \ker \mathbf{E} \cap \ker \mathbf{D}$ and assume by contradiction that $\xi \neq 0$. Consider the quantities defined in \cref{rmk:formulas for xi in ker nonzero}. We compute
    \begin{IEEEeqnarray*}{rCls+x*}
        0
        & =    & \operatorname{ind} \mathbf{D} - 2 & \quad [\text{by assumption}] \\
        & =    & 2 c_1(E,\mathbf{A}_{\mathbf{z}})  & \quad [\text{by Equation \eqref{eq:chern and index}}] \\
        & =    & 2 Z(\xi) + 2 Z_{\infty}(\xi)      & \quad [\text{by Equation \eqref{eq:c1 and asy vanishing}}] \\
        & \geq & 0                                 & \quad [\text{by definition of $Z$ and $Z_{\infty}$}],
    \end{IEEEeqnarray*}
    which implies that $Z(\xi) = 0$. This gives the desired contradiction, because
    \begin{IEEEeqnarray*}{rCls+x*}
        0
        & =    & Z(\xi)                                             & \quad [\text{by the previous computation}] \\
        & =    & \sum_{z \in \xi^{-1}(0)} \operatorname{ord}(\xi;z) & \quad [\text{by definition of $Z$}] \\
        & \geq & \operatorname{ord}(\xi;w)                          & \quad [\text{since $\xi_w = \mathbf{E}(\xi) = 0$}] \\
        & >    & 0                                                  & \quad [\text{by \cref{rmk:formulas for xi in ker nonzero}}]. & \qedhere
    \end{IEEEeqnarray*}
\end{proof}

\subsection{Cauchy--Riemann operators as sections}
\phantomsection\label{sec:functional analytic setup}

In this subsection, we phrase the notion of a map $u \colon \dot{\Sigma} \longrightarrow \hat{X}$ being holomorphic in terms of $u$ being in the zero set of a section $\overline{\partial} \colon \mathcal{T} \times \mathcal{B} \longrightarrow \mathcal{E}$. The goal of this point of view is that we can then think of moduli spaces of holomorphic curves in $\hat{X}$ as the zero set of the section $\overline{\partial}$. To see if such a moduli space is regular near $(j, u)$, one needs to consider the linearization $\mathbf{L}_{(j,u)}$ of $\overline{\partial}$ at $(j,u)$, and prove that it is surjective. We will see that a suitable restriction of $\mathbf{L}_{(j,u)}$ is a real linear Cauchy--Riemann operator, and therefore we can use the theory from the last subsection to show that $\mathbf{L}_{(j,u)}$ is surjective in some particular cases (\cref{lem:Du is surjective case n is 1,lem:DX surj implies DY surj}).

\begin{definition}
    \label{def:moduli space with asy markers}
    Let $(X, \omega, \lambda)$ be a symplectic cobordism, $J \in \mathcal{J}(X)$ be a cylindrical almost complex structure on $\hat{X}$, and $\Gamma^{\pm} = (\gamma^{\pm}_1, \ldots, \gamma^{\pm}_{p^{\pm}})$ be tuples of Reeb orbits on $\partial^{\pm} X$. Consider the sphere $S^2$ together with a set of punctures $\mathbf{z}^{\pm} = \{z^{\pm}_1, \ldots, z^{\pm}_{p^{\pm}}\} \subset S^2$ and a corresponding set of asymptotic markers $\mathbf{v}^{\pm} = \{v^{\pm}_1, \ldots, v^{\pm}_{p^{\pm}}\}$ (i.e., $v_i^{\pm} \in (T_{z_i^{\pm}} S^2 \setminus \{0\}) / \R_{> 0}$). Define $\mathcal{M}^{\$,J}_X(\Gamma^+, \Gamma^-)$ to be the moduli space of (equivalence classes of) pairs $(j,u)$, where $j$ is an almost complex structure on $S^2$ and $u \colon (\dot{S}^2, j) \longrightarrow (\hat{X}, J)$ is an asymptotically cylindrical holomorphic curve such that
    \begin{enumerate}
        \item $u$ is positively/negatively asymptotic to $\gamma^{\pm}_i$ at $z^{\pm}_i$ for all $i$;
        \item If $c$ is a path in $S^2$ with $c(0) = z_i^{\pm}$ and $\dot{c}(0) = v_i^{\pm}$ for some $i$, then $\lim_{t \to 0^+} u(c(t)) = (\pm \infty, \gamma^{\pm}_i(0))$.
    \end{enumerate}
\end{definition}

\begin{remark}
    \label{rmk:counts of moduli spaces with or without asy markers}
    There is a surjective map $\pi^{\$} \colon \mathcal{M}^{\$, J}_{X}(\Gamma^+, \Gamma^-) \longrightarrow \mathcal{M}^{J}_{X}(\Gamma^+, \Gamma^-)$ given by forgetting the asymptotic markers. By \cite[Proposition 11.1]{wendlLecturesSymplecticField2016}, for every $u \in \mathcal{M}^{J}_{X}(\Gamma^+, \Gamma^-)$ the preimage $(\pi^{\$})^{-1}(u)$ contains exactly
    \begin{IEEEeqnarray*}{c+x*}
        \frac{\prod_{\gamma \in \Gamma^+ \cup \Gamma^-} m(\gamma)}{|\operatorname{Aut}(u)|}
    \end{IEEEeqnarray*}
    elements, where $m(\gamma)$ is the multiplicity of the Reeb orbit $\gamma$ and $\operatorname{Aut}(u)$ is the automorphism group of $u = (\Sigma, j, \mathbf{z}, u)$, i.e. an element of $\operatorname{Aut}(u)$ is a biholomorphism $\phi \colon \Sigma \longrightarrow \Sigma$ such that $u \circ \phi = u$ and $\phi(z_i^{\pm}) = z_i^{\pm}$ for every $i$.
\end{remark}

We will work with the following assumptions. Let $\Sigma = S^2$ be the sphere, (without any specified almost complex structure). Let $z \in \Sigma$ be a puncture\footnote{We point out that the results of this subsection can be stated in the case where $\Sigma$ has no negative punctures and any number of positive punctures. Since for our purposes it is enough to consider the case of one positive puncture, we will restrict ourselves to this case to keep the notation simpler.} and $v \in (T_{z} \Sigma \setminus \{0\}) / \R_{> 0}$ be a corresponding asymptotic marker. There are cylindrical coordinates $(s, t)$ on $\dot{\Sigma}$ near $z$, with the additional property that $v$ agrees with the direction $t = 0$. We will also assume that $\mathcal{T} \subset \mathcal{J}(\Sigma)$ is a Teichmüller slice as in \cite[Section 3.1]{wendlAutomaticTransversalityOrbifolds2010}, where $\mathcal{J}(\Sigma)$ denotes the set of almost complex structures on $\Sigma = S^2$. Let $(X, \lambda)$ be a nondegenerate Liouville domain of dimension $2n$ and $J \in \mathcal{J}(X)$ be an admissible almost complex structure on $\hat{X}$. Let $\gamma$ be a Reeb orbit in $\partial X$. Denote by $m$ the multiplicity of $\gamma$ and by $T$ the period of the simple Reeb orbit underlying $\gamma$ (so, the period of $\gamma$ is $m T$). Choose once and for all a parametrization $\phi \colon S^1 \times D^{2n-2} \longrightarrow O$, where $O \subset \partial X$ is an open neighbourhood of $\gamma$ and
\begin{IEEEeqnarray*}{c+x*}
    D^{2n-2} \coloneqq \{(z_1,\ldots,z_{n-1}) \in \C^{n-1} \mid |z_1| < 1, \ldots, |z_{n-1}| < 1 \}
\end{IEEEeqnarray*}
is the polydisk, such that $t \longmapsto \phi(t,0)$ is a parametrization of the simple Reeb orbit underlying $\gamma$. In this case, we denote by $(\vartheta, \zeta) = \phi^{-1} \colon O \longrightarrow S^1 \times D^{2n-2}$ the coordinates near $\gamma$.

We define a vector bundle $\pi \colon \mathcal{E} \longrightarrow \mathcal{T} \times \mathcal{B}$ as follows. Let $\mathcal{B}$ be the set of maps $u \colon \dot{\Sigma} \longrightarrow \hat{X}$ of class $W^{k,p}_{\mathrm{loc}}$ satisfying the following property. Write $u$ with respect to the cylindrical coordinates $(s,t)$ near $z$. First, we require that $u(s,t) \in \R_{\geq 0} \times O$ for $s$ big enough. Write $u$ with respect to the coordinates $(\vartheta, \zeta)$ near $\gamma$ on the target and cylindrical coordinates $(s,t)$ on the domain:
\begin{IEEEeqnarray*}{rCls+x*}
    u(s,t)
    & = & (\pi_{\R} \circ u(s,t), \pi_{\partial X} \circ u (s,t)) \\
    & = & (\pi_{\R} \circ u(s,t), \vartheta(s,t), \zeta(s,t)).
\end{IEEEeqnarray*}
Finally, we require that there exists $a \in \R$ such that the map
\begin{IEEEeqnarray*}{c+x*}
    (s,t) \longmapsto (\pi_{\R} \circ u(s,t), \vartheta(s,t), \zeta(s,t)) - (m T s + a, m T t, 0)
\end{IEEEeqnarray*}
is of class $W^{k,p,\delta}$. The fibre, total space, projection and zero section are defined by
\begin{IEEEeqnarray*}{rCls+x*}
    \mathcal{E}_{(j,u)} & \coloneqq & W^{k-1,p,\delta}(\Hom^{0,1}((T \dot{\Sigma}, j), (u^* T \hat{X}, J))), \quad \text{for every } (j,u) \in \mathcal{T} \times \mathcal{B}, \\
    \mathcal{E}         & \coloneqq & \coprod_{(j,u) \in \mathcal{T} \times \mathcal{B}} \mathcal{E}_{(j,u)} = \{ (j, u, \xi) \mid (j,u) \in \mathcal{T} \times \mathcal{B}, \, \xi \in \mathcal{E}_{(j,u)} \}, \\
    \pi(j,u, \eta)      & \coloneqq & (j,u), \\
    z(j,u)              & \coloneqq & (j,u,0).
\end{IEEEeqnarray*}
The {Cauchy--Riemann operators} are sections
\begin{IEEEeqnarray*}{rClCrCl}
    \overline{\partial}_j \colon \mathcal{B}                  & \longrightarrow & \mathcal{E}, & \qquad & \overline{\partial}_j(u) & \coloneqq & T u + J \circ Tu \circ j \in \mathcal{E}_{(j,u)}, \\
    \overline{\partial} \colon \mathcal{T} \times \mathcal{B} & \longrightarrow & \mathcal{E}, & \qquad & \overline{\partial}(j,u) & \coloneqq & \overline{\partial}_j(u).
\end{IEEEeqnarray*}
Let $(j,u) \in \mathcal{T} \times \mathcal{B}$ be such that $\overline{\partial}(j ,u) = 0$. There is a vertical projection map $P_{(j,u)} \colon T_{(j,u,0)} \mathcal{E} \longrightarrow \mathcal{E}_{(j,u)}$ which is given by
\begin{IEEEeqnarray*}{c+x*}
    P_{(j,u)}(\eta) \coloneqq (\id - \dv (z \circ \pi)(j,u,0)) \eta.
\end{IEEEeqnarray*}
The linearizations of $\overline{\partial}_j$ and $\overline{\partial}$ at $(j, u)$ are then given by
\begin{IEEEeqnarray*}{rCls+x*}
    \mathbf{D}_{(j,u)} & \coloneqq & P_{(j,u)} \circ \dv (\overline{\partial}_j)(u) \colon T_u \mathcal{B} \longrightarrow \mathcal{E}_{(j,u)}, \\
    \mathbf{L}_{(j,u)} & \coloneqq & P_{(j,u)} \circ \dv (\overline{\partial})(j,u) \colon T_j \mathcal{T} \oplus T_u \mathcal{B} \longrightarrow \mathcal{E}_{(j,u)}.
\end{IEEEeqnarray*}
Define also the restriction
\begin{IEEEeqnarray*}{c+x*}
    \mathbf{F}_{(j,u)} \coloneqq \mathbf{L}_{(j,u)}|_{T_j \mathcal{T}} \colon T_j \mathcal{T} \longrightarrow \mathcal{E}_{(j,u)}.
\end{IEEEeqnarray*}
Now choose a smooth function $f \colon \dot{\Sigma} \longrightarrow \R$ such that $f(s,t) = \delta s$ on the cylindrical end of $\dot{\Sigma}$. Define the \textbf{restriction} of $\mathbf{D}_{(j,u)}$, denoted $\mathbf{D}_{\delta}$, and the \textbf{conjugation} of $\mathbf{D}_{(j,u)}$, denoted $\mathbf{D}_0$, to be the unique maps such that the diagram
\begin{IEEEeqnarray}{c+x*}
    \phantomsection\label{eq:restriction conjugation}
    \begin{tikzcd}
        T_u \mathcal{B} \ar[d, swap, "\mathbf{D}_{(j,u)}"] & W^{k,p,\delta}(u^* T \hat{X}) \ar[d, "\mathbf{D}_{\delta}"] \ar[l, hook'] \ar[r, hook, two heads, "\xi \mapsto e^f \xi"] & W^{k,p}(u^* T \hat{X}) \ar[d, "\mathbf{D}_0"] \\
        \mathcal{E}_{(j,u)} \ar[r, equals]                 & W^{k-1,p,\delta}(\Hom^{0,1}(T \dot{\Sigma}, u^* T \hat{X})) \ar[r, hook, two heads, swap, "\eta \mapsto e^f \eta"]       & W^{k-1,p}(\Hom^{0,1}(T \dot{\Sigma}, u^* T \hat{X}))
    \end{tikzcd}
\end{IEEEeqnarray}
commutes. The maps $\mathbf{D}_\delta$ and $\mathbf{D}_0$ are real linear Cauchy--Riemann operators.

% NOTE: Existence of global trivializations.
% - First question is if asymptotic trivializations exist or not. They should exist always, because otherwise how does this theory work? Wendl creates his theory in such a generality that the trivializations always exist.
% - Second question is when global trivializations exist. [Wen16, p. 81] explains why such a trivialization exists: if there are nonzero punctures then the punctured surface can be retracted to its one skeleton. Tbh I do not understand why this is enough but clearly it suffices that the number of punctures is not zero.
\begin{lemma}
    \phantomsection\label{lem:Du is surjective case n is 1}
    If $n=1$ then $\mathbf{L}_{(j,u)}$ is surjective.
\end{lemma}
\begin{proof}
    Let $\tau_1$ be a global complex trivialization of $u^* T \hat{X}$ extending to an asymptotic unitary trivialization near $z$. Let $\tau_2$ be the unitary trivialization of $u^* T \hat{X}$ near $z$ which is induced from the decomposition $T_{(r,x)}(\R \times \partial X) = \p{<}{}{\partial_r} \oplus \p{<}{}{R^{\partial X}_x}$. It is shown in the proof of \cite[Lemma 7.10]{wendlLecturesSymplecticField2016} that the operator $\mathbf{D}_0$ is asymptotic at $z$ to $- J \partial_t + \delta$, which is nondegenerate and has Conley--Zehnder index $\conleyzehnder^{\tau_2}(- J \partial_t + \delta) = -1$. Therefore, $z$ is an odd puncture and $\# \mathbf{z}_0 = 0$. We show that $c_1^{\tau_2}(u^* T \hat{X}) = m$, where $m$ is the multiplicity of the asymptotic Reeb orbit $\gamma$:%
    \begin{IEEEeqnarray*}{rCls+x*}
        c_1^{\tau_2}(u^* T \hat{X})
        & = & c_1^{\tau_1}(u^* T \hat{X}) + \deg(\tau_1|_{E_{z}} \circ (\tau_2|_{E_{z}})^{-1}) & \quad [\text{by \cite[Exercise 5.3]{wendlLecturesSymplecticField2016}}] \\
        & = & \deg(\tau_1|_{E_{z}} \circ (\tau_2|_{E_{z}})^{-1})                               & \quad [\text{by definition of $c_1^{\tau_1}$}] \\
        & = & m,
    \end{IEEEeqnarray*}
    where in the last equality we have used the fact that if $(s,t)$ are the cylindrical coordinates near $z$, then for $s$ large enough the map $t \longmapsto \tau_1|_{u(s,t)} \circ (\tau_2|_{u(s,t)})^{-1}$ winds around the origin $m$ times. We show that $\operatorname{ind} \mathbf{D}_0 \geq 2$.
    \begin{IEEEeqnarray*}{rCls+x*}
        \operatorname{ind} \mathbf{D}_0
        & =    & n \chi(\dot{\Sigma}) + 2 c_1^{\tau_2}(u^* T \hat{X}) + \conleyzehnder^{\tau_2}(- J \partial_t + \delta) & \quad [\text{by \cref{thm:riemann roch with punctures}}] \\
        & =    & 2m                                                                                                      & \quad [\text{since $n = 1$ and $g = 0$}] \\
        & \geq & 2                                                                                                       & \quad [\text{since $m \geq 1$}].
    \end{IEEEeqnarray*}
    By \cref{lem:conditions for D surjective genus zero}, this implies that $\mathbf{D}_0$ is surjective. By Diagram \eqref{eq:restriction conjugation}, the operator $\mathbf{D}_{(j,u)}$ is also surjective. Therefore, $\mathbf{L}_{(j,u)} = \mathbf{F}_{(j,u)} + \mathbf{D}_{(j,u)}$ is also surjective.
\end{proof}

From now until the end of this subsection, let $(X, \lambda^X)$ be a Liouville domain of dimension $2n$ and $(Y, \lambda^Y)$ be a Liouville domain of dimension $2n + 2$ such that
\begin{enumerate}
    \item $X \subset Y$ and $\partial X \subset \partial Y$;
    \item the inclusion $\iota \colon X \longrightarrow Y$ is a Liouville embedding;
    \item if $x \in X$ then $Z_x^{X} = Z_x^{Y}$;
    \item if $x \in \partial X$ then $R_x^{\partial X} = R^{\partial Y}_x$.
\end{enumerate}
In this case, we have an inclusion of completions $\hat{X} \subset \hat{Y}$ as sets. By assumption, $Z^X$ is $\iota$-related to $Z^Y$, which implies that there is a map $\hat{\iota} \colon \hat{X} \longrightarrow \hat{Y}$ on the level of completions. Since in this case $\hat{X} \subset \hat{Y}$, $\hat{\iota}$ is the inclusion. Assume that $J^X \in \mathcal{J}(X)$ and $J^Y \in \mathcal{J}(Y)$ are almost complex structures on $\hat{X}$ and $\hat{Y}$ respectively, such that $\hat{\iota} \colon \hat{X} \longrightarrow \hat{Y}$ is holomorphic. As before, let $\gamma$ be a Reeb orbit in $\partial X$. Notice that $\gamma$ can also be seen as a Reeb orbit in $\partial Y$. Choose once and for all parametrizations $\phi^X \colon S^1 \times D^{2n-2} \longrightarrow O^X$ and $\phi^Y \colon S^1 \times D^{2n} \longrightarrow O^Y$ near $\gamma$ with the properties as before, and also such that the diagram
\begin{IEEEeqnarray*}{c+x*}
    \begin{tikzcd}
        S^1 \times D^{2n - 2} \ar[r, hook, two heads, "\phi^X"] \ar[d, hook] & O^X \ar[r, hook] \ar[d, hook, dashed, "\exists !"] & \partial X \ar[d, hook, "\iota_{\partial Y, \partial X}"] \\
        S^1 \times D^{2n} \ar[r, hook, two heads, "\phi^Y"]                  & O^Y \ar[r, hook]                                   & \partial Y
    \end{tikzcd}
\end{IEEEeqnarray*}
commutes. We will consider the Cauchy--Riemann operator and its linearization for both $X$ and $Y$. We will use the notation
\begin{IEEEeqnarray*}{rClCrClCrCl}
    \pi^X \colon \mathcal{E}X & \longrightarrow & \mathcal{T} \times \mathcal{B}X, & \qquad & \overline{\partial}\vphantom{\partial}^X \colon \mathcal{T} \times \mathcal{B}X & \longrightarrow & \mathcal{E} X, & \qquad & \mathbf{L}^X_{(j,u)} \colon T_j \mathcal{T} \oplus T_u \mathcal{B} X & \longrightarrow & \mathcal{E}_{(j,u)} X, \\
    \pi^Y \colon \mathcal{E}Y & \longrightarrow & \mathcal{T} \times \mathcal{B}Y, & \qquad & \overline{\partial}\vphantom{\partial}^Y \colon \mathcal{T} \times \mathcal{B}Y & \longrightarrow & \mathcal{E} Y, & \qquad & \mathbf{L}^Y_{(j,u)} \colon T_j \mathcal{T} \oplus T_u \mathcal{B} Y & \longrightarrow & \mathcal{E}_{(j,u)} Y
\end{IEEEeqnarray*}
to distinguish the bundles and maps for $X$ and $Y$. Define maps
\begin{IEEEeqnarray*}{rClCrCl}
    \mathcal{B}\iota \colon \mathcal{B} X & \longrightarrow & \mathcal{B}Y, & \quad & \mathcal{B}\iota(u)        & \coloneqq & \hat{\iota} \circ u, \\
    \mathcal{E}\iota \colon \mathcal{E} X & \longrightarrow & \mathcal{E}Y, & \quad & \mathcal{E}\iota(j,u,\eta) & \coloneqq & (j, \hat{\iota} \circ u, T \hat{\iota} \circ \eta).
\end{IEEEeqnarray*}
Then, the diagrams
\begin{IEEEeqnarray*}{c+x*}
    \begin{tikzcd}
        \mathcal{E}X \ar[r, "\pi^X"] \ar[d, swap, "\mathcal{E}\iota"]                                                                                & \mathcal{T} \times \mathcal{B}X \ar[d, "\id_{\mathcal{T}} \times \mathcal{B}\iota"] & & \mathcal{T} \times \mathcal{B}X \ar[d, swap, "\id_{\mathcal{T}} \times \mathcal{B}\iota"] \ar[r, "z^X"] & \mathcal{E}X \ar[d, "\mathcal{E}\iota"] \\
        \mathcal{E}Y \ar[r, swap, "\pi^Y"]                                                                                                           & \mathcal{T} \times \mathcal{B}Y                                                     & & \mathcal{T} \times \mathcal{B}Y \ar[r, swap, "z^Y"]                                                     & \mathcal{E}Y                            \\
        \mathcal{T} \times \mathcal{B}X \ar[r, "\overline{\partial}\vphantom{\partial}^X"] \ar[d, swap, "\id_{\mathcal{T}} \times \mathcal{B}\iota"] & \mathcal{E}X \ar[d, "\mathcal{E}\iota"]                                             & & (z^X)^* T \mathcal{E} X \ar[r, "P^X"] \ar[d, swap, "T \mathcal{E} \iota"]                               & \mathcal{E} X \ar[d, "\mathcal{E} \iota"] \\
        \mathcal{T} \times \mathcal{B}Y \ar[r, swap, "\overline{\partial}\vphantom{\partial}^Y"]                                                     & \mathcal{E}Y                                                                        & & (z^Y)^* T \mathcal{E} Y \ar[r, swap, "P^Y"]                                                             & \mathcal{E} Y
    \end{tikzcd}
\end{IEEEeqnarray*}
commute. By the chain rule, the diagram
\begin{IEEEeqnarray}{c+x*}
    \phantomsection\label{eq:diag naturality of lcro}
    \begin{tikzcd}
        T_u \mathcal{B} X                     \ar[rr, bend left = 40, "\mathbf{D}^X_{(j,u)}"]                          \ar[r, "\dv \overline{\partial}\vphantom{\partial}^X_j(u)"]                         \ar[d, swap, "\dv(\mathcal{B} \iota)(u)"] & T_{(j,u,0)} \mathcal{E} X                     \ar[r, "P_{(j,u)}^X"]                         \ar[d, "\dv(\mathcal{E}\iota)(\overline{\partial}\vphantom{\partial}^X_j(u))"] & \mathcal{E}_{(j,u)} X                    \ar[d, "\mathcal{E}_{(j,u)} \iota"] \\
        T_{\hat{\iota} \circ u} \mathcal{B} Y \ar[rr, swap, bend right = 40, "\mathbf{D}^Y_{(j,\hat{\iota} \circ u)}"] \ar[r, swap, "\dv \overline{\partial}\vphantom{\partial}^Y_j(\hat{\iota} \circ u)"]                                           & T_{(j, \hat{\iota} \circ u, 0)} \mathcal{E} Y \ar[r, swap, "P^Y_{(j,\hat{\iota} \circ u)}"]                                                                                & \mathcal{E}_{(j, \hat{\iota} \circ u)} Y
    \end{tikzcd}
\end{IEEEeqnarray}
is also commutative whenever $\overline{\partial}\vphantom{\partial}^X(j,u) = 0$. For simplicity, we will denote $\hat{\iota} \circ u \in \mathcal{B} Y$ by $u$. Let $w \in \dot{\Sigma}$ and define the {evaluation map}
\begin{IEEEeqnarray*}{rrCl}
    \operatorname{ev}^X \colon & \mathcal{B} X & \longrightarrow & \hat{X} \\
                               & u             & \longmapsto     & u(w)
\end{IEEEeqnarray*}
as well as its derivative $\mathbf{E}^X_u \coloneqq \dv (\operatorname{ev}^{X})(u) \colon T_u \mathcal{B} X \longrightarrow T_{u(w)} \hat{X}$. In the following lemma, we show that if a holomorphic curve $u$ in $X$ is regular (in $X$) then the corresponding holomorphic curve $u$ in $Y$ is also regular. See also \cite[Proposition A.1]{mcduffSymplecticCapacitiesUnperturbed2022} for a similar result.

\begin{lemma}
    \phantomsection\label{lem:DX surj implies DY surj}
    Let $(j,u) \in \mathcal{T} \times \mathcal{B} X$ be such that $\overline{\partial}^X(j, u) = 0$. Assume that the normal Conley--Zehnder index of $\gamma$ is $1$.
    \begin{enumerate}
        \item \label{lem:DX surj implies DY surj 1} If $\mathbf{L}_{(j,u)}^X$ is surjective then so is $\mathbf{L}^Y_{(j,u)}$.
        \item \label{lem:DX surj implies DY surj 2} If $\mathbf{L}_{(j,u)}^X \oplus \mathbf{E}^X_u$ is surjective then so is $\mathbf{L}^Y_{(j,u)} \oplus \mathbf{E}^Y_u$.
    \end{enumerate}
\end{lemma}
\begin{proof}
    Consider the decomposition $T_x \hat{Y} = T_x \hat{X} \oplus (T_x \hat{X})^{\perp}$ for $x \in \hat{X}$. Let $\tau$ be a global complex trivialization of $u^* T \hat{Y}$, extending to an asymptotic unitary trivialization near the punctures, and such that $\tau$ restricts to a trivialization of $u^* T \hat{X}$ and $u^* (T \hat{X})^{\perp}$. There are splittings
    \begin{IEEEeqnarray*}{rClCrCl}
        T_u \mathcal{B} Y     & = & T_u \mathcal{B} X \oplus T_u^{\perp} \mathcal{B} X,         & \quad \text{where} \quad & T_u^{\perp} \mathcal{B} X     & = & W^{k,p,\delta}(u^*(T \hat{X})^{\perp}), \\
        \mathcal{E}_{(j,u)} Y & = & \mathcal{E}_{(j,u)} X \oplus \mathcal{E}_{(j,u)}^{\perp} X, & \quad \text{where} \quad & \mathcal{E}^{\perp}_{(j,u)} X & = & W^{k-1,p,\delta}(\Hom^{0,1}(T \dot{\Sigma}, u^*(T \hat{X})^{\perp})).
    \end{IEEEeqnarray*}
    We can write the maps
    \begin{IEEEeqnarray*}{rCl}
        \mathbf{L}_{(j,u)}^Y & \colon & T_j \mathcal{T} \oplus T_u \mathcal{B} X \oplus T_u^{\perp} \mathcal{B} X \longrightarrow \mathcal{E}_{(j,u)} X \oplus \mathcal{E}_{(j,u)}^{\perp} X, \\
        \mathbf{D}_{(j,u)}^Y & \colon & T_u \mathcal{B} X \oplus T_u^{\perp} \mathcal{B} X                        \longrightarrow \mathcal{E}_{(j,u)} X \oplus \mathcal{E}_{(j,u)}^{\perp} X, \\
        \mathbf{L}_{(j,u)}^X & \colon & T_j \mathcal{T} \oplus T_u \mathcal{B} X                                  \longrightarrow \mathcal{E}_{(j,u)} X, \\
        \mathbf{F}_{(j,u)}^Y & \colon & T_j \mathcal{T}                                                           \longrightarrow \mathcal{E}_{(j,u)} X \oplus \mathcal{E}_{(j,u)}^{\perp} X, \\
        \mathbf{E}_{u}^Y     & \colon & T_u \mathcal{B} X \oplus T_u^{\perp} \mathcal{B} X                        \longrightarrow T_{u(w)} \hat{X} \oplus (T_{u(w)} \hat{X})^{\perp}
    \end{IEEEeqnarray*}
    as block matrices
    \begin{IEEEeqnarray}{rCl}
        \mathbf{L}_{(j,u)}^Y
        & = &
        \begin{bmatrix}
            \mathbf{F}^X_{(j,u)} & \mathbf{D}^X_{(j,u)} & \mathbf{D}^{TN}_{(j,u)} \\
            0                    & 0                    & \mathbf{D}^{NN}_{(j,u)}
        \end{bmatrix}, \phantomsection\label{eq:decomposition of cr ops 1}\\
        \mathbf{D}_{(j,u)}^Y
        & = &
        \begin{bmatrix}
            \mathbf{D}^X_{(j,u)} & \mathbf{D}^{TN}_{(j,u)} \\
            0                    & \mathbf{D}^{NN}_{(j,u)}
        \end{bmatrix}, \phantomsection\label{eq:decomposition of cr ops 2}\\
        \mathbf{L}_{(j,u)}^X
        & = &
        \begin{bmatrix}
            \mathbf{F}^X_{(j,u)} & \mathbf{D}^X_{(j,u)}
        \end{bmatrix}, \phantomsection\label{eq:decomposition of cr ops 3}\\
        \mathbf{F}_{(j,u)}^Y
        & = &
        \begin{bmatrix}
            \mathbf{F}^X_{(j,u)} \\
            0
        \end{bmatrix}, \phantomsection\label{eq:decomposition of cr ops 4}\\
        \mathbf{E}_{u}^Y
        & = &
        \begin{bmatrix}
            \mathbf{E}^X_{u} & 0 \\
            0                & \mathbf{E}^{NN}_{u}
        \end{bmatrix}, \phantomsection\label{eq:decomposition of cr ops 5}
    \end{IEEEeqnarray}
    where \eqref{eq:decomposition of cr ops 5} follows by definition of the evaluation map, \eqref{eq:decomposition of cr ops 4} is true since $\mathbf{F}^{Y}_{(j,u)}$ is given by the formula $\mathbf{F}^{Y}_{(j,u)}(y) = J \circ T u \circ y$, \eqref{eq:decomposition of cr ops 2} follows because Diagram \eqref{eq:diag naturality of lcro} commutes, and \eqref{eq:decomposition of cr ops 3} and \eqref{eq:decomposition of cr ops 1} then follow by definition of the linearized Cauchy--Riemann operator. Let $\mathbf{D}^{NN}_\delta$ be the restriction and $\mathbf{D}_0^{NN}$ be the conjugation of $\mathbf{D}^{NN}_{(j,u)}$ (as in Diagram \eqref{eq:restriction conjugation}). Denote by $\mathbf{B}^{NN}_{\gamma}$ the asymptotic operator of $\mathbf{D}^{NN}_{\delta}$ at $z$. Then the asymptotic operator of $\mathbf{D}^{NN}_0$ at $z$ is $\mathbf{B}^{NN}_{\gamma} + \delta$, which by assumption has Conley--Zehnder index equal to $1$. We show that $\operatorname{ind} \mathbf{D}_0^{NN} = 2$.
    \begin{IEEEeqnarray*}{rCls+x*}
        \operatorname{ind} \mathbf{D}_0^{NN}
        & = & \chi(\dot{\Sigma}) + 2 c_1^{\tau}(u^* T \hat{X}) + \conleyzehnder^{\tau}(\mathbf{B}^{NN}_{{\gamma}} + \delta) & \quad [\text{by \cref{thm:riemann roch with punctures}}] \\
        & = & 2                                                                                                             & \quad [\text{since $\conleyzehnder^{\tau}(\mathbf{B}^{NN}_{{\gamma}} + \delta) = 1$}].
    \end{IEEEeqnarray*}
    We prove \ref{lem:DX surj implies DY surj 1}.
    \begin{IEEEeqnarray*}{rCls+x*}
        \operatorname{ind} \mathbf{D}_0^{NN} = 2
        & \Longrightarrow & \mathbf{D}_0^{NN}       \text{ is surjective} & \quad [\text{by \cref{lem:conditions for D surjective genus zero}}] \\
        & \Longrightarrow & \mathbf{D}_\delta^{NN}  \text{ is surjective} & \quad [\text{$\mathbf{D}_0^{NN}$ and $\mathbf{D}_{\delta}^{NN}$ are conjugate}] \\
        & \Longrightarrow & \mathbf{D}_{(j,u)}^{NN} \text{ is surjective} & \quad [\text{$\mathbf{D}_{\delta}^Y$ is a restriction of $\mathbf{D}_{(j,u)}^Y$}] \\
        & \Longrightarrow & \mathbf{L}_{(j,u)}^Y    \text{ is surjective} & \quad [\text{$\mathbf{L}_{(j,u)}^X$ is surjective by assumption}].
    \end{IEEEeqnarray*}
    We prove \ref{lem:DX surj implies DY surj 2}.
    \begin{IEEEeqnarray*}{rCls+x*}
        \IEEEeqnarraymulticol{3}{l}{\operatorname{ind} \mathbf{D}_0^{NN} = 2}\\ \quad
        & \Longrightarrow & \mathbf{D}_0^{NN}       \oplus \mathbf{E}_u^{NN} \text{ is surjective} & \quad [\text{by \cref{lem:D plus E is surjective}}] \\
        & \Longrightarrow & \mathbf{D}_\delta^{NN}  \oplus \mathbf{E}_u^{NN} \text{ is surjective} & \quad [\text{$\mathbf{D}_0^{NN} \oplus \mathbf{E}^{NN}_u$ and $\mathbf{D}_{\delta}^{NN} \oplus \mathbf{E}^{NN}_{u}$ are conjugate}] \\
        & \Longrightarrow & \mathbf{D}_{(j,u)}^{NN} \oplus \mathbf{E}_u^{NN} \text{ is surjective} & \quad [\text{$\mathbf{D}_{\delta}^Y \oplus \mathbf{E}^{Y}_{u}$ is a restriction of $\mathbf{D}_{(j,u)}^Y \oplus \mathbf{E}^{Y}_u$}] \\
        & \Longrightarrow & \mathbf{L}_{(j,u)}^Y    \oplus \mathbf{E}_u^{Y}  \text{ is surjective} & \quad [\text{$\mathbf{L}_{(j,u)}^X \oplus \mathbf{E}_u^{X}$ is surjective by assumption}].                                                & \qedhere
    \end{IEEEeqnarray*}
\end{proof}

\subsection{Moduli spaces of curves in ellipsoids}
\label{sec:augmentation map of an ellipsoid}

We now use the techniques explained in the past two subsections to compute the augmentation map of an ellipsoid (\cref{thm:augmentation is nonzero}). The proof of this theorem consists in an explicit count of curves in the ellipsoid satisfying a tangency constraint (\cref{lem:moduli spaces of ellipsoids have 1 element}) together with the fact that the moduli space of such curves is transversely cut out (\cref{prp:moduli spaces without point constraint are tco,prp:moduli spaces w point are tco,prp:moduli spaces w tangency are tco}). Therefore, the explicit count agrees with the virtual count. We now state the assumptions for this subsection.

Let $a_1 < \cdots < a_n \in \R_{> 0}$ be rationally linearly independent and consider the ellipsoid $E(a_1,\ldots,a_n) \subset \C^n$. By \cite[Section 2.1]{guttSymplecticCapacitiesPositive2018}, $\partial E(a_1, \ldots, a_n)$ has exactly $n$ simple Reeb orbits $\gamma_1, \ldots, \gamma_n$, which satisfy
\begin{IEEEeqnarray}{rCls+x*}
    \gamma_j(t)                & = & \sqrt{\frac{a_j}{\pi}} e^{\frac{2 \pi i t}{a_j}} e_j, \\
    \mathcal{A}(\gamma^m_j)    & = & m a_j, \\
    \conleyzehnder(\gamma^m_j) & = & n - 1 + 2 \sum_{i=1}^{n} \p{L}{2}{\frac{m a_j}{a_i}}, \phantomsection\label{eq:cz of reeb in ellipsoid}
\end{IEEEeqnarray}
where $\gamma_j \colon \R / a_j \Z \longrightarrow \partial E(a_1, \ldots, a_n)$ and $e_j$ is the $j$th vector of the canonical basis of $\C^n$ as a vector space over $\C$. For simplicity, for every $\ell = 1, \ldots, n$ denote $E_\ell = E(a_1,\ldots,a_\ell) \subset \C^\ell$. Notice that $\gamma_1$ is a Reeb orbit of $\partial E_1, \ldots, \partial E_n$. Define maps
\begin{IEEEeqnarray*}{rClCrCl}
    \iota_{\ell} \colon \C^{\ell} & \longrightarrow & \C^{\ell + 1}, & \quad & \iota_\ell(z_1,\ldots,z_\ell) & \coloneqq & (z_1,\ldots,z_\ell,0) \\
    h_{\ell} \colon \C^{\ell}     & \longrightarrow & \C,            & \quad & h_\ell(z_1,\ldots,z_\ell)     & \coloneqq & z_1.
\end{IEEEeqnarray*}
The maps $\iota_{\ell} \colon E_\ell \longrightarrow E_{\ell+1}$ are Liouville embeddings satisfying the assumptions in \cref{sec:functional analytic setup}. Define also
\begin{IEEEeqnarray*}{rCls+x*}
    x_\ell   & \coloneqq & 0 \in \C^\ell, \\
    D_{\ell} & \coloneqq & \{ (z_1,\ldots,z_\ell) \in \C^{\ell} \mid z_1 = 0 \} = h_{\ell}^{-1}(0).
\end{IEEEeqnarray*}
Choose an admissible almost complex structure $J_{\ell} \in \mathcal{J}(E_\ell, D_\ell)$ on $\hat{E}_{\ell}$ such that $J_{\ell}$ is the canonical almost complex structure of $\C^\ell$ near $0$. We assume that the almost complex structures are chosen in such a way that $\hat{\iota}_{\ell} \colon \hat{E}_{\ell} \longrightarrow \hat{E}_{\ell + 1}$ is holomorphic and also such that there exists a biholomorphism $\varphi \colon \hat{E}_1 \longrightarrow \C$ such that $\varphi(z) = z$ for $z$ near $0 \in \C$ (see \cref{lem:biholomorphism explicit} below). Let $m \in \Z_{\geq 1}$ and assume that $m a_1 < a_2 < \cdots < a_n$.

Consider the sphere $S^2$, without any specified almost complex structure, with a puncture $z \in S^2$ and an asymptotic marker $v \in (T_{z} S^2 \setminus \{0\}) / \R_{> 0}$, and also a marked point $w \in \dot{S}^2 = S^2 \setminus \{z\}$. For $k \in \Z_{\geq 0}$, denote%
\begin{IEEEeqnarray*}{lCls+x*}
    \mathcal{M}^{\ell,(k)}_{\mathrm{p}}
    & \coloneqq & \mathcal{M}_{E_{\ell}}^{\$, J_{\ell}}(\gamma^m_1)\p{<}{}{\mathcal{T}^{(k)}x_\ell}_{\mathrm{p}} \\
    & \coloneqq & \left\{
        (j, u)
        \ \middle\vert
        \begin{array}{l}
            j \text{ is an almost complex structure on }S^2, \\
            u \colon (\dot{S}^2, j) \longrightarrow (\hat{E}_\ell, J_\ell) \text{ is as in \cref{def:moduli space with asy markers}}, \\
            u(w) = x_\ell \text{ and $u$ has contact order $k$ to $D_\ell$ at $x_\ell$}
        \end{array}
    \right\}.
\end{IEEEeqnarray*}
Here, the subscript $\mathrm{p}$ means that the moduli space consists of parametrized curves, i.e. we are not quotienting by biholomorphisms. Denote the moduli spaces of regular curves and of unparametrized curves by
\begin{IEEEeqnarray*}{lCls+x*}
    \mathcal{M}^{\ell,(k)}_{\mathrm{p,reg}} & \coloneqq & \mathcal{M}_{E_{\ell}}^{\$, J_{\ell}}(\gamma^m_1)\p{<}{}{\mathcal{T}^{(k)}x_\ell}_{\mathrm{p,reg}}, \\
    \mathcal{M}^{\ell,(k)}                  & \coloneqq & \mathcal{M}_{E_{\ell}}^{\$, J_{\ell}}(\gamma^m_1)\p{<}{}{\mathcal{T}^{(k)}x_\ell} \coloneqq \mathcal{M}^{\ell,(k)}_{\mathrm{p}} / \sim.
\end{IEEEeqnarray*}
Here, $\mathcal{M}^{\ell,(0)} \coloneqq \mathcal{M}_{E_{\ell}}^{\$, J_{\ell}}(\gamma^m_1)\p{<}{}{\mathcal{T}^{(0)}x_\ell} \coloneqq \mathcal{M}_{E_{\ell}}^{\$, J_{\ell}}(\gamma^m_1)$ and analogously for $\mathcal{M}^{\ell,(0)}_{\mathrm{p,reg}}$ and $\mathcal{M}^{\ell,(0)}_{\mathrm{p}}$.

\begin{lemma}
    \phantomsection\label{lem:biholomorphism explicit}
    For any $a > 0$, there exists an almost complex structure $J$ on $\hat{B}(a)$ and a biholomorphism $\varphi \colon \hat{B}(a) \longrightarrow \C$ such that
    \begin{enumerate}
        \item \label{lem:biholomorphism explicit 1} $J$ is cylindrical on $\R_{\geq 0} \times \partial B(a)$;
        \item \label{lem:biholomorphism explicit 2} $J$ is the canonical almost complex structure of $\C$ near $0 \in B(a) \subset \C$;
        \item \label{lem:biholomorphism explicit 3} $\varphi(z) = z$ for $z$ near $0 \in B(a) \subset \C$.
    \end{enumerate}
\end{lemma}
\begin{proof}
    Choose $\rho_0 < 0$ and let $g \colon \R \longrightarrow \R_{>0}$ be a function such that $g(\rho) = a/4 \pi$ for $\rho \leq \rho_0$ and $g(\rho) = 1$ for $\rho \geq 0$. For $(\rho, w) \in \R \times \partial B(a)$, define
    \begin{IEEEeqnarray*}{rCls+x*}
        f(\rho)                         & \coloneqq & \exp \p{}{2}{\frac{\rho_0}{2} + \frac{2 \pi}{a} \int_{\rho_0}^{\rho} g(\sigma) \edv \sigma}, \\
        J_{(\rho, w)} (\partial_{\rho}) & \coloneqq & g (\rho) R^{\partial B(a)}_{w}, \\
        \varphi(\rho, w)                & \coloneqq & f(\rho) w.
    \end{IEEEeqnarray*}
    Property \ref{lem:biholomorphism explicit 1} follows from the fact that $g(\rho) = 1$ for $\rho \geq 0$. Consider the Liouville vector field of $\C$, which is denoted by $Z$ and given by $Z(w) = w/2$. Let $\Phi \colon \R \times \partial B(a) \longrightarrow \C$ be the map given by $\Phi(\rho, w) = \phi^\rho_Z(w) = \exp(\rho/2) w$. By definition of completion, $\Phi|_{B(a) \setminus \{0\}} \colon B(a) \setminus \{0\} \longrightarrow \C$ is the inclusion. To prove property \ref{lem:biholomorphism explicit 3}, it suffices to show that $\varphi(\rho, w) = \Phi(\rho, w)$ for every $(\rho, w) \in \R_{\leq \rho_0} \times \partial B(a)$. For this, simply note that
    \begin{IEEEeqnarray*}{rCls+x*}
        f(\rho)
        & = & \exp \p{}{2}{\frac{\rho_0}{2} + \frac{2 \pi}{a} \int_{\rho_0}^{\rho} g(\sigma) \edv \sigma} & \quad [\text{by definition of $f$}] \\
        & = & \exp \p{}{2}{\frac{\rho_0}{2} + \frac{2 \pi}{a} (\rho - \rho_0) \frac{a}{4 \pi} }           & \quad [\text{$\rho \leq \rho_0$ implies $g(\rho) = a / 4 \pi$}] \\
        & = & \exp \p{}{2}{\frac{\rho}{2}}.
    \end{IEEEeqnarray*}
    Therefore, $\varphi(z) = z$ for $z$ near $0 \in B(a) \subset \C$, and in particular $\varphi$ can be extended smoothly to a map $\varphi \colon \hat{B}(a) \longrightarrow \C$. We show that $\varphi$ is holomorphic.
    \begin{IEEEeqnarray*}{rCls+x*}
        j \circ \dv \varphi(\rho, w) (\partial_{\rho})
        & = & j \p{}{2}{\pdv{}{\rho} \p{}{1}{f(\rho) |w|} \pdv{}{r}\Big|_{\varphi(\rho, w)}}                & \quad [\text{by definition of $\varphi$}] \\
        & = & \frac{2 \pi}{a} \, g(\rho) \, j \p{}{2}{ f(\rho) |w| \pdv{}{r}\Big|_{\varphi(\rho, w)}}       & \quad [\text{by definition of $f$}] \\
        & = & \frac{2 \pi}{a} \, g(\rho) \, j \p{}{2}{ |\varphi(\rho,w)| \pdv{}{r}\Big|_{\varphi(\rho, w)}} & \quad [\text{by definition of $\varphi$}] \\
        & = & \frac{2 \pi}{a} \, g(\rho) \, \pdv{}{\theta}\Big|_{\varphi(\rho, w)}                          & \quad [\text{by definition of $j$}] \\
        & = & g(\rho) \, \dv \varphi(\rho, w) (R^{\partial B(a)}_w)                                         & \quad [\text{by \cite[Equation (2.2)]{guttSymplecticCapacitiesPositive2018}}] \\
        & = & \dv \varphi(\rho, w) \circ J (\partial_{\rho})                                                & \quad [\text{by definition of $J$}],
    \end{IEEEeqnarray*}
    Where $(r, \theta)$ are the polar coordinates of $\C$. Since $\varphi$ is holomorphic and $\varphi$ is the identity near the origin, we conclude that $J$ is the canonical almost complex structure of $\C$ near the origin. In particular, $J$ can be extended smoothly to an almost complex structure on $\hat{B}(a)$, which proves \ref{lem:biholomorphism explicit 2}. Finally, we show that $\varphi$ is a diffeomorphism. For this, it suffices to show that $\Phi^{-1} \circ \varphi \colon \R \times \partial B(a) \longrightarrow \R \times \partial B(a)$ is a diffeomorphism. This map is given by $\Phi^{-1} \circ \varphi(\rho, w) = (2 \ln(f(\rho)), w)$. Since
    \begin{IEEEeqnarray*}{c+x*}
        \odv{}{\rho} (2 \ln(f(\rho))) = 2 \frac{f'(\rho)}{f(\rho)} = \frac{4 \pi}{a} g(\rho) > 0,
    \end{IEEEeqnarray*}
    $\varphi$ is a diffeomorphism.
\end{proof}

\begin{lemma}
    \label{lem:psi j}
    Let $\operatorname{inv} \colon \overline{\C} \longrightarrow \overline{\C}$ be the map given by $\operatorname{inv}(z) = 1/z$ and consider the vector $V \coloneqq \dv \operatorname{inv}(0) \partial_x \in T_{\infty} \overline{\C}$. For every $j \in \mathcal{T}$ there exists a unique biholomorphism $\psi_j \colon (\overline{\C}, j_0) \longrightarrow (S^2, j)$ such that
    \begin{IEEEeqnarray*}{c+x*}
        \psi_j(0) = w, \qquad \psi_j(\infty) = z, \qquad \dv \psi_j(\infty) V = \frac{v}{\| v \|},
    \end{IEEEeqnarray*}
    where $\| \cdot \|$ is the norm coming from the canonical Riemannian metric on $S^2$ as the sphere of radius $1$ in $\R^3$.
\end{lemma}
\begin{proof}
    By the uniformization theorem \cite[Theorem XII.0.1]{desaint-gervaisUniformizationRiemannSurfaces2016}, there exists a biholomorphism $\phi \colon (S^2, j) \longrightarrow (\overline{\C}, j_0)$. Since there exists a unique Möbius transformation $\psi_0 \colon (\overline{\C}, j_0) \longrightarrow (\overline{\C}, j_0)$ such that
    \begin{IEEEeqnarray*}{c+x*}
        \psi_0(0) = \phi(w), \qquad \psi_0(\infty) = \phi(z), \qquad \dv \psi_0 (\infty) V = \dv \phi(z) \frac{v}{\| v \|},
    \end{IEEEeqnarray*}
    the result follows.
\end{proof}

We will denote also by $\psi_j$ the restriction $\psi_j \colon (\C, j_0) \longrightarrow (S^2, j)$.

\begin{lemma}
    \phantomsection\label{lem:u is a polynomial}
    If $(j,u) \in \mathcal{M}^{1,(0)}$ then $\varphi \circ u \circ \psi_j \colon \C \longrightarrow \C$ is a polynomial of degree $m$.
\end{lemma}
\begin{proof}
    Since $u$ is positively asymptotic to $\gamma^m_1$, the map $\varphi \circ u \circ \psi_j$ goes to $\infty$ as $z$ goes to $\infty$. Therefore, $\varphi \circ u \circ \psi_j$ is a polynomial. Again using the fact that $u$ is positively asymptotic to $\gamma^m_1$, we conclude that for $r$ big enough the path $\theta \longmapsto \varphi \circ u \circ \psi_j(r e^{i \theta})$ winds around the origin $m$ times. This implies that the degree of $\varphi \circ u \circ \psi_j$ is $m$.
\end{proof}

We now with to compute the normal Conley--Zehnder index of $\gamma_1^m$. For this, we will use the following result.

\begin{proposition}[{\cite[Proposition 41]{guttConleyZehnderIndex2012}}]
    \label{prp:gutts cz formula}
    Let $S$ be a symmetric, nondegenerate $2 \times 2$-matrix and $T > 0$ be such that $\exp(T J_0 S) \neq I$. Consider the path of symplectic matrices $A \colon [0,T] \longrightarrow \operatorname{Sp}(2)$ given by
    \begin{IEEEeqnarray*}{c+x*}
        A(t) \coloneqq \exp(t J_0 S).
    \end{IEEEeqnarray*}
    Let $a_1$ and $a_2$ be the eigenvalues of $S$ and $\signature S$ be its signature. Then,
    \begin{IEEEeqnarray*}{c+x*}
        \conleyzehnder(A) =
        \begin{cases}
            \p{}{1}{\frac{1}{2} + \p{L}{1}{\frac{\sqrt{a_1 a_2} T}{2 \pi}}} \signature S & \text{if } \signature S \neq 0, \\
            0                                                                            & \text{if } \signature S = 0.
        \end{cases}
    \end{IEEEeqnarray*}
\end{proposition}

\begin{lemma}
    \label{lem:normal cz is one}
    For every $\ell = 1,\ldots,n-1$, view $\gamma^m_1$ as a Reeb orbit of $\partial E_{\ell} \subset \partial E_{\ell + 1}$. The normal Conley--Zehnder index of $\gamma^m_1$ is $1$.
\end{lemma}
\begin{proof}
    By \cite[Equation (2.2)]{guttSymplecticCapacitiesPositive2018}, the Reeb vector field of $\partial E_{\ell + 1}$ is given by
    \begin{IEEEeqnarray*}{c+x*}
        R^{\partial E_{\ell + 1}} = 2 \pi \sum_{j=1}^{\ell+1} \frac{1}{a_j} \pdv{}{\theta_{j}},
    \end{IEEEeqnarray*}
    where $\theta_j$ denotes the angular polar coordinate of the $j$th summand of $\C^{\ell+1}$. Therefore, the flow of $R^{\partial E_{\ell + 1}}$ is given by
    \begin{IEEEeqnarray*}{rrCl}
        \phi^{t}_{R} \colon & \partial E_{\ell+1}     & \longrightarrow & \partial E_{\ell+1} \\
                                                  & (z_1,\ldots,z_{\ell+1}) & \longmapsto     & \p{}{2}{e^{\frac{2 \pi i}{a_1}} z_1, \ldots, e^{\frac{2 \pi i}{a_{\ell+1}}} z_{\ell+1}}.
    \end{IEEEeqnarray*}
    The diagram
    \begin{IEEEeqnarray*}{c+x*}
        \begin{tikzcd}
            \xi^{\partial E_{\ell}}_{\gamma^m_1(0)} \ar[r] \ar[d, swap, "\dv \phi^t_{R}(\gamma^m_1(0))"] & \xi^{\partial E_{\ell+1}}_{\gamma^m_1(0)} \ar[d, "\dv \phi^t_{R}(\gamma^m_1(0))"] & \big(\xi^{\partial E_{\ell+1}}_{\gamma^m_1(0)}\big)^{\perp} \ar[l] \ar[d, "\dv \phi^t_{R}(\gamma^m_1(0))"] \ar[r, equals] & \C \ar[d, "\times \exp \p{}{1}{\frac{2 \pi i t}{a_{\ell+1}}}"] \\
            \xi^{\partial E_{\ell}}_{\gamma^m_1(t)} \ar[r]                                               & \xi^{\partial E_{\ell+1}}_{\gamma^m_1(t)}                                         & \big(\xi^{\partial E_{\ell+1}}_{\gamma^m_1(t)}\big)^{\perp} \ar[l] \ar[r, equals]                                         & \C
        \end{tikzcd}
    \end{IEEEeqnarray*}
    commutes. Define a path $A_{\gamma^m_1} \colon [0,m a_1] \longrightarrow \operatorname{Sp}(2)$ by $A_{\gamma^m_1}(t) = \exp (t J_0 S)$, where
    \begin{IEEEeqnarray*}{c+x*}
        S = \frac{2 \pi}{a_{\ell + 1}}
        \begin{bmatrix}
            1 & 0 \\
            0 & 1
        \end{bmatrix}.
    \end{IEEEeqnarray*}
    The only eigenvalue of $S$ is $2 \pi / a_{\ell+1}$, which has multiplicity $2$. Therefore, the signature of $S$ is $\signature S = 2$. These facts allow us to compute $\conleyzehnder^{\perp}(\gamma^m_1)$ using \cref{prp:gutts cz formula}:
    \begin{IEEEeqnarray*}{rCls+x*}
        \conleyzehnder^{\perp}(\gamma^m_1)
        & = & \conleyzehnder(A_{\gamma^m_1})                                                                                                & \quad [\text{by definition of $\conleyzehnder^{\perp}$}] \\
        & = & \p{}{2}{\frac{1}{2} + \p{L}{2}{\sqrt{\frac{2 \pi}{a_{\ell + 1}}\frac{2 \pi}{a_{\ell + 1}}} \frac{m a_1}{2 \pi}}} \signature S & \quad [\text{by \cref{prp:gutts cz formula}}] \\
        & = & \frac{1}{2} \signature S                                                                                                      & \quad [\text{since $m a_1 < a_2 < \cdots < a_n$}] \\
        & = & 1                                                                                                                             & \quad [\text{by the discussion above}].                    & \qedhere
    \end{IEEEeqnarray*}
\end{proof}

\begin{lemma}
    \label{lem:unique reeb orbit with cz equal to}
    If $\ell = 1,\ldots,n$ then $\gamma^m_1$ is the unique Reeb orbit of $\partial E_{\ell}$ such that $\conleyzehnder(\gamma^m_1) = \ell - 1 + 2m$.
\end{lemma}
\begin{proof}
    First, notice that
    \begin{IEEEeqnarray*}{rCls+x*}
        \conleyzehnder(\gamma^m_1)
        & = & \ell - 1 + 2 \sum_{j=1}^{\ell} \p{L}{2}{\frac{m a_1}{a_j}} & \quad [\text{by equation \eqref{eq:cz of reeb in ellipsoid}}] \\
        & = & \ell - 1 + 2 m                                          & \quad [\text{since $m a_1 < a_2 < \cdots < a_n$}].
    \end{IEEEeqnarray*}
    Conversely, let $\gamma = \gamma^k_i$ be a Reeb orbit of $\partial E_\ell$ with $\conleyzehnder(\gamma) = \ell - 1 + 2m$. By equation \eqref{eq:cz of reeb in ellipsoid}, this implies that
    \begin{IEEEeqnarray}{c+x*}
        \label{eq:k is sum of floors}
        m = \sum_{j=1}^{\ell} \p{L}{2}{\frac{k a_i}{a_j}}.
    \end{IEEEeqnarray}
    We show that $i = 1$. Assume by contradiction otherwise. Then
    \begin{IEEEeqnarray*}{rCls+x*}
        m
        & =    & \sum_{1 \leq j \leq \ell} \p{L}{2}{\frac{k a_i}{a_j}}                          & \quad [\text{by equation \eqref{eq:k is sum of floors}}] \\
        & \geq & \sum_{1 \leq j \leq i} \p{L}{2}{\frac{k a_i}{a_j}}                             & \quad [\text{since every term in the sum is $\geq 0$}] \\
        & =    & \p{L}{2}{\frac{k a_i}{a_1}} + \sum_{1 < j < i} \p{L}{2}{\frac{k a_i}{a_j}} + k & \quad [\text{since by assumption, $i > 1$}] \\
        & \geq & (m + i - 1) k                                                                  & \quad [\text{$m a_1 < a_2 < \cdots < a_i$}] \\
        & >    & m k                                                                            & \quad [\text{since by assumption, $i > 1$}],
    \end{IEEEeqnarray*}
    which is a contradiction, and therefore $i = 1$. We show that $k = m$, using the fact that $m \geq \lfloor k a_i / a_1 \rfloor = k$.
    \begin{IEEEeqnarray*}{rCls+x*}
        m
        & = & \sum_{1 \leq j \leq \ell} \p{L}{2}{\frac{k a_1}{a_j}}     & \quad [\text{by equation \eqref{eq:k is sum of floors} and since $i = 1$}] \\
        & = & k + \sum_{2 \leq j \leq \ell} \p{L}{2}{\frac{k a_1}{a_j}} & \\
        & = & k                                                         & \quad [\text{since $k \leq m$ and $k a_1 \leq m a_1 < a_1 < \cdots < a_n$}]. & \qedhere
    \end{IEEEeqnarray*}
\end{proof}

Using the previous results, we can now compute the linearized contact homology of $E_n$.

\begin{lemma}
    \label{lem:lch of ellipsoid}
    The module $CH_{n - 1 + 2m}(E_n)$ is the free $\Q$-module generated by $\gamma^m_1$.
\end{lemma}
\begin{proof}
    By equation \eqref{eq:cz of reeb in ellipsoid}, every Reeb orbit of $\partial E_n$ is good. We claim that the differential $\partial \colon CC(E_n) \longrightarrow CC(E_n)$ is zero. Assume by contradiction that there exists a Reeb orbit $\gamma$ such that $\partial \gamma \neq 0$. By definition of $\partial$, this implies that there exist Reeb orbits $\eta, \alpha_1, \ldots, \alpha_p$ such that
    \begin{IEEEeqnarray*}{rCls+x*}
        0 & \neq & \#^{\mathrm{vir}} \overline{\mathcal{M}}^{J_n}_{\partial E_n}(\gamma; \eta, \alpha_1, \ldots, \alpha_p), \\
        0 & \neq & \#^{\mathrm{vir}} \overline{\mathcal{M}}^{J_n}_{E_n}(\alpha_j), \quad \text{for } j=1,\ldots,p.
    \end{IEEEeqnarray*}
    By assumption on the virtual perturbation scheme,
    \begin{IEEEeqnarray*}{rCls+x*}
        0 & =   & \operatorname{virdim} \overline{\mathcal{M}}^{J_n}_{E_n}(\alpha_j) = n - 3 + \conleyzehnder(\alpha_j) \quad \text{for every } j = 1,\ldots,p, \\ \\
        0 & =   & \operatorname{virdim} \overline{\mathcal{M}}^{J_n}_{\partial E_n}(\gamma; \eta, \alpha_1, \ldots, \alpha_p) \\
          & =   & (n-3)(2 - (2+p)) + \conleyzehnder(\gamma) - \conleyzehnder(\eta) - \sum_{j=1}^{p} \conleyzehnder(\alpha_j) - 1 \\
          & =   & \conleyzehnder(\gamma) - \conleyzehnder(\eta) - 1 \\
          & \in & 1 + 2 \Z,
    \end{IEEEeqnarray*}
    where in the last line we used equation \eqref{eq:cz of reeb in ellipsoid}. This gives the desired contradiction, and we conclude that $\partial \colon CC(E_n) \longrightarrow CC(E_n)$ is zero. Therefore, $CH(E_n) = CC(E_n)$ is the free $\Q$-module generated by the Reeb orbits of $\partial E_n$. By \cref{lem:unique reeb orbit with cz equal to}, $\gamma^m_1$ is the unique Reeb orbit of $\partial E_n$ with $\conleyzehnder(\gamma^m_1) = n - 1 + 2m$, from which the result follows.
\end{proof}

\begin{lemma}
    \phantomsection\label{lem:moduli spaces of ellipsoids are all equal}
    If $\ell = 1,\ldots,n$ and $k \in \Z_{\geq 1}$ then $\mathcal{M}^{\ell,(k)}_{\mathrm{p}} = \mathcal{M}^{1,(k)}_{\mathrm{p}}$ and $\mathcal{M}^{\ell,(k)} = \mathcal{M}^{1,(k)}$.
\end{lemma}
\begin{proof}
    It suffices to show that $\mathcal{M}^{\ell,(k)}_{\mathrm{p}} = \mathcal{M}^{\ell+1,(k)}_{\mathrm{p}}$ for every $\ell = 1,\ldots,n-1$. The inclusion $\mathcal{M}^{\ell,(k)}_{\mathrm{p}} \subset \mathcal{M}^{\ell+1,(k)}_{\mathrm{p}}$ follows from the fact that the inclusion $\hat{E}_\ell \hookrightarrow \hat{E}_{\ell+1}$ is holomorphic and the assumptions on the symplectic divisors. To prove that $\mathcal{M}^{\ell+1,(k)}_{\mathrm{p}} \subset \mathcal{M}^{\ell,(k)}_{\mathrm{p}}$, it suffices to assume that $(j,u) \in \mathcal{M}^{\ell+1,(k)}_{\mathrm{p}}$ and to show that the image of $u$ is contained in $\hat{E}_\ell \subset \hat{E}_{\ell+1}$. Since $u$ has contact order $k$ to $D_{\ell+1}$ at $x_{\ell+1} = \iota_{\ell}(x_{\ell})$, we conclude that $u$ is not disjoint from $\hat{E}_\ell$. By \cref{lem:stabilization 2}, $u$ is contained in $\hat{E}_\ell$.
\end{proof}

We now prove that the moduli spaces $\mathcal{M}^{\ell,(k)}$ are regular. The proof strategy is as follows.
\begin{enumerate}
    \item \cref{prp:moduli spaces without point constraint are tco} deals with the moduli spaces $\mathcal{M}^{1,(0)}$. We show that the linearized Cauchy--Riemann operator is surjective using \cref{lem:Du is surjective case n is 1}.
    \item \cref{prp:moduli spaces w point are tco} deals with the moduli spaces $\mathcal{M}^{\ell,(1)}$. Here, we need to consider the linearized Cauchy--Riemann operator together with an evaluation map. We show inductively that this map is surjective using \cref{lem:DX surj implies DY surj}.
    \item Finally, \cref{prp:moduli spaces w tangency are tco} deals with the moduli spaces $\mathcal{M}^{\ell,(k)}$. We now need to consider the jet evaluation map. We prove inductively that this map is surjective by writing it explicitly.
\end{enumerate}

\begin{proposition}
    \label{prp:moduli spaces without point constraint are tco}
    The moduli spaces $\mathcal{M}^{1,(0)}_{\mathrm{p}}$ and $\mathcal{M}^{1,(0)}$ are transversely cut out.
\end{proposition}
\begin{proof}
    It is enough to show that $\mathcal{M}^{1,(0)}_{\mathrm{p}}$ is transversely cut out, since this implies that $\mathcal{M}^{1,(0)}$ is transversely cut out as well. Recall that $\mathcal{M}^{1,(0)}_{\mathrm{p}}$ can be written as the zero set of the Cauchy--Riemann operator $\overline{\partial}\vphantom{\partial}^{1} \colon \mathcal{T} \times \mathcal{B} E_{1} \longrightarrow \mathcal{E} E_{1}$. It suffices to assume that $(j,u) \in (\overline{\partial}\vphantom{\partial}^{1})^{-1}(0)$ and to prove that the linearization
    \begin{IEEEeqnarray*}{c+x*}
        \mathbf{L}_{(j,u)}^1 \colon T_j \mathcal{T} \oplus T_u \mathcal{B} E_1 \longrightarrow \mathcal{E}_{(j,u)} E_1
    \end{IEEEeqnarray*}
    is surjective. This follows from \cref{lem:Du is surjective case n is 1}.
\end{proof}

\begin{proposition}
    \label{prp:moduli spaces w point are tco}
    If $\ell = 1,\ldots,n$ then $\mathcal{M}^{\ell,(1)}_{\mathrm{p}}$ and $\mathcal{M}^{\ell,(1)}$ are transversely cut out.
\end{proposition}
\begin{proof}
    We will use the notation of \cref{sec:functional analytic setup} with $X = E_{\ell}$ and $Y = E_{\ell + 1}$. We will show by induction on $\ell$ that $\mathcal{M}^{\ell,(1)}_{\mathrm{p}}$ is transversely cut out. This implies that $\mathcal{M}^{\ell,(1)}$ is transversely cut out as well.

    We prove the base case. By \cref{prp:moduli spaces without point constraint are tco}, $\mathcal{M}^{1,(0)}_{\mathrm{p}}$ is a smooth manifold. Consider the evaluation map
    \begin{IEEEeqnarray*}{rrCl}
        \operatorname{ev}^{1} \colon & \mathcal{M}^{1,(0)}_{\mathrm{p}} & \longrightarrow & \hat{E}_1 \\
                                     & (j,u)                            & \longmapsto     & u(w).
    \end{IEEEeqnarray*}
    Notice that $\mathcal{M}^{1,(1)}_{\mathrm{p}} = (\operatorname{ev}^1)^{-1}(x_1)$. We wish to show that the linearized evaluation map $\mathbf{E}^1_{(j,u)} = \dv (\operatorname{ev}^1)(j,u) \colon T_{(j,u)} \mathcal{M}^{1,(0)}_{\mathrm{p}} \longrightarrow T_{u(w)} \hat{E}_1$ is surjective whenever $u(w) = \operatorname{ev}^{1}(j,u) = x_1$. There are commutative diagrams
    \begin{IEEEeqnarray*}{c+x*}
        \begin{tikzcd}
            \mathcal{M}^{1,(0)}_{\mathrm{p}} \ar[r, two heads, "\Phi"] \ar[d, swap, "\operatorname{ev}^1"] & \mathcal{M} \ar[d, "\operatorname{ev}_{\mathcal{M}}"] & \mathcal{C} \ar[l, swap, hook', two heads, "\mathcal{P}"] \ar[d, "\operatorname{ev}_{\mathcal{C}}"] & & T_{(j,u)} \mathcal{M}^{1,(0)}_{\mathrm{p}} \ar[r, two heads, "{\dv \Phi(j,u)}"] \ar[d, swap, "{\mathbf{E}^1_{(j,u)}}"] & T_f \mathcal{M} \ar[d, "\mathbf{E}_{\mathcal{M}}"] & \C^{m+1} \ar[l, swap, hook', two heads, "\dv \mathcal{P}(a)"] \ar[d, "\mathbf{E}_{\mathcal{C}}"] \\
            \hat{E}_1 \ar[r, hook, two heads, swap, "\varphi"]                                             & \C \ar[r, equals]                                     & \C                                                                                                  & & T_{x_1} \hat{E}_1 \ar[r, hook, two heads, swap, "\dv \varphi(x_1)"]                                                    & \C \ar[r, equals]                                  & \C
        \end{tikzcd}
    \end{IEEEeqnarray*}
    where
    \begin{IEEEeqnarray*}{rCls+x*}
        \mathcal{M}                                     & \coloneqq & \{f \colon \C \longrightarrow \C \mid f \text{ is a polynomial of degree }m \}, \\
        \mathcal{C}                                     & \coloneqq & \{(a_0,\ldots,a_m) \in \C^{m+1} \mid a_m \neq 0\}, \\
        \Phi(j,u)                                       & \coloneqq & \varphi \circ u \circ \psi_j, \\
        \operatorname{ev}_{\mathcal{M}}(f)              & \coloneqq & f(0), \\
        \operatorname{ev}_{\mathcal{C}}(a_0,\ldots,a_m) & \coloneqq & a_0, \\
        \mathcal{P}(a_0,\ldots,a_m)(z)                  & \coloneqq & a_0 + a_1 z + \cdots + a_m z^m,
    \end{IEEEeqnarray*}
    and the diagram on the right is obtained by linearizing the one on the left. The map $\Phi$ is well-defined by \cref{lem:u is a polynomial}. Since $\mathbf{E}_{\mathcal{C}}(a_0,\ldots,a_m) = a_0$ is surjective, $\mathbf{E}^1_u$ is surjective as well. This finishes the proof of the base case.

    We prove the induction step, i.e. that if $\mathcal{M}^{\ell,(1)}_p$ is transversely cut out then so is $\mathcal{M}^{\ell+1,(1)}_p$. We prove that $\mathcal{M}^{\ell,(1)}_{\mathrm{p,reg}} \subset \mathcal{M}^{\ell+1,(1)}_{\mathrm{p,reg}}$. For this, assume that $(j,u) \in \mathcal{M}^{\ell,(1)}_{\mathrm{p}}$ is such that $\mathbf{L}_{(j,u)}^\ell \oplus \mathbf{E}_u^\ell \colon T_j \mathcal{T} \oplus T_{u} \mathcal{B} E_\ell \longrightarrow \mathcal{E}_{(j,u)} E_\ell \oplus T_{x_\ell} \hat{E}_\ell$ is surjective. By \cref{lem:DX surj implies DY surj},
    \begin{IEEEeqnarray*}{c+x*}
        \mathbf{L}_{(j,u)}^{\ell+1} \oplus \mathbf{E}_u^{\ell+1} \colon T_j \mathcal{T} \oplus T_{u} \mathcal{B} E_{\ell+1} \longrightarrow \mathcal{E}_{(j,u)} E_{\ell+1} \oplus T_{x_{\ell+1}} \hat{E}_{\ell+1}
    \end{IEEEeqnarray*}
    is also surjective, which means that $(j,u) \in \mathcal{M}^{\ell+1,(1)}_{\mathrm{p,reg}}$. This concludes the proof of $\mathcal{M}^{\ell,(1)}_{\mathrm{p,reg}} \subset \mathcal{M}^{\ell+1,(1)}_{\mathrm{p,reg}}$. Finally, we show that $\mathcal{M}^{\ell+1,(1)}_{\mathrm{p,reg}} = \mathcal{M}^{\ell+1,(1)}_{\mathrm{p}}$.
    \begin{IEEEeqnarray*}{rCls+x*}
        \mathcal{M}^{\ell+1,(1)}_{\mathrm{p,reg}}
        & \subset & \mathcal{M}^{\ell+1,(1)}_{\mathrm{p}}     & \quad [\text{since regular curves form a subset}] \\
        & =       & \mathcal{M}^{\ell,(1)}_{\mathrm{p}}       & \quad [\text{by \cref{lem:moduli spaces of ellipsoids are all equal}}] \\
        & =       & \mathcal{M}^{\ell,(1)}_{\mathrm{p,reg}}   & \quad [\text{by the induction hypothesis}] \\
        & \subset & \mathcal{M}^{\ell+1,(1)}_{\mathrm{p,reg}} & \quad [\text{proven above}].                                             & \qedhere
    \end{IEEEeqnarray*}
\end{proof}

% One possible explanation for the weirdness with m = k + 1. If m = k + 1, then the polynomial has contact order k + 1, which means that the ith = 0,\ldots,k derivative at 0 is 0. This implies that the polynomial is constant equal to zero. So the proof is not supposed to work in that case. "Dimension mismatch" -> take into account that one needs to quotient by group action.
\begin{proposition}
    \label{prp:moduli spaces w tangency are tco}
    If $\ell = 1,\ldots, n$ and $k = 1,\ldots,m$ then $\mathcal{M}^{\ell,(k)}_{\mathrm{p}}$ and $\mathcal{M}^{\ell,(k)}$ are transversely cut out.
\end{proposition}
\begin{proof}
    By \cref{prp:moduli spaces w point are tco}, $\mathcal{M}^{\ell,(1)}_{\mathrm{p}}$ is a smooth manifold. Consider the jet evaluation map
    \begin{IEEEeqnarray*}{rrCl}
        j^{\ell,(k)} \colon & \mathcal{M}^{\ell,(1)}_{\mathrm{p}} & \longrightarrow & \C^{k-1} \\
                            & (j,u)                               & \longmapsto     & ((h_{\ell} \circ u \circ \psi_j)^{(1)}(0), \ldots, (h_{\ell} \circ u \circ \psi_j)^{(k-1)}(0)).
    \end{IEEEeqnarray*}
    The moduli space $\mathcal{M}^{\ell,(k)}_{\mathrm{p}}$ is given by $\mathcal{M}^{\ell,(k)}_{\mathrm{p}} = (j^{\ell,(k)})^{-1}(0)$. We will prove by induction on $\ell$ that $\mathcal{M}^{\ell,(k)}_{\mathrm{p}}$ is transversely cut out. This shows that $\mathcal{M}^{\ell,(k)}$ is transversely cut out as well. Define $\mathbf{J}^{\ell,(k)}_{(j,u)} \coloneqq \dv(j^{\ell,(k)})(j,u) \colon T_{(j,u)} \mathcal{M}^{\ell,(1)}_{\mathrm{p}} \longrightarrow \C^{k-1}$.

    We prove the base case, i.e. that $\mathcal{M}^{1,(k)}_{\mathrm{p}}$ is transversely cut out. For this, it suffices to assume that $(j,u) \in \mathcal{M}^{1,(1)}_{\mathrm{p}}$ is such that $j^{1,(k)}(j,u) = 0$ and to prove that $\mathbf{J}^{1,(k)}_{(j,u)}$ is surjective. There are commutative diagrams
    \begin{IEEEeqnarray*}{c+x*}
        \begin{tikzcd}
            \mathcal{M}^{1,(1)}_{\mathrm{p}} \ar[r, two heads, "\Phi"] \ar[d, swap, "j^{1,(k)}"] & \mathcal{M} \ar[d, "j^{(k)}_{\mathcal{M}}"] & \mathcal{C} \ar[l, swap, hook', two heads, "\mathcal{P}"] \ar[d, "j^{(k)}_{\mathcal{C}}"] & & T_{(j,u)} \mathcal{M}^{1,(1)}_{\mathrm{p}} \ar[r, two heads, "{\dv \Phi(j,u)}"] \ar[d, swap, "{\mathbf{J}^{1,(k)}_{(j,u)}}"] & T_f \mathcal{M} \ar[d, "\mathbf{J}^{(k)}_{\mathcal{M}}"] & \C^{m} \ar[l, swap, hook', two heads, "\dv \mathcal{P}(a)"] \ar[d, "\mathbf{J}^{(k)}_{\mathcal{C}}"] \\
            \C^{k-1} \ar[r, equals]                                                              & \C^{k-1} \ar[r, equals]                     & \C^{k-1}                                                                                  & & \C^{k-1} \ar[r, equals]                                                                                                      & \C^{k-1} \ar[r, equals]                                  & \C^{k-1}
        \end{tikzcd}
    \end{IEEEeqnarray*}
    where
    \begin{IEEEeqnarray*}{rCls+x*}
        \mathcal{M}                                     & \coloneqq & \{f \colon \C \longrightarrow \C \mid f \text{ is a polynomial of degree }m \text{ with }f(0)=0 \}, \\
        \mathcal{C}                                     & \coloneqq & \{(a_1,\ldots,a_m) \in \C^{m} \mid a_m \neq 0\}, \\
        \Phi(j,u)                                       & \coloneqq & \varphi \circ u \circ \psi_j, \\
        j^{(k)}_{\mathcal{M}}(f)                        & \coloneqq & (f^{(1)}(0),\ldots,f^{(k-1)}(0)), \\
        j^{(k)}_{\mathcal{C}}(a_1,\ldots,a_m)           & \coloneqq & (a_1,\ldots,(k-1)! a_{k-1}), \\
        \mathcal{P}(a_1,\ldots,a_m)(z)                  & \coloneqq & a_1 z + \cdots + a_m z^m,
    \end{IEEEeqnarray*}
    and the diagram on the right is obtained by linearizing the one on the left. The map $\Phi$ is well-defined by \cref{lem:u is a polynomial}. Since $\mathbf{J}^{(k)}_{\mathcal{C}}(a_1,\ldots,a_m) = (a_1,\ldots,(k-1)! a_{k-1})$ is surjective, $\mathbf{J}^{1,(k)}_u$ is surjective as well. This finishes the proof of the base case.

    We prove the induction step, i.e. that if $\mathcal{M}^{\ell,(k)}_{\mathrm{p}}$ is transversely cut out then so is $\mathcal{M}^{\ell+1,(k)}_{\mathrm{p}}$. We show that $\mathcal{M}^{\ell,(k)}_{\mathrm{p,reg}} \subset \mathcal{M}^{\ell+1,(k)}_{\mathrm{p,reg}}$. For this, it suffices to assume that $(j,u) \in \mathcal{M}^{\ell,(k)}_{\mathrm{p}}$ is such that $\mathbf{J}^{\ell,(k)}_{(j,u)}$ is surjective, and to prove that $\mathbf{J}^{\ell+1,(k)}_{(j,u)}$ is surjective as well. This follows because the diagrams
    \begin{IEEEeqnarray*}{c+x*}
        \begin{tikzcd}
            \mathcal{M}^{\ell,(1)}_{\mathrm{p}} \ar[d] \ar[dr, "j^{\ell,(k)}"]   &          & & T_{(j,u)} \mathcal{M}^{\ell,(1)}_{\mathrm{p}} \ar[d] \ar[dr, "\mathbf{J}^{\ell,(k)}_u"] \\
            \mathcal{M}^{\ell+1,(1)}_{\mathrm{p}} \ar[r, swap, "j^{\ell+1,(k)}"] & \C^{k-1} & & T_{(j,u)} \mathcal{M}^{\ell+1,(1)}_{\mathrm{p}} \ar[r, swap, "\mathbf{J}_u^{\ell+1,(k)}"] & \C^{k-1}
        \end{tikzcd}
    \end{IEEEeqnarray*}
    commute. Finally, we show that $\mathcal{M}^{\ell+1,(k)}_{\mathrm{p,reg}} = \mathcal{M}^{\ell+1,(k)}_{\mathrm{p}}$.
    \begin{IEEEeqnarray*}{rCls+x*}
        \mathcal{M}^{\ell+1,(k)}_{\mathrm{p,reg}}
        & \subset & \mathcal{M}^{\ell+1,(k)}_{\mathrm{p}}     & \quad [\text{since regular curves form a subset}] \\
        & =       & \mathcal{M}^{\ell,(k)}_{\mathrm{p}}       & \quad [\text{by \cref{lem:moduli spaces of ellipsoids are all equal}}] \\
        & =       & \mathcal{M}^{\ell,(k)}_{\mathrm{p,reg}}   & \quad [\text{by the induction hypothesis}] \\
        & \subset & \mathcal{M}^{\ell+1,(k)}_{\mathrm{p,reg}} & \quad [\text{proven above}].                                             & \qedhere
    \end{IEEEeqnarray*}
\end{proof}

\begin{proposition}
    \label{lem:moduli spaces of ellipsoids have 1 element}
    If $\ell = 1,\ldots,n$ then $\#^{\mathrm{vir}} \overline{\mathcal{M}}^{\ell,(m)} = \# \overline{\mathcal{M}}^{\ell,(m)} = 1$.
\end{proposition}
\begin{proof}
    By assumption on the perturbation scheme and \cref{prp:moduli spaces w tangency are tco}, $\#^{\mathrm{vir}} \overline{\mathcal{M}}^{\ell,(m)} = \# \overline{\mathcal{M}}^{\ell,(m)}$. Again by \cref{prp:moduli spaces w tangency are tco}, the moduli space $\mathcal{M}^{\ell,(m)}$ is transversely cut out and%
    \begin{IEEEeqnarray*}{c}
        \dim \mathcal{M}^{\ell,(m)} = (n -3)(2 - 1) + \conleyzehnder(\gamma_1^m) - 2 \ell - 2 m + 4 = 0,
    \end{IEEEeqnarray*}
    where in the second equality we have used \cref{lem:unique reeb orbit with cz equal to}. This implies that $\mathcal{M}^{\ell,(m)}$ is compact, and in particular $\# \overline{\mathcal{M}}^{\ell,(m)} = \# \mathcal{M}^{\ell,(m)}$. By \cref{lem:moduli spaces of ellipsoids are all equal}, $\# \mathcal{M}^{\ell,(m)} = \# \mathcal{M}^{1,(m)}$. It remains to show that $\# \mathcal{M}^{1,(m)} = 1$. For this, notice that $\mathcal{M}^{1,(m)}$ is the set of equivalence classes of pairs $(j,u)$, where $j$ is an almost complex structure on $\Sigma = S^2$ and $u \colon (\dot{\Sigma}, j) \longrightarrow (\hat{E}_1, J_1)$ is a holomorphic map such that
    \begin{enumerate}
        \item $u(w) = x_1$ and $u$ has contact order $m$ to $D_1$ at $x_1$;
        \item if $(s,t)$ are the cylindrical coordinates on $\dot{\Sigma}$ near $z$ such that $v$ agrees with the direction $t = 0$, then
            \begin{IEEEeqnarray*}{rrCls+x*}
                \lim_{s \to +\infty} & \pi_{\R} \circ u(s,t)           & = & + \infty, \\
                \lim_{s \to +\infty} & \pi_{\partial E_1} \circ u(s,t) & = & \gamma_1 (a_1 m t).
            \end{IEEEeqnarray*}
    \end{enumerate}
    Here, two pairs $(j_0, u_0)$ and $(j_1, u_1)$ are equivalent if there exists a biholomorphism $\phi \colon (\Sigma, j_0) \longrightarrow (\Sigma, j_1)$ such that
    \begin{IEEEeqnarray*}{c+x*}
        u_0 = u_1 \circ \phi, \quad \phi(w) = w, \qquad \phi(z) = z, \qquad \dv \phi(z) v = v.
    \end{IEEEeqnarray*}
    We claim that any two pairs $(j_0, u_0)$ and $(j_1, u_1)$ are equivalent. By \cref{lem:u is a polynomial}, the maps $\varphi \circ u_0 \circ \psi_{j_0}$ and $\varphi \circ u_1 \circ \psi_{j_1}$ are polynomials of degree $m$:
    \begin{IEEEeqnarray*}{rCls+x*}
        \varphi \circ u_0 \circ \psi_{j_0} (z) & = & a_0 + \cdots + a_m z^m, \\
        \varphi \circ u_1 \circ \psi_{j_1} (z) & = & b_0 + \cdots + b_m z^m.
    \end{IEEEeqnarray*}
    Since $u_0$ and $u_1$ have contact order $m$ to $D_1$ at $x_1$, for every $\nu = 0,\ldots,m-1$ we have
    \begin{IEEEeqnarray*}{rCls+x*}
        0 & = & (\varphi \circ u_0 \circ \psi_{j_0})^{(\nu)}(0) = \nu! a_{\nu}, \\
        0 & = & (\varphi \circ u_1 \circ \psi_{j_1})^{(\nu)}(0) = \nu! b_{\nu}.
    \end{IEEEeqnarray*}
    Since $u_0$ and $u_1$ have the same asymptotic behaviour, $\operatorname{arg}(a_m) = \operatorname{arg}(b_m)$. Hence, there exists $\lambda \in \R_{>0}$ such that $\lambda^m b_m = a_m$. Then,
    \begin{IEEEeqnarray*}{c+x*}
        u_1 \circ \psi_{j_1} (\lambda z) = u_0 \circ \psi_{j_0} (z).
    \end{IEEEeqnarray*}
    Therefore, $(j_0, u_0)$ and $(j_1, u_1)$ are equivalent and $\# \mathcal{M}^{1,(m)} = 1$.
\end{proof}

\begin{remark}
    In \cite[Proposition 3.4]{cieliebakPuncturedHolomorphicCurves2018}, Cieliebak and Mohnke show that the signed count of the moduli space of holomorphic curves in $\C P^n$ in the homology class $[\C P^1]$ which satisfy a tangency condition $\p{<}{}{\mathcal{T}^{(n)}x}$ equals $(n-1)!$. It is unclear how this count relates to the one of \cref{lem:moduli spaces of ellipsoids have 1 element}.
\end{remark}

Finally, we will use the results of this subsection to compute the augmentation map of the ellipsoid $E_n$.

\begin{theorem}
    \label{thm:augmentation is nonzero}
    The augmentation map $\epsilon_m \colon CH_{n - 1 + 2m}(E_n) \longrightarrow \Q$ is an isomorphism.
\end{theorem}
\begin{proof}
    By \cref{lem:moduli spaces of ellipsoids have 1 element}, \cref{rmk:counts of moduli spaces with or without asy markers} and definition of the augmentation map, we have $\epsilon_m(\gamma^m_1) \neq 0$. By \cref{lem:lch of ellipsoid}, $\epsilon_m$ is an isomorphism.
\end{proof}

\subsection{Computations with contact homology}

Finally, we use the tools developed in this section to prove \cref{conj:the conjecture} (see \cref{thm:my main theorem}). The proof we give is the same as that of \cref{lem:computation of cl}, with the update that we will use the capacity $\mathfrak{g}_{k}$ to prove that
\begin{IEEEeqnarray*}{c+x*}
    \tilde{\mathfrak{g}}_k(X) \leq \mathfrak{g}_k(X) = \cgh{k}(X)
\end{IEEEeqnarray*}
for any nondegenerate Liouville domain $X$. Notice that in \cref{lem:computation of cl}, $\tilde{\mathfrak{g}}_k(X) \leq \cgh{k}(X)$ held because by assumption $X$ was a $4$-dimensional convex toric domain. We start by showing that $\tilde{\mathfrak{g}}_k(X) \leq \mathfrak{g}_k(X)$. This result has already been proven in \cite[Section 3.4]{mcduffSymplecticCapacitiesUnperturbed2022}, but we include a proof for the sake of completeness.

\begin{theorem}[{\cite[Section 3.4]{mcduffSymplecticCapacitiesUnperturbed2022}}]
    \phantomsection\label{thm:g tilde vs g hat}
    If $X$ is a Liouville domain then
    \begin{IEEEeqnarray*}{c+x*}
        \tilde{\mathfrak{g}}_k(X) \leq \mathfrak{g}_k(X).
    \end{IEEEeqnarray*}
\end{theorem}
\begin{proof}
    By \cref{lem:can prove ineqs for ndg}, we may assume that $X$ is nondegenerate. Choose a point $x \in \itr X$ and a symplectic divisor $D$ through $x$. Let $J \in \mathcal{J}(X,D)$ be an almost complex structure on $\hat{X}$ and consider the complex $CC(X)$, computed with respect to $J$. Suppose that $a > 0$ is such that the augmentation map
    \begin{IEEEeqnarray*}{c+x*}
        \epsilon_k \colon CH^a(X) \longrightarrow \Q
    \end{IEEEeqnarray*}
    is nonzero. By \cref{def:g tilde}, it is enough to show that there exists a Reeb orbit $\gamma$ such that $\mathcal{A}(\gamma) \leq a$ and $\mathcal{M}^{J}_{X}(\gamma)\p{<}{}{\mathcal{T}^{(k)}x} \neq \varnothing$. Choose a homology class $\beta \in CH^{a}(X)$ such that $\epsilon_k(\beta) \neq 0$, and write $\beta$ as a finite linear combination of Reeb orbits $\beta = \sum_{i=1}^{m} a_i \rho_i$, where every Reeb orbit has action $\mathcal{A}(\rho_i) \leq a$. One of the orbits in this linear combination, which we denote by $\rho$, is such that $\#^{\mathrm{vir}} \overline{\mathcal{M}}^{J}_{X}(\rho)\p{<}{}{\mathcal{T}^{(k)}x} \neq 0$. By assumption on the virtual perturbation scheme, $\overline{\mathcal{M}}^{J}_{X}(\rho)\p{<}{}{\mathcal{T}^{(k)}x}$ is nonempty. Choose $F = (F^1,\ldots,F^N) \in \overline{\mathcal{M}}^{J}_{X}(\rho)\p{<}{}{\mathcal{T}^{(k)}x}$ and denote by $C$ the component of $F$ which inherits the tangency constraint. Then, $C \in \mathcal{M}^J_X(\gamma)\p{<}{}{\mathcal{T}^{(k)}x}$ for some Reeb orbit $\gamma$ satisfying $\mathcal{A}(\gamma) \leq \mathcal{A}(\rho)$.
\end{proof}

\begin{theorem}
    \label{thm:g hat vs gh}
    If $X$ is a Liouville domain such that $\pi_1(X) = 0$ and $2 c_1(TX) = 0$ then%
    \begin{IEEEeqnarray*}{c+x*}
        \mathfrak{g}_k(X) = \cgh{k}(X).
    \end{IEEEeqnarray*}
\end{theorem}
\begin{proof}
    By \cref{lem:can prove ineqs for ndg}, we may assume that $X$ is nondegenerate. Let $E = E(a_1,\ldots,a_n)$ be an ellipsoid as in \cref{sec:augmentation map of an ellipsoid} such that there exists a strict exact symplectic embedding $\phi \colon E \longrightarrow X$. In \cite{bourgeoisEquivariantSymplecticHomology2016}, Bourgeois--Oancea define an isomorphism between linearized contact homology and positive $S^1$-equivariant contact homology, which we will denote by $\Phi_{\mathrm{BO}}$. This isomorphism commutes with the Viterbo transfer maps and respects the action filtration. In addition, the Viterbo transfer maps in linearized contact homology commute with the augmentation maps. Therefore, there is a commutative diagram
    \begin{IEEEeqnarray*}{c+x*}
        \begin{tikzcd}
            SH^{S^1,(\varepsilon,a]}_{n - 1 + 2k}(X) \ar[r, "\iota^{S^1,a}"] \ar[d, hook, two heads, swap, "\Phi_{\mathrm{BO}}^a"] & SH^{S^1,+}_{n - 1 + 2k}(X) \ar[r, "\phi_!^{S^1}"] \ar[d, hook, two heads, "\Phi_{\mathrm{BO}}"] & SH^{S^1,+}_{n - 1 + 2k}(E) \ar[d, hook, two heads, "\Phi_{\mathrm{BO}}"] \\
            CH^{a}_{n - 1 + 2k}(X) \ar[r, "\iota^{a}"] \ar[d, equals]                                & CH_{n - 1 + 2k}(X) \ar[r, "\phi_{!}"] \ar[d, equals]                  & CH_{n - 1 + 2k}(E) \ar[d, hook, two heads, "{\epsilon}^E_k"] \\
            CH^{a}_{n - 1 + 2k}(X) \ar[r, swap, "\iota^{a}"]                                         & CH_{n - 1 + 2k}(X) \ar[r, swap, "{\epsilon}_k^X"]                 & \Q
        \end{tikzcd}
    \end{IEEEeqnarray*}
    Here, the map ${\epsilon}_k^E$ is nonzero, or equivalently an isomorphism, by \cref{thm:augmentation is nonzero}. Then,%
    \begin{IEEEeqnarray*}{rCls+x*}
        \cgh{k}(X)
        & = & \inf \{ a > 0 \mid \phi_!^{S^1} \circ \iota^{S^1,a} \neq 0 \} & \quad [\text{by \cref{def:ck alternative}}] \\
        & = & \inf \{ a > 0 \mid {\epsilon}_k^X \circ \iota^{a} \neq 0 \}   & \quad [\text{since the diagram commutes}] \\
        & = & {\mathfrak{g}}_k(X)                                           & \quad [\text{by \cref{def:capacities glk}}].  & \qedhere
    \end{IEEEeqnarray*}
\end{proof}

\begin{theorem}
    \phantomsection\label{thm:my main theorem}
    Under \cref{assumption}, if $X_\Omega$ is a convex or concave toric domain then%
    \begin{IEEEeqnarray*}{c+x*}
        c_L(X_{\Omega}) = \delta_\Omega.
    \end{IEEEeqnarray*}
\end{theorem}
\begin{proof}
    Since $X_{\Omega}$ is concave or convex, we have $X_{\Omega} \subset N(\delta_\Omega)$. For every $k \in \Z_{\geq 1}$,
    \begin{IEEEeqnarray*}{rCls+x*}
        \delta_\Omega
        & \leq & c_P(X_{\Omega})                                & \quad [\text{by \cref{lem:c square geq delta}}] \\
        & \leq & c_L(X_{\Omega})                                & \quad [\text{by \cref{lem:c square leq c lag}}] \\
        & \leq & \frac{\tilde{\mathfrak{g}}_{k}(X_{\Omega})}{k} & \quad [\text{by \cref{thm:lagrangian vs g tilde}}] \\
        & \leq & \frac{{\mathfrak{g}}_{k}(X_{\Omega})}{k}       & \quad [\text{by \cref{thm:g tilde vs g hat}}] \\
        & =    & \frac{\cgh{k}(X_{\Omega})}{k}                  & \quad [\text{by \cref{thm:g hat vs gh}}] \\
        & \leq & \frac{\cgh{k}(N(\delta_\Omega))}{k}            & \quad [\text{since $X_{\Omega} \subset N(\delta_{\Omega})$}] \\
        & =    & \frac{\delta_\Omega(k+n-1)}{k}                 & \quad [\text{by \cref{lem:cgh of nondisjoint union of cylinders}}].
    \end{IEEEeqnarray*}
    The result follows by taking the infimum over $k$.
\end{proof}

% Bibliography
\AtEndDocument{
    \bibliographystyle{alpha}
    \bibliography{lagrangian}
}

\end{document}

%% file: preamble.tex
\usepackage[utf8]{inputenc}                  % 
\usepackage[T1]{fontenc}                     % 
\usepackage{lmodern}                         % 
\usepackage{geometry}                        % Customize page dimensions
\usepackage[square,numbers]{natbib}          % Bibliography
\usepackage[nottoc,notlot,notlof]{tocbibind} % Add bibliography to TOC
\usepackage{enumitem}                        % Lists
\usepackage{xparse}                          % Create commands and environments
\usepackage{xstring}                         % Program and manipulate strings (parenthesis macro)
\usepackage{etoolbox}                        % Delete warnings on underfull in bibliography
\usepackage{parskip}                         % Customize paragraphs
\usepackage{titling}

\usepackage{mathtools}              % Write formulas (includes package amsmath)
\usepackage{amssymb}                % Has mathematics symbols (includes package amsfonts)
\usepackage{amsthm}                 % Define theorem environments
\usepackage{IEEEtrantools}          % To use the equation environment IEEEeqnarray
\usepackage{tensor}                 % To write indices in formulas correctly
\usepackage{tikz}                   % To create drawings (includes graphicx)

\usepackage{hyperref}               % Include hyperlinks in the PDF
\usepackage{bookmark}               % Upgrades hyperref
\usepackage[capitalise]{cleveref}   % Smart referencing
\usepackage[all]{hypcap}            % Make hyperlinks to floats point to the top of the float instead of the caption

\apptocmd{\sloppy}{\hbadness 10000\relax}{}{}                       % Delete warnings on underfull in bibliography (related package: etoolbox)
\allowdisplaybreaks                                                 % Allow that an equation breaks from one page to the next
\graphicspath{{./figures/}}                                         % Define the folder where to place figures

\newlength{\alphabet}
\setlist[description]{font=\normalfont}
\setlist[enumerate]{font=\normalfont}
\setlist[enumerate,1]{label = {(\arabic*)}}
\setlist[enumerate,2]{label = {(\arabic{enumi}.\arabic*)}}

\newcounter{dummy}
\makeatletter
\newcommand\myitem[1][]{\item[#1]\refstepcounter{dummy}\def\@currentlabel{#1}}
\makeatother

\usetikzlibrary{decorations.pathreplacing} % Library for curly braces
\usetikzlibrary{math}                      % Library to write short mathematical programs inside TikZ
\usetikzlibrary{calc}                      % Library that makes TikZ able to perform computations
\usetikzlibrary{cd}                        % Library to draw commutative diagrams
\tikzset{
    symbol/.style={
        draw=none,
        every to/.append style={
            edge node={node [sloped, allow upside down, auto=false]{$#1$}}
        },
    },
}

\hypersetup{
    bookmarksnumbered=true,
    colorlinks=true,
    linkcolor=blue,
    citecolor=blue,
    urlcolor=blue
}

\theoremstyle{plain}      % Theorem-like environments - theorem
\newtheorem{theorem}     {Theorem} [section]
\newtheorem{proposition} [theorem] {Proposition}
\newtheorem{lemma}       [theorem] {Lemma}
\newtheorem{corollary}   [theorem] {Corollary}
\newtheorem{conjecture}  [theorem] {Conjecture}

\theoremstyle{definition} % Theorem-like environments - definition

\newtheorem{definition}  [theorem] {Definition}

\newtheorem{remark}      [theorem] {Remark}
\newtheorem{assumption}  [theorem] {Assumption}

%% file: macros.tex
\NewDocumentCommand {\cgh} {m} {c^{\mathrm{GH}}_{#1}}
\NewDocumentCommand {\csh} {m} {c^{S^1}_{#1}}

\NewDocumentCommand{\signature}{}{\operatorname{sign}}

\NewDocumentCommand{\Z}{}{\mathbb{Z}} %zz
\NewDocumentCommand{\Q}{}{\mathbb{Q}} %qq
\NewDocumentCommand{\R}{}{\mathbb{R}} %rr
\NewDocumentCommand{\C}{}{\mathbb{C}} %cc

\NewDocumentCommand{\id} {}{\operatorname{id}} %id
\NewDocumentCommand{\img}{}{\operatorname{im}} %img

\NewDocumentCommand{\Hom}   { }{\operatorname{Hom}}   %hm
\NewDocumentCommand{\itr} {}{\operatorname{int}}  %itr
\NewDocumentCommand  {\dv}     {}    {\mathrm{D}}                          %dv
\NewDocumentCommand  {\odv}    {m m} {\frac{\mathrm{d} #1}{\mathrm{d} #2}} %odv
\NewDocumentCommand  {\pdv}    {m m} {\frac{\partial #1}{\partial #2}}     %pdv
\NewDocumentCommand  {\edv}    {}    {\mathrm{d}}                          %edv
\NewDocumentCommand  {\del}    {}    {\partial}                            %dl
\NewDocumentCommand{\morse}        {}{\mu_{\operatorname{M}}}  %moi
\NewDocumentCommand{\conleyzehnder}{}{\mu_{\operatorname{CZ}}} %czi
\newcommand{\lpar}{(}
\newcommand{\rpar}{)}
\newcommand{\lsize}{}
\newcommand{\rsize}{}

\NewDocumentCommand{\SetParenthesisTypeSize}{m m}{
	\renewcommand{\lpar}{(}
	\renewcommand{\rpar}{)}
	\renewcommand{\lsize}{}
	\renewcommand{\rsize}{}
    \IfEq{#1}{(} { \renewcommand{\lpar}{(}       \renewcommand{\rpar}{)}       }{}
    \IfEq{#1}{()}{ \renewcommand{\lpar}{(}       \renewcommand{\rpar}{)}       }{}
	\IfEq{#1}{c} { \renewcommand{\lpar}{\{}      \renewcommand{\rpar}{\}}      }{}
	\IfEq{#1}{<} { \renewcommand{\lpar}{\langle} \renewcommand{\rpar}{\rangle} }{}
	\IfEq{#1}{[} { \renewcommand{\lpar}{[}       \renewcommand{\rpar}{]}       }{}
    \IfEq{#1}{[]}{ \renewcommand{\lpar}{[}       \renewcommand{\rpar}{]}       }{}
	\IfEq{#1}{|} { \renewcommand{\lpar}{\lvert}  \renewcommand{\rpar}{\rvert}  }{}
	\IfEq{#1}{||}{ \renewcommand{\lpar}{\lVert}  \renewcommand{\rpar}{\rVert}  }{}
    \IfEq{#1}{L} { \renewcommand{\lpar}{\lfloor} \renewcommand{\rpar}{\rfloor} }{}
    \IfEq{#1}{T} { \renewcommand{\lpar}{\lceil}  \renewcommand{\rpar}{\rceil}  }{}
    \IfEq{#2}{0}{ \renewcommand{\lsize}{}       \renewcommand{\rsize}{}       }{}
    \IfEq{#2}{1}{ \renewcommand{\lsize}{\bigl}  \renewcommand{\rsize}{\bigr}  }{}
    \IfEq{#2}{2}{ \renewcommand{\lsize}{\Bigl}  \renewcommand{\rsize}{\Bigr}  }{}
    \IfEq{#2}{3}{ \renewcommand{\lsize}{\biggl} \renewcommand{\rsize}{\biggr} }{}
    \IfEq{#2}{4}{ \renewcommand{\lsize}{\Biggl} \renewcommand{\rsize}{\Biggr} }{}
    \IfEq{#2}{a}{ \renewcommand{\lsize}{\left}  \renewcommand{\rsize}{\right} }{}
}

\NewDocumentCommand{\p}{m m m}{
    \IfEq{#1}{n}{}{\SetParenthesisTypeSize{#1}{#2} \lsize \lpar}
    #3
    \IfEq{#1}{n}{}{\SetParenthesisTypeSize{#1}{#2} \rsize \rpar}
}

\NewDocumentCommand{\sbn}{o m m}{
    \IfValueF{#1}{
        \{ #2 \ | \ #3 \}
    }{}
    \IfValueT{#1}{
        \IfEq{#1}{0}{        \{ #2 \       | \ #3        \} }{}
        \IfEq{#1}{1}{ \bigl  \{ #2 \ \big  | \ #3 \bigr  \} }{}
        \IfEq{#1}{2}{ \Bigl  \{ #2 \ \Big  | \ #3 \Bigr  \} }{}
        \IfEq{#1}{3}{ \biggl \{ #2 \ \bigg | \ #3 \biggr \} }{}
        \IfEq{#1}{4}{ \Biggl \{ #2 \ \Bigg | \ #3 \Biggr \} }{}
    }{}
}

\NewDocumentEnvironment{copiedtheorem}
    {o m}
    {
        \theoremstyle{plain}
        \newtheorem*{copytheorem:#2}{\cref{#2}}
        \IfNoValueTF{#1}
        {
            \begin{copytheorem:#2}
        }
        {
            \begin{copytheorem:#2}[{#1}]
        }
    }
    {
        \end{copytheorem:#2}
    }

\NewDocumentEnvironment{secondcopy}
    {o m}
    {
        \IfNoValueTF{#1}
        {
            \begin{copytheorem:#2}
        }
        {
            \begin{copytheorem:#2}[{#1}]
        }
    }
    {
        \end{copytheorem:#2}
    }

%% file: lagrangian.bbl
\newcommand{\etalchar}[1]{$^{#1}$}
\begin{thebibliography}{FOOO10b}

\bibitem[BEH{\etalchar{+}}03]{bourgeoisCompactnessResultsSymplectic2003}
Fr{\'e}d{\'e}ric Bourgeois, Yakov Eliashberg, Helmut Hofer, Krzysztof Wysocki,
  and Eduard Zehnder.
\newblock \href{https://msp.org/gt/2003/7-2/p09.xhtml}{Compactness results in
  symplectic field theory}.
\newblock {\em Geometry \& Topology}, 7(2):799--888, December 2003.

\bibitem[BH18]{baoDefinitionCylindricalContact2018}
Erkao Bao and Ko~Honda.
\newblock
  \href{https://onlinelibrary.wiley.com/doi/abs/10.1112/topo.12077}{Definition
  of cylindrical contact homology in dimension three}.
\newblock {\em Journal of Topology}, 11(4):1002--1053, 2018.

\bibitem[BH21]{baoSemiglobalKuranishiCharts2021}
Erkao Bao and Ko~Honda.
\newblock \href{http://arxiv.org/abs/1512.00580}{Semi-global Kuranishi charts
  and the definition of contact homology}.
\newblock {\em arXiv:1512.00580 [math]}, September 2021.

\bibitem[BM04]{bourgeoisCoherentOrientationsSymplectic2004}
Fr{\'e}d{\'e}ric Bourgeois and Klaus Mohnke.
\newblock \href{https://doi.org/10.1007/s00209-004-0656-x}{Coherent
  orientations in symplectic field theory}.
\newblock {\em Mathematische Zeitschrift}, 248(1):123--146, September 2004.

\bibitem[BO10]{bourgeoisFredholmTheoryTransversality2010}
Fr{\'e}d{\'e}ric Bourgeois and Alexandru Oancea.
\newblock
  \href{https://www.ems-ph.org/journals/show_abstract.php?issn=1435-9855&vol=12&iss=5&rank=5}{Fredholm
  theory and transversality for the parametrized and for the
  ${{S^1}}$-invariant symplectic action}.
\newblock {\em Journal of the European Mathematical Society}, 12(5):1181--1229,
  August 2010.

\bibitem[BO13]{bourgeoisGysinExactSequence2013}
Fr{\'e}d{\'e}ric Bourgeois and Alexandru Oancea.
\newblock
  \href{https://www.worldscientific.com/doi/abs/10.1142/S1793525313500210}{The
  Gysin exact sequence for ${{S^1}}$-equivariant symplectic homology}.
\newblock {\em Journal of Topology and Analysis}, 05(04):361--407, December
  2013.

\bibitem[BO16]{bourgeoisEquivariantSymplecticHomology2016}
Fr{\'e}d{\'e}ric Bourgeois and Alexandru Oancea.
\newblock
  \href{https://academic.oup.com/imrn/article-lookup/doi/10.1093/imrn/rnw029}{${{S^1}}$-equivariant
  symplectic homology and linearized contact homology}.
\newblock {\em International Mathematics Research Notices}, page rnw029, June
  2016.

\bibitem[CM05]{cieliebakCompactnessPuncturedHolomorphic2005}
Kai Cieliebak and Klaus Mohnke.
\newblock \href{https://projecteuclid.org/euclid.jsg/1154467631}{Compactness
  for punctured holomorphic curves}.
\newblock {\em Journal of Symplectic Geometry}, 3(4):589--654, December 2005.

\bibitem[CM07]{cieliebakSymplecticHypersurfacesTransversality2007}
Kai Cieliebak and Klaus Mohnke.
\newblock
  \href{http://www.intlpress.com/site/pub/pages/journals/items/jsg/content/vols/0005/0003/a002/}{Symplectic
  hypersurfaces and transversality in Gromov\textendash Witten theory}.
\newblock {\em Journal of Symplectic Geometry}, 5(3):281--356, 2007.

\bibitem[CM18]{cieliebakPuncturedHolomorphicCurves2018}
Kai Cieliebak and Klaus Mohnke.
\newblock \href{http://link.springer.com/10.1007/s00222-017-0767-8}{Punctured
  holomorphic curves and Lagrangian embeddings}.
\newblock {\em Inventiones mathematicae}, 212(1):213--295, April 2018.

\bibitem[dB16]{desaint-gervaisUniformizationRiemannSurfaces2016}
Henri~Paul {de Saint-Gervais} and Robert Burns.
\newblock {\em \href{http://www.ems-ph.org/doi/10.4171/145}{Uniformization of
  Riemann surfaces: revisiting a hundred-year-old theorem}}.
\newblock European Mathematical Society Publishing House, Z\"urich,
  Switzerland, January 2016.

\bibitem[EGH10]{eliashbergIntroductionSymplecticField2010}
Y.~Eliashberg, A.~Glvental, and H.~Hofer.
\newblock \href{https://doi.org/10.1007/978-3-0346-0425-3_4}{Introduction to
  {{Symplectic Field Theory}}}.
\newblock In N.~Alon, J.~Bourgain, A.~Connes, M.~Gromov, and V.~Milman,
  editors, {\em Visions in {{Mathematics}}: {{GAFA}} 2000 {{Special}} Volume,
  {{Part II}}}, Modern {{Birkh\"auser Classics}}, pages 560--673.
  {Birkh\"auser}, {Basel}, 2010.

\bibitem[FHW94]{floerApplicationsSymplecticHomology1994}
Andreas Floer, Helmut Hofer, and Krzysztof Wysocki.
\newblock \href{http://link.springer.com/10.1007/BF02571962}{Applications of
  symplectic homology I}.
\newblock {\em Mathematische Zeitschrift}, 217(1):577--606, September 1994.

\bibitem[FOOO10a]{fukayaLagrangianIntersectionFloer2010}
Kenji Fukaya, Yong-Geun Oh, Hiroshi Ohta, and Kaoru Ono.
\newblock {\em \href{http://www.ams.org/amsip/046.1}{Lagrangian intersection
  {{Floer}} theory: anomaly and obstruction, part I}}, volume 46.1 of {\em
  {{AMS}}/{{IP Studies}} in {{Advanced Mathematics}}}.
\newblock {American Mathematical Society}, June 2010.

\bibitem[FOOO10b]{fukayaLagrangianIntersectionFloer2010a}
Kenji Fukaya, Yong-Geun Oh, Hiroshi Ohta, and Kaoru Ono.
\newblock {\em \href{http://www.ams.org/amsip/046.2}{Lagrangian intersection
  {{Floer}} theory: anomaly and obstruction, part II}}, volume 46.2 of {\em
  {{AMS}}/{{IP Studies}} in {{Advanced Mathematics}}}.
\newblock {American Mathematical Society}, June 2010.

\bibitem[GH18]{guttSymplecticCapacitiesPositive2018}
Jean Gutt and Michael Hutchings.
\newblock \href{https://msp.org/agt/2018/18-6/p11.xhtml}{Symplectic capacities
  from positive ${{S^1}}$-equivariant symplectic homology}.
\newblock {\em Algebraic \& Geometric Topology}, 18(6):3537--3600, October
  2018.

\bibitem[GPR22]{guttCubeNormalizedSymplectic}
Jean Gutt, Miguel Pereira, and Vinicius G.~B. Ramos.
\newblock \href{http://arxiv.org/abs/2208.13666}{Cube normalized symplectic
  capacities}.
\newblock {\em arXiv:2208.13666 [math.SG]}, August 2022.
\newblock arXiv:2208.13666 [math].

\bibitem[GS18]{ginzburgFilteredSymplecticHomology2018}
Viktor~L. Ginzburg and Jeongmin Shon.
\newblock
  \href{https://www.worldscientific.com/doi/abs/10.1142/S0129167X18500714}{On
  the filtered symplectic homology of prequantization bundles}.
\newblock {\em International Journal of Mathematics}, 29(11):1850071, October
  2018.

\bibitem[Gut12]{guttConleyZehnderIndex2012}
Jean Gutt.
\newblock \href{http://arxiv.org/abs/1201.3728}{The Conley\textendash Zehnder
  index for a path of symplectic matrices}.
\newblock {\em arXiv:1201.3728 [math]}, January 2012.

\bibitem[Gut14]{guttMinimalNumberPeriodic2014}
Jean Gutt.
\newblock {\em \href{https://tel.archives-ouvertes.fr/tel-01016954v2}{On the
  minimal number of periodic Reeb orbits on a contact manifold}}.
\newblock PhD thesis, Universit\'e de Strasbourg and Universit\'e libre de
  Bruxelles, 2014.

\bibitem[Gut17]{guttPositiveEquivariantSymplectic2017}
Jean Gutt.
\newblock
  \href{https://www.intlpress.com/site/pub/pages/journals/items/jsg/content/vols/0015/0004/a003/}{The
  positive equivariant symplectic homology as an invariant for some contact
  manifolds}.
\newblock {\em Journal of Symplectic Geometry}, 15(4):1019--1069, December
  2017.

\bibitem[HN16]{hutchingsCylindricalContactHomology2016}
Michael Hutchings and Jo~Nelson.
\newblock
  \href{https://www.intlpress.com/site/pub/pages/journals/items/jsg/content/vols/0014/0004/a001/}{Cylindrical
  contact homology for dynamically convex contact forms in three dimensions}.
\newblock {\em Journal of Symplectic Geometry}, 14(4):983--1012, December 2016.

\bibitem[Hof93]{hoferPseudoholomorphicCurvesSymplectizations1993}
H.~Hofer.
\newblock \href{https://doi.org/10.1007/BF01232679}{Pseudoholomorphic curves in
  symplectizations with applications to the {{Weinstein}} conjecture in
  dimension three}.
\newblock {\em Inventiones mathematicae}, 114(1):515--563, December 1993.

\bibitem[HWZ21]{hoferPolyfoldFredholmTheory2021}
Helmut Hofer, Krzysztof Wysocki, and Eduard Zehnder.
\newblock {\em
  \href{https://link.springer.com/book/10.1007/978-3-030-78007-4}{Polyfold and
  Fredholm theory}}.
\newblock Number 3. Folge, Volume 72 in Ergebnisse der Mathematik und ihrer
  Grenzgebiete. Springer, Cham, Switzerland, 2021.

\bibitem[Iri21]{irieSymplecticHomologyFiberwise2021}
Kei Irie.
\newblock \href{http://arxiv.org/abs/1907.09749}{Symplectic homology of
  fiberwise convex sets and homology of loop spaces}.
\newblock {\em arXiv:1907.09749 [math]}, June 2021.

\bibitem[Ish18]{ishikawaConstructionGeneralSymplectic2018}
Suguru Ishikawa.
\newblock \href{http://arxiv.org/abs/1807.09455}{Construction of general
  symplectic field theory}.
\newblock {\em arXiv:1807.09455 [math]}, August 2018.

\bibitem[MS12]{mcduffHolomorphicCurvesSymplectic2012}
Dusa McDuff and Dietmar Salamon.
\newblock {\em
  \href{https://doi.org/http://dx.doi.org/10.1090/coll/052}{$J$-holomorphic
  curves and symplectic topology}}, volume~52 of {\em Colloquium Publications}.
\newblock American Mathematical Society, Providence, RI, 2012.

\bibitem[MS22]{mcduffSymplecticCapacitiesUnperturbed2022}
Dusa McDuff and Kyler Siegel.
\newblock \href{http://arxiv.org/abs/2111.00515}{Symplectic capacities,
  unperturbed curves, and convex toric domains}.
\newblock {\em arXiv:2111.00515 [math]}, February 2022.

\bibitem[Nel15]{nelsonAutomaticTransversalityContact2015}
Jo~Nelson.
\newblock \href{https://doi.org/10.1007/s12188-015-0112-3}{Automatic
  transversality in contact homology I: regularity}.
\newblock {\em Abhandlungen aus dem Mathematischen Seminar der Universit\"at
  Hamburg}, 85(2):125--179, October 2015.

\bibitem[Par16]{pardonAlgebraicApproachVirtual2016}
John Pardon.
\newblock \href{https://msp.org/gt/2016/20-2/p04.xhtml}{An algebraic approach
  to virtual fundamental cycles on moduli spaces of pseudo-holomorphic curves}.
\newblock {\em Geometry \& Topology}, 20(2):779--1034, April 2016.

\bibitem[Par19]{pardonContactHomologyVirtual2019}
John Pardon.
\newblock
  \href{https://www.ams.org/jams/2019-32-03/S0894-0347-2019-00924-0/}{Contact
  homology and virtual fundamental cycles}.
\newblock {\em Journal of the American Mathematical Society}, 32(3):825--919,
  July 2019.

\bibitem[Per22]{Per}
Miguel Pereira.
\newblock {\em \href{https://arxiv.org/abs/2205.13381v1}{Equivariant symplectic
  homology, linearized contact homology and the {{Lagrangian}} capacity}}.
\newblock PhD thesis, University of Augsburg, May 2022.

\bibitem[Sch00]{schwarzActionSpectrumClosed2000}
Matthias Schwarz.
\newblock \href{http://msp.berkeley.edu/pjm/2000/193-2/p10.xhtml}{On the action
  spectrum for closed symplectically aspherical manifolds}.
\newblock {\em Pacific Journal of Mathematics}, 193(2):419--461, April 2000.

\bibitem[Sei08]{seidelBiasedViewSymplectic2008}
Paul Seidel.
\newblock
  \href{https://projecteuclid.org/ebooks/current-developments-in-mathematics/Current-Developments-in-Mathematics-2006/chapter/A-biased-view-of-symplectic-cohomology/cdm/1223654543}{A
  biased view of symplectic cohomology}.
\newblock {\em Current Developments in Mathematics, 2006}, 2006:211--254,
  February 2008.

\bibitem[Sie20]{siegelHigherSymplecticCapacities2020}
Kyler Siegel.
\newblock \href{http://arxiv.org/abs/1902.01490}{Higher symplectic capacities}.
\newblock {\em arXiv:1902.01490 [math-ph]}, February 2020.

\bibitem[Vit92]{viterboSymplecticTopologyGeometry1992}
Claude Viterbo.
\newblock \href{https://doi.org/10.1007/BF01444643}{Symplectic topology as the
  geometry of generating functions}.
\newblock {\em Mathematische Annalen}, 292(1):685--710, March 1992.

\bibitem[Vit99]{viterboFunctorsComputationsFloer1999}
Claude Viterbo.
\newblock \href{http://link.springer.com/10.1007/s000390050106}{Functors and
  computations in Floer homology with applications, I}.
\newblock {\em Geometric And Functional Analysis}, 9(5):985--1033, December
  1999.

\bibitem[Wen10]{wendlAutomaticTransversalityOrbifolds2010}
Chris Wendl.
\newblock \href{http://www.ems-ph.org/doi/10.4171/CMH/199}{Automatic
  transversality and orbifolds of punctured holomorphic curves in dimension
  four}.
\newblock {\em Commentarii Mathematici Helvetici}, pages 347--407, 2010.

\bibitem[Wen16]{wendlLecturesSymplecticField2016}
Chris Wendl.
\newblock \href{http://arxiv.org/abs/1612.01009}{Lectures on symplectic field
  theory}.
\newblock {\em arXiv:1612.01009 [math]}, December 2016.

\end{thebibliography}
